\documentclass[11pt]{amsart}

\usepackage{amsmath}
\usepackage{amssymb}
\usepackage{amsthm}
\usepackage{bbm}
\usepackage{esint} %integrals
\usepackage[colorlinks,citecolor=red,pagebackref,hypertexnames=false, hyperindex,breaklinks]{hyperref}

\usepackage{esint} %integrals
\usepackage{enumitem} %list with alphabet
\usepackage{verbatim} %for \begin{comment}\end{comment}

\usepackage{tikz}
\newcommand{\R}{{\mathbb R}}       % Field of real numbers
\newcommand{\N}{{\mathbb N}}

\newcommand{\bH}{{\mathbb H}}
\newcommand{\bS}{{\mathbb S}}
\newcommand{\bM}{{\mathbb M}}

\newcommand{\HH}{{\mathcal H}}

\newcommand{\EE}{{\mathcal E}}

\newcommand{\diam}{{\rm diam}}
\newcommand{\dist}{{\rm dist}}

\newcommand{\rf}[1]{{(\ref{#1})}}

\newcommand{\supp}{\operatorname{supp}}

\newcommand{\vphi}{{\varphi}}
\newcommand{\ve}{{\varepsilon}}
\newcommand{\vv}{{\vspace{2mm}}}
\newcommand{\vvv}{{\vspace{3mm}}}
\newcommand{\wt}[1]{{\widetilde{#1}}}
\newcommand{\wh}[1]{{\widehat{#1}}}

\newcommand{\rad}{{\operatorname{rad}}}

\newcommand{\poms}{{\partial_{\bS^n}\Omega}}
\newcommand{\pom}{{\partial \Omega}}

\newcommand{\capp}{\operatorname{Cap}}

\def\Xint#1{\mathchoice
	{\XXint\displaystyle\textstyle{#1}}%
	{\XXint\textstyle\scriptstyle{#1}}%
	{\XXint\scriptstyle\scriptscriptstyle{#1}}%
	{\XXint\scriptscriptstyle\scriptscriptstyle{#1}}%
	\!\int}
\def\XXint#1#2#3{{\setbox0=\hbox{$#1{#2#3}{\int}$ }
		\vcenter{\hbox{$#2#3$ }}\kern-.58\wd0}}

\def\avint{\;\Xint-}

\newtheorem{theorem}{Theorem}[section]
\newtheorem{lemma}[theorem]{Lemma}
\newtheorem{mlemma}[theorem]{Main Lemma}
\newtheorem{coro}[theorem]{Corollary}

\newtheorem{claim}{Claim}

\newtheorem*{theorem*}{Theorem}

\newtheorem{theorema}{Theorem}[section]

\theoremstyle{definition}

\theoremstyle{remark}
\newtheorem{rem}[theorem]{\bf Remark}

\numberwithin{equation}{section}

\newcommand{\brem}{\begin{rem}}
\newcommand{\erem}{\end{rem}}

\usepackage{xcolor}

\begin{document}
	
	\title[ Faber-Krahn inequalities, the ACF formula, and Carleson's conjecture]
	{Faber-Krahn inequalities, the Alt-Caffarelli-Friedman formula, and Carleson's $\ve^2$ conjecture in higher dimensions}

\begin{abstract}
The main aim of this article is to prove quantitative spectral inequalities for the Laplacian with Dirichlet boundary conditions. More specifically, we prove sharp quantitative stability for the Faber-Krahn inequality in terms of Newtonian capacities and Hausdorff contents of positive codimension, thus providing an answer to a question posed by De Philippis and Brasco.
	
One of our results asserts that for any bounded domain $\Omega\subset\R^n$, $n\geq3$, with Lebesgue measure equal to that of the
unit ball $B_0$ and
whose first eigenvalue is $\lambda_\Omega$, 
denoting by $\lambda_{B_0}$ the first eigenvalue for $B_0$,
for any $a\in (0,1)$ it holds
$$\lambda_\Omega - \lambda_{B_0} \geq C(a) \,\inf_B
\bigg(\sup_{t\in (0,1)} \avint_{\partial ((1-t) B)} \frac{\capp_{n-2}(B(x,atr_B)\setminus \Omega)}{(t\,r_B)^{n-3}}\,d\HH^{n-1}(x)\bigg)^2,$$
where the infimum is taken over all balls $B$ with the same Lebesgue measure as $\Omega$ and $\capp_{n-2}$ is the Newtonian capacity of homogeneity $n-2$. In fact, this holds for bounded subdomains of the sphere and the hyperbolic space, as well. 

In a second result, we also apply the new Faber-Krahn type inequalities to quantify the Hayman-Friedland inequality about the characteristics of disjoint domains in the unit sphere.
Thirdly, we propose a natural extension of  Carleson's $\ve^2$-conjecture to higher dimensions
in terms of a square function involving the characteristics of certain spherical domains, and we prove
the necessity of the finiteness of such square function in the tangent points via the Alt-Caffarelli-Friedman monotonicity formula. 
Finally, we answer in the negative a question posed by Allen, Kriventsov and Neumayer in connection to rectifiability and the positivity set of the ACF monotonicity formula.
\end{abstract}

\author{Ian Fleschler}
\address{Department of Mathematics, Fine Hall, Princeton University, Washington Road, Princeton, NJ
08540, USA}
\email{imf@princeton.edu}

\author{Xavier Tolsa}

\address{ICREA (Barcelona, Catalonia)\\
 Universitat Aut\`onoma de Barcelona \\
and Centre de Recerca Matem\`atica (Barcelona, Catalonia).}
\email{xavier.tolsa@uab.cat}

\author{Michele Villa}

\address{University of Oulu (Oulu, Finland)
and Universitat Aut\`onoma de Barcelona (Barcelona, Catalonia).}

\email{michele.villa@oulu.fi}

\thanks{
I.F. acknowledges the support of the National Science Foundation through the grant FRG-1854147. 
X.T and M.V. are supported by the European Research Council (ERC) under the European Union's Horizon 2020 research and innovation programme (grant agreement 101018680). X.T. is also partially supported by MICINN (Spain) under the grant PID2020-114167GB-I00, the María de Maeztu Program for units of excellence (Spain) (CEX2020-001084-M), and 2021-SGR-00071 (Catalonia).  
M.V. is supported by the Academy of Finland via the project ``Higher dimensional Analyst's Traveling Salesman theorems and Dorronsoro estimates on non-smooth sets", grant No. 347828/24304228.
This work is based upon research funded by the  Deutsche Forschungsgemeinschaft (DFG, German Research Foundation) under Germany's Excellence Strategy – EXC-2047/1 – 390685813, while the authors were in residence at the Hausdorff Institute of Mathematics in Spring 2022 during the program ``Interactions between geometric measure theory, singular integrals, and PDEs''.
}

\maketitle
	
	\tableofcontents

\section{Introduction}

The main aim of this article is to prove quantitative spectral inequalities for the Laplacian with Dirichlet boundary conditions. More specifically, we show that the Faber-Krahn inequality is stable, where stability is quantified in terms of Newtonian capacities and Hausdorff contents of positive codimension. Our results hold for subdomains of the Euclidean space, the unit sphere and the hyperbolic space, and they are sharp up to a constant factor. We remark that Theorem \ref{teomain0}, our main result (see below), can be seen as a \textit{geometric} solution of an issue raised by Brasco and De Philippis (\cite[Open Problem 7.23]{BD}), and it foundamentally rests on the \textit{analytic} solution of the same problem by Allen, Kriventsov and Neumayer \cite{AKN1}.
%We are not aware of previous quantitative Faber-Krahn inequalities involving Hausdorff measures or contents of positive codimension.
Our interest in these versions of the Faber-Krahn inequality arise also from applications in connection with the Alt-Caffarelli-Friedman monotonicity formula and with the Carleson $\ve^2$-conjecture 
about tangent points for Jordan domains
in the plane. Indeed, in this paper we also propose a natural extension of  Carleson $\ve^2$-conjecture to higher dimensions
in terms of a square function involving the characteristics of certain spherical domains, and we prove
the necessity of the finiteness of such square function at almost every tangent point. Sufficiency is shown in the companion paper \cite{FTV}.

\subsection{Quantitative Faber-Krahn inequalities}

Given a bounded open set $\Omega\subset \R^n$, we say that $u\in W^{1,2}_0(\Omega)$ is a Dirichlet eigenfunction of
$\Omega$ (for the Laplacian) if $u\not\equiv 0$ and
$$-\Delta u = \lambda\,u,$$
for some $\lambda\in \R\setminus \{0\}$. The number $\lambda$ is the eigenvalue associated with $u$. It is well known that all the eigenvalues of the Laplacian are positive and the smallest one, i.e., the first eigenvalue $\lambda_1$, satisfies
$$\lambda_1 = \inf_{u\in  W^{1,2}_0(\Omega)} \frac{\int_\Omega|\nabla u|^2\,dx}{\int_\Omega|u|^2\,dx}.$$
Further the infimum is attained by an eigenfunction $u_0$ which does not change sign, and so which can be assumed to be non-negative. We will denote by $\lambda_\Omega$ (or $\lambda(\Omega)$) the first eigenvalue of $\Omega$ and by $u_\Omega$ the associated non-negative eigenfunction, normalized so that $\|u_\Omega\|_{L^2(\Omega)}=1$.

The classical Faber-Krahn inequality asserts that among all bounded open sets with a fixed volume, a ball minimizes the first eigenvalue.
Following many previous works (see, for example, \cite{HN} and \cite{Melas}), Brasco, De Philippis, and Velichkov proved in \cite{BDV} the following
sharp quantitative version of the Faber-Krahn inequality.

\begin{theorem}\label{teoBDV}
For $n\geq2$, let $\Omega\subset \R^n$ be a bounded open set with $\HH^n(\Omega)=1$ and let $B_0\subset \R^n$ be a ball with $\HH^n(B_0)=1$.
Then
\begin{equation}\label{eqBDV}
\lambda_\Omega - \lambda_{B_0} \geq c\,\inf_B\HH^n(\Omega \triangle B)^2,
\end{equation}
where $c$ is a positive absolute constant and the infimum is taken over all balls $B$ with $\HH^n(B)=1$.
\end{theorem}

The inequality above is sharp in the sense that the power $2$ on the right hand side cannot be lowered.

The classical Faber-Krahn inequality also holds for subdomains of the sphere $\bS^n$ or the hyperbolic space $\bH^n$. In this context one should take the Laplacian with respect to the Riemannian metric of the space (i.e., the Laplace-Beltrami operator) and one should consider geodesic balls. That is, for subdomains $\Omega$ of $\bS^n$ or $\bH^n$ with a given volume, the minimal
value of $\lambda_\Omega$ is attained again by a geodesic ball among all open bounded domains of the 
same volume. Further, recently Allen, Kriventsov, and Neumayer \cite{AKN1} have obtained the following quantitative form.

\begin{theorem}\label{teoAKN}
For $n\geq2$, let $\bM^n$ be either $\R^n$, $\bS^n$, or $\bH^n$, and let $\beta>0$.
Let $\Omega$ be a relatively open subset of $\bM^n$ and let $B_0$ be a geodesic ball in $\bM^n$ such that $\HH^n(B_0)=\HH^n(\Omega)$.
In the case $\bM^n=\bS^n$, suppose also that $\beta\leq \HH^n(\Omega)\leq \HH^n(\bS^n)-\beta$, and in the case
  $\bM^n=\R^n$ or $\bM^n=\bH^n$ only assume that $\HH^n(\Omega)\leq \beta$.
 Denote by $\lambda_\Omega$ and $\lambda_{B_0}$ the respective first Dirichlet eigenvalues of $-\Delta_{\bM^n}$ in $\Omega$ and $B_0$. Then
\begin{equation}\label{eqAKN}
\lambda_\Omega - \lambda_{B_0} \geq c(\beta)\,\inf_{B} \Big(\HH^n(\Omega \triangle B)^2 + \int_{\bM^n} |u_\Omega - u_{B}|^2\,d\HH^n\Big),
\end{equation}
where $c(\beta)>0$, the infimum is taken over all geodesic balls $B\subset\bM^n$  
such that $\HH^n(B)=\HH^n(\Omega)$, and $u_\Omega$ and $u_B$ are the corresponding first eigenfunctions in $\Omega$ and $B$ normalized so that
they are positive and $\|u_\Omega\|_{L^2(\bM^n)} = \|u_B\|_{L^2(\bM^n)}=1$.
 In the case $\bM^n=
\bS^n$, \rf{eqAKN} also holds  with the infimum over $\bS^n$ replaced 
by the choice of $B$ equal to a ball centered at the $\bS^n$-barycenter of $\Omega$ with $\HH^n(B)=\HH^n(\Omega)$ 
 (possibly with a different constant $c(\beta)$).
\end{theorem}

In the theorem we have denoted by $\Delta_{\bM^n}$ the Laplace-Beltrami operator on $\bM^n$.
For the definition of $\bS^n$-barycenter of a ball in $\bS^n$, see Section \ref{secnot}. The assumption involving the parameter $\beta$ is necessary to prevent the domain from being too big and, in the case $\bM^n=\bS^n$, also too small.
Remark that in the case $\bM^n=\R^n$ one can get an appropriate scaling invariant statement by renormalising. This is not case for $\bS^n$ or $\bH^n$.
 Notice the presence of the additional term $\int_{\bM^n} |u_\Omega - u_B|^2\,d\HH^n$ in the inequality \rf{eqAKN} when compared to \rf{eqBDV}. In this term we assume that the functions $u_\Omega$ and $u_B$ vanish outside of $\Omega$ and $B$ respectively.

\vv
To state our results we need to introduce some additional notation and terminology. For  $\bM^n=\R^{n}$ or $\bM^n=\bS^{n}$.
We denote by $\dist_{\bM^n}$ the geodesic distance in $\bM^n$ and by $\partial_{\bM^n}A$ the boundary
of any set $A\subset \bM^n$. We also write  
$B_{\bM^n}(x,r)$ to denote an open ball in $\bM^n$ centered in $x$, with radius $r$. 
Given $B=B_{\bM^n}(x,r)$ and $\rho>0$, we set $\rho B=B_{\bM^n}(x,\rho r)$ and we denote $\delta_B(x) =\dist_{\bM^n}(x,\partial_{\bM^n}B)$,

%(x,t)\approx|x-y|$ for all $x,y\in\bS^n$.
%$$\delta_B(x) = \dist_{\bS^n}(x,\partial_{\bS^n}B),$$
%where $\dist_{\bS^n}$ is the geodesic distance in $\bS^n$  and $\partial_{\bS^n}B$ stands for the
%boundary of $B$ relative to $\bS^n$, i.e.\ the geodesic circumference bounding $B$ in $\bS^n$. Notice that $\dist_{\bS^n}(x,t)\approx|x-y|$ for all $x,y\in\bS^n$.

We denote by $\capp_{n-2}$ the Newtonian capacity of homogeneity $n-2$, and by $\capp_L$ the logarithmic capacity
(see Section \ref{seccap} for the precise definitions).

Our first main result in this paper is the following.

\begin{theorema}\label{teomain0}
Given $n\geq2$, let $\bM^n$ be either $\R^n$, $\bS^n$ or $\bH^n$, and let $\beta>0$, $a\in (0,1)$.
Let $\Omega$ be a relatively open subset of $\bM^n$ and let $B_0$ be a geodesic ball in $\bM^n$ such that $\HH^n(B_0)=\HH^n(\Omega)$.
In the case $\bM^n=\bS^n$, suppose also that $\beta\leq \HH^n(\Omega)\leq \HH^n(\bS^n)-\beta$, while in the case
$\bM^n=\R^n$ or $\bM^n=\bH^n$ only assume that $\HH^n(\Omega)\leq \beta$.  Denote by $\lambda_\Omega$ and $\lambda_{B_0}$ the respective first Dirichlet eigenvalues of $-\Delta_{\bM^n}$ in $\Omega$ and $B_0$.
Then, there is some constant $C(a,\beta)>0$ such that in the case $n\geq 3$ we have
%\begin{equation}\label{eqmain00}
%\lambda_\Omega - \lambda_B \geq C(a,\beta)\,
%\bigg(\sup_{h\in (0,\rad_{\bM^n}(B))} \int_{S_h} h^2\capp(B_{\bM^n}(x,a\delta_B(x))\setminus \Omega)\,d\HH^{n-1}(x)\bigg)^2.
%\end{equation}
\begin{equation}\label{eqmain00}
\lambda_\Omega - \lambda_{B_0} \geq C(a,\beta)\,\inf_{B}
\bigg(\sup_{t\in (0,1)} \avint_{\partial_{\bM^n} ((1-t) B)} \frac{\capp_{n-2}(B_{\bM^n}(x,atr_B)\setminus \Omega)}{(t\,r_B)^{n-3}}\,d\HH^{n-1}(x)\bigg)^2,
\end{equation}
where the infimum is taken over all geodesic balls $B\subset\bM^n$  
such that $\HH^n(B)=\HH^n(\Omega)$ and  $r_B$ is the radius of $B$.
In the case $n=2$ we have
\begin{equation}\label{eqmain01}
\lambda_\Omega - \lambda_{B_0} \geq C(a,\beta)\,\inf_{B}
\bigg(\sup_{t\in (0,1)} \avint_{\partial_{\bM^2} ((1-t) B)} \frac{t\,r_B}{
\log\frac{2ta r_B}{\capp_L(B_{\bM^2}(x,at r_B)\setminus \Omega)}}\,d\HH^{1}(x)\bigg)^2.
\end{equation}
 In the case $\bM^n=
\bS^n$, \rf{eqmain00} and \rf{eqmain01} also hold  
 with the infimum over $\bM^n$ replaced by choice of $B$ equal to a ball centered at the $\bS^n$-barycenter of $\Omega$ with $\HH^n(B)=\HH^n(\Omega)$ (possibly with a different constant $C(a,\beta)$).
\end{theorema}

In the theorem we used the standard notation $\avint_{A} f\, d\mu = \frac1{\mu(A)}\int_A f\,d\mu$.
Notice that in \rf{eqmain00} we have
$$
\frac{\capp_{n-2}(B_{\bM^n}(x,atr_B)\setminus \Omega)}{(t\,r_B)^{n-3}}\lesssim t\,r_B,$$
because of the $(n-2)$-homogeneity of $\capp_{n-2}$.
Remark also that, for all $a,a'\in (0,1)$, we have
\begin{multline*}
\sup_{t\in (0,1)} \avint_{\partial_{\bM^n} ((1-t) B)} \frac{\capp_{n-2}(B_{\bM^n}(x,atr_B)\setminus \Omega)}{(t\,r_B)^{n-3}}\,d\HH^{n-1}(x)\\ \approx_{a,a'}
 \sup_{t\in (0,1)} \avint_{\partial_{\bM^n} ((1-t) B)} \frac{\capp_{n-2}(B_{\bM^n}(x,a'tr_B)\setminus \Omega)}{(t\,r_B)^{n-3}}\,d\HH^{n-1}(x).
\end{multline*}
 The same happens in the case $n=2$ concerning the integral on the right hand side of \rf{eqmain01}. See Lemmas~\ref{lemcompar} and \ref{lemcompar'}.
 
 The estimates in Theorem are sharp up to a constant factor. Indeed, 
as in \cite{BDV} and \cite{AKN2}, we can consider ellipsoidal perturbations of the unit ball, such as
$$\Omega_\ve = \{x\in\R^{n}:(1+\ve)x_1^2 + (1-\ve)x_2^2+x_3^2+\ldots+x_{n}^2\leq 1\},$$
with $\ve\to0$. 
 As remarked in \cite{BDV} and \cite{BD}, letting $B_\ve$ be a ball such that $\HH^{n}(B_\ve) = \HH^{n}(\Omega_\ve)$,
 it holds $\lambda_{\Omega_\ve} - \lambda_{B_\ve}\approx \ve^2$, uniformly as $\ve\to0$. On the other hand, it is easy to check that
the right hand side terms of \rf{eqmain00} and \rf{eqmain01} are also comparable to $\ve^2$, also uniformly as $\ve\to0$.
 
As far as we know, Theorem \ref{teomain0} provides the first sharp quantitative Faber-Krahn inequality stated in terms of the capacity of (subsets of) $B\setminus\Omega$. As remarked above, Theorem \ref{teomain0}
is related to the open problem 7.23 from \cite{BD}, which asks for a sharp quantitative Faber-Krahn inequality in terms of a suitable a capacitary asymmetry of $B$ and $\Omega$.
Notice however that Theorem \ref{teomain0} only involves the capacity of subsets of $B\setminus \Omega$, and not of $\Omega\setminus B$. This is a natural fact because there are domains such that $\capp_{n-2}(\Omega\setminus B)$ is large while $\lambda_\Omega - \lambda_B$ is as small as wished. Indeed, 
for $0<\ve\leq1/10$, consider
an $\ve$-neighborhood $U_\ve$ of the planar set $B(0,1)\cup[1,2]$ and contract it by a suitable factor so that 
the resulting domain $\Omega_\ve$ has the same area as $B=B(0,1)$. It is easy to check that $\lambda_{\Omega_\ve}\to\lambda_B$ as $\ve\to0$, but $\capp_{n-2}(\Omega_\ve\setminus B')$ is large for any given ball $B'$ with the same area as $\Omega_\ve$, due to the fact that $[-1,2]\subset \Omega_\ve$, for example.

Using the connection between Hausdorff contents and capacities (see Lemma \ref{lemcap-contingut} below), we deduce the following corollary from the preceding theorem.

\begin{coro}\label{coromain0}
Under the assumptions and notation of Theorem \ref{teomain0}, for any $s>n-2$ there is some constant $C(s,a,\beta)>0$ such that, in the case $n\geq 3$, we have
\begin{equation}\label{eqmain00**}
\lambda_\Omega - \lambda_{B_0} \geq C(s,a,\beta)\,
\inf_{B}\bigg(\sup_{t\in (0,1)} \avint_{\partial_{\bM^n} ((1-t) B)} \left(\frac{\HH^s_\infty(B_{\bM^n}(x,atr_B)\setminus \Omega)}{(t\,r_B)^s}\right)^{\frac{n-2}s}\,t\,r_B\,d\HH^{n-1}(x)\bigg)^2.
\end{equation}
%$$\lambda_\Omega - \lambda_B \gtrsim C(\beta) \left(\int_{\partial_{\bS^n} (sB)}
%\frac{\capp_{n-2}(B_{\bS^n}(y,c_3 (1-s)r_B)\setminus  \Omega)}{((1-s)r_B)^{n-3}}\,d\HH^{n-1}(y)\right)^2,$$
In the case $n=2$, for any $s>0$ we have
\begin{equation}\label{eqmain01**}
\lambda_\Omega - \lambda_{B_0} \geq C(s,a,\beta)\,
\inf_{B}\bigg(\sup_{t\in (0,1)} \avint_{\partial_{\bM^2} ((1-t) B)} \frac{t\,r_B}{
\log\left(\frac{(2atr_B)^s}{\HH^s_\infty(B_{\bM^n}(x,atr_B)\setminus \Omega)}\right)}\,d\HH^{1}(x)\bigg)^2.
\end{equation}
In the case $\bM^n=
\bS^n$, \rf{eqmain00**} and \rf{eqmain01**} also hold  
 with the infimum over $\bM^n$ replaced 
by the choice of $B$ equal to a ball centered at the $\bS^n$-barycenter of $\Omega$ with $\HH^n(B)=\HH^n(\Omega)$  
  (possibly with a different constant $C(s,a,\beta)$).
\end{coro}
\vv

In the case when  we just quantify $\lambda_\Omega - \lambda_{B_0}$ in terms of some integral over a suitable family of ``thick points'' we get a somewhat sharper result. To state it, we need some additional notation.
For given $c_0>0$ and $a\in (0,1)$, in the case $n\geq3$, we denote
\begin{equation}\label{eqtsc01}
T_{c_0}(\Omega,B,a) = \big\{x\in B\setminus \Omega:\capp_{n-2}(B_{\bM^n}(x,a\,\delta_B(x))\setminus \Omega)\geq c_0\,\delta_B(x)^{n-2}\big\},
\end{equation}
and, in the case $n=2$, 
\begin{equation}\label{eqtsc02}
 T_{c_0}(\Omega,B,a) = \big\{x\in B\setminus \Omega:\capp_L(B_{\bM^n}(x,a\,\delta_B(x))\setminus \Omega)\geq c_0\,\delta_B(x)\big\}.
\end{equation}
We should understand the condition in the definition of $T_{c_0}(\Omega,B,a)$ as a thickness type condition. In particular, if 
$\Omega$ satisfies the capacity density condition or CDC (see Section \ref{seccap}), or $\partial\Omega$ is lower $s$-content regular for some $s>n-2$, then $T_{c_0}(\Omega,B,a)  = B\setminus \Omega$, for $c_0$ small enough.

\begin{theorema}\label{teomain1}
Given $n\geq2$ and $0<s\leq n$, let $\bM^n$ be either $\R^n$, $\bS^n$ or $\bH^n$, and let $c_0,\beta>0$, $a\in (0,1)$.
Let $\Omega$ be a relatively open bounded subset of $\bM^n$ and let $B_0$ be a geodesic ball in $\bM^n$ 
such that
 $\HH^n(B_0)=\HH^n(\Omega)$.
For $\bM^n=\bS^n$, suppose that $\beta\leq \HH^n(\Omega)\leq \HH^n(\bS^n)-\beta$, while for
$\bM^n=\R^n$ or $\bM^n=\bH^n$ assume only that $\HH^n(\Omega)\leq \beta$.
 Denote by $\lambda_\Omega$ and $\lambda_{B_0}$ the first Dirichlet eigenvalues of $-\Delta_{\bM^n}$ in $\Omega$ and $B_0$, respectively.
Then, 
\begin{equation}\label{eqmain1}
\lambda_\Omega - \lambda_{B_0} \geq C(a,s,\beta,c_0)\,\inf_{B}\bigg(\int_{T_{c_0}(\Omega,B,a)} \delta_B(x)^{n-s}\,d\HH^{s}_\infty(x)\bigg)^2,
\end{equation}
where where the infimum is taken over all geodesic balls $B\subset\bM^n$  
such that $\HH^n(B)=\HH^n(\Omega)$. In the case $\bM^n=
\bS^n$, \rf{eqmain1} also holds 
 with the infimum over $\bM^n$ replaced by the choice
 of $B$ equal to a ball centered at the $\bS^n$-barycenter of $\Omega$ with $\HH^n(B)=\HH^n(\Omega)$ 
 (possibly with a different constant $C(a,s,\beta,c_0)$).
\end{theorema}

%Notice that if $\poms$ is $s$-lower content regular for some $s>n-2$, then $T_{c_0}(\Omega,B,a)=\poms$ for  suitable $c_0,a$ and thus the domain of integration of the integral on the right hand side of \rf{eqmain1} is the whole $\poms$.
\vv

 Remark that the integral of a function
$f:\bM^n\to[0,\infty)$ with respect to the Hausdorff content $\HH^{s}_\infty$ is given by
$$\int_{\bM^n}f\,d\HH^{s}_\infty = \int_0^\infty \HH^{s}_\infty(\{x\in \bM^n:f(t)>t\})\,dt.$$
Observe also that in the case when $\Omega$ satisfies the CDC, from \rf{eqmain1} we deduce that, for $0<s\leq n$,
$$\lambda_\Omega - \lambda_B \gtrsim \bigg(\int_{B\setminus \Omega} \delta_B(x)^{n-s}\,d\HH^{s}_\infty(x)\bigg)^2 \geq \bigg(\int_{B\cap \partial \Omega} \delta_B(x)^{n-s}\,d\HH^{s}_\infty(x)\bigg)^2.$$ 
The example given by the ellipsoidal perturbations of the unit ball mentioned above shows that this estimate is also sharp.  For $0<s<n$, the domain of integration in the middle term cannot be augmented to $B\triangle \Omega$, and the one in the last term to the full $\partial\Omega$ because of the same discussion
after the statement of Theorem \ref{teomain0}. 
%We are not aware of previous quantitative Faber-Krahn inequalities involving Hausdorff measures or contents of dimension smaller than $n$ in $\bM^n$.

%Remark that the constant $C$ above is uniform in on $\rad(B)$ as soon as $B$ does not degenerate, i.e., $\HH^n(B)$ is far away both from $0$ and $\HH^n(\bS^n)$.

A basic ingredient for the proof of Theorems \ref{teomain0} and \ref{teomain1} is Theorem  \ref{teoAKN} from \cite{AKN1}. %Indeed, essentially, we show that the right hand side of \rf{eqmain1} is bounded above by the right \rf{eqAKN}. 
The very rough strategy of the proof of both results is the following.
We consider $\bM^n$ embedded in $\R^{n+1}$ and we consider suitable extensions $\wt u_B$ and $\wt u_\Omega$ of the eigenfunctions $u_B$ and $u_\Omega$ appearing in \rf{eqAKN}, respectively, to some open subsets $\wt B,\wt\Omega\subset\R^{n+1}$, so that $\wt u_B$ is harmonic in $\wt B$ and $\wt u_\Omega$ is ``almost harmonic" in $\wt\Omega$ (i.e., $\Delta \wt u_\Omega$ is very small). 
Then  we apply apply \rf{eqAKN} by estimating $\wt u_B - \wt u_\Omega$ from below in terms of the harmonic measure for $\wt \Omega$,
using the fact that $\wt u_B(x)\approx \dist(x,\partial\wt B)$ while 
$\wt u_\Omega(x)=0$ in a large part of $\wt B\cap \partial \wt\Omega$.
 To relate this estimate to the 
terms on the right hand side of \rf{eqmain00}, \rf{eqmain01}, and \rf{eqmain1}, we consider an auxiliary Lipschitz domain $\wt \Omega_\Gamma$ and we relate the harmonic measure for $\wt \Omega$ to the one for 
$\wt \Omega_\Gamma$ using the maximum principle. 
%The last step consists in transferring our estimates in terms of the harmonic measure for $\wt \Omega_\Gamma$ to other estimates involving either the capacity of the Hausdorff content $\HH^{d}_\infty$ on $B\setminus \Omega$.
 Lastly we take advantage of the fact that
the behavior of harmonic measure is well understood on Lipschitz domains.

\subsection{The Alt-Caffarelli-Friedman monotonicity formula and the Friedland-Hayman inequality}

Given a domain in the unit sphere, $\Omega\subset \mathbb S^n\subset \R^{n+1}$ whose first Dirichlet
eigenvalue is $\lambda_\Omega$, the characteristic constant of $\Omega$ is the positive number $\alpha_\Omega$ such that $\lambda_\Omega = \alpha_\Omega(n-1+\alpha_\Omega)$.

Recall that the Alt-Caffarelli-Friedman (ACF) monotonicity formula asserts the following:
	
	\begin{theorem} \label{teoACF-elliptic}  Let $x \in  \R^{n+1}$ and $R>0$. Let $u_1,u_2\in
		W^{1,2}(B(x,R))\cap C(B(x,R))$ be nonnegative subharmonic functions such that $u_1(x)=u_2(x)=0$ and $u_1\cdot u_2\equiv 0$. 
		Set
		\begin{equation}\label{eqACF2}
			J(x,r) = \left(\frac{1}{r^{2}} \int_{B(x,r)} \frac{|\nabla u_1(y)|^{2}}{|y-x|^{n-1}}dy\right)\cdot \left(\frac{1}{r^{2}} \int_{B(x,r)} \frac{|\nabla u_2(y)|^{2}}{|y-x|^{n-1}}dy\right)
		\end{equation}
		Then $J(x,r)$ is an absolutely continuous function of $r\in (0,R)$ and
			\begin{equation}\label{eqprec1}
			\frac{\partial_rJ(x,r)}{J(x,r)}\geq \frac2r\bigl(\alpha_1 + \alpha_2  - 2 \bigr). 
		\end{equation}
		where $\alpha_i$ is the characteristic constant of the open subset  $\Omega_i\subset\bS^n$ given by  $$ \Omega_i=\bigl\{r^{-1}(y-x): y\in\partial B(x,r),\,u_i(y)>0\bigr\}.$$
		Further, for $r\in (0,R/2)$ and $i=1,2$, we have
		\begin{equation}\label{eqaux*}
		\frac{1}{r^{2}} \int_{B(x,r)} \frac{|\nabla u_i(y)|^{2}}{|y-x|^{n-1}}dy\lesssim \frac1{r^{n+1}}\|\nabla u_i\|_{L^2(B(x,2r))}^2.
		\end{equation}	
	\end{theorem}
	
\vv
The Friedland-Hayman \cite{FH} inequality ensures that, for any two disjoint open subsets $\Omega_1,\Omega_2\subset \bS^n$,
 $$\alpha_1+\alpha_2-2\geq 0,$$
so that $J(x,r)$ is non-decreasing on $r$, by \rf{eqprec1}.	
In fact, more is known.
By Sperner's inequality \cite{Sperner}, among all the open subsets with a fixed measure $\HH^n$ on $\bS^n$, the one that minimizes the characteristic constant 	is a spherical ball with the same measure $\HH^n$. That is to say,
if $B_i$ is a spherical ball such that $\HH^n(B_i) = \HH^n(\Omega_i)$ and $\bar\alpha_i$ denotes its characteristic constant, then
$$\alpha_i\geq
\bar\alpha_i.$$
Further, if one of the sets $\Omega_i$ differs from a hemisphere by a surface measure $h$,
that is, 
$$\Big|\HH^n(\Omega_i)- \frac12\HH^n(\bS^n)\Big| \geq h,$$
then 
$$\alpha_1+\alpha_2-2\geq c\,h^2.$$

In this paper we will deduce other more precise estimates for $\alpha_1+\alpha_2-2$ using Theorem \ref{teomain1} and Theorem \ref{teoAKN}.
To state the precise result, we need some additional notation.
Let  $\Omega_1,\Omega_2\subset \bS^{n}$ be open and disjoint and let $H\subset\R^{n+1}$ be half-space such that $0\in\partial H$. Denote
$$S_{H,1} = \bS^n \cap H,\qquad S_{H,2} = \bS^n \setminus \overline{H}.$$ 
We denote
$$V_{c_0}(\Omega_1,\Omega_2,H,a) = T_{c_0}(\Omega_1,S_{H,1},a) \cup T_{c_0}(\Omega_2,S_{H,2},a).$$
Notice that if both $\Omega_1,\Omega_2$ satisfy the the CDC and $c_0$ is small enough, then $V_{c_0} =(\Omega_1,\Omega_2,H,a) = T_{c_0}(S_{H,1}\setminus \Omega_1) \cup T_{c_0}(S_{H,2}\setminus \Omega_2)$.

For  fixed $c_0>0$ and $a\in(0,1)$, we denote for $0<s< n$,
$$\ve_s(\Omega_1,\Omega_2) = \inf_H \int_{V_{c_0}(\Omega_1,\Omega_2,H,a)} \dist(y,\partial H)^{n-s}\,d\HH_\infty^s(y),$$
with the infimum taken over all half-space $H$ such that $0\in\partial H$.
On the other hand, for $s=n$, we set
$$\ve_n(\Omega_1,\Omega_2) = \inf_H \HH^n\big((S_{H,1}\setminus \Omega_1) \cup (S_{H,2}\setminus \Omega_2)\big).$$
We will prove the following:

\begin{theorema}\label{teo22}
Let $n\geq2$ and  $0<s\leq n$.
Let $\Omega_1,\Omega_2\subset \bS^n$ be open and disjoint.
 For any $c_0>0$ and $a\in (0,1)$ we have
$$\ve_s(\Omega_1,\Omega_2)^2 \lesssim_{s,c_0,a} \min(1,\alpha_1 + \alpha_2 - 2).$$
\end{theorema}

Remark that when $\HH^n(\Omega_i)\to 0$, we have $\alpha_i\to\infty$, and thus $\alpha_1 + \alpha_2 - 2\to\infty$ too.
On the other hand, we always have $\ve_s(\Omega_1,\Omega_2)\lesssim 1$ by definition. This is the reason for
writing $\min(1,\alpha_1 + \alpha_2 - 2)$ instead of $\alpha_1 + \alpha_2 - 2$ on the right hand side of the inequality in Theorem \ref{teo22}.

\vv

% ***************************************************************************

\subsection{Carleson's conjecture}
Next we introduce the precise notion of a tangent point for a pair of disjoint open sets in $\R^{n+1}$.
For a point $x\in\R^{n+1}$, a unit vector $u$,
and an aperture parameter $a\in(0,1)$ we consider the two sided cone with axis in the direction of $u$ defined by
$$X_a(x,u)=\bigl\{y\in\R^{n+1}:|(y-x)\cdot u|> a|y-x|\bigr\}.$$
Given disjoint open sets $\Omega_1,\Omega_2\subset\R^{n+1}$ and $x\in\partial\Omega_1\cap\partial\Omega_2$,
we say that $x$ is a tangent point for the pair $\Omega_1,\Omega_2$ if $x\in\pom_1\cap\pom_2$ and there exists a unit vector $u$ such that, for all
$a\in(0,1)$, there exists some $r>0$ such that
$$(\partial \Omega_1\cup \partial \Omega_2)\cap X_a(x,u)  \cap B(x,r) =\varnothing,$$
and moreover, one component of $X_a(x,u)\cap B(x,r)$ is contained in $\Omega_1$ and the other in $\Omega_2$.
The hyperplane $L$ orthogonal to $u$ through $x$ is called a tangent hyperplane at $x$.  In case that $\Omega_2=\R^{n+1}\setminus\overline{\Omega_1}$, we say that $x$ is a tangent point for $\Omega_1$.
%Notice that this notion of tangent is associated with the domain $\Omega$, and it would be more appropriate to say that $L$ is a tangent for $\Omega$

  Let $\Omega_1$ be a Jordan domain in $\R^2$, and set $\Gamma =\partial\Omega_1$ and
$\Omega_2 = \R^2\setminus \overline{\Omega_1}$.
For $x\in\R^2$ and $r>0$, denote by $I_1(x,r)$ and $I_2(x,r)$ the longest open arcs of the circumference
$\partial B(x,r)$ contained in $\Omega_1$ and $\Omega_2$, respectively (they may be empty). Then we define
\begin{equation}\label{eqepsiloncoef}
\ve(x,r) = \frac1r\,\max\big(\big|\pi r- \HH^1(I_1(x,r))\big|,\, \big|\pi r- \HH^1(I_2(x,r))\big|\big).
\end{equation}
  The Carleson $\ve^2$-square function is given by
\begin{equation}\label{eqEE}
\EE(x)^2 :=\int_0^1 \ve(x,r)^2\,\frac{dr}r.
\end{equation}

  Carleson's conjecture, now a theorem, asserts the following.

\begin{theorem}\label{teo-carleson}
Let $\Omega_1\subset \R^2$ be a Jordan domain, let $\Gamma=
\partial\Omega_1$, and let $\EE$ be the
associated square function defined in \rf{eqEE}. Then
the set of tangent points for $\Omega_1$
coincides with the subset of those points $x\in\Gamma$ such that $\EE(x)<\infty$, up to a set of zero measure $\HH^1$. In particular, the set
$G=\{x\in\Gamma:\EE(x)<\infty\}$ is $1$-rectifiable.
\end{theorem}

Recall here that a set $E\subset\R^{n+1}$ is called $d$-rectifiable if there are  Lipschitz maps
$f_i:\R^d\to\R^{n+1}$, $i\in \N$, such that
\begin{equation}\label{eq001}
\HH^d\Big(E\setminus\textstyle\bigcup_{i=1}^{\infty} f_i(\R^d)\Big) = 0.
\end{equation}

The fact that $\EE(x)<\infty$ for $\HH^1$-a.e.\ tangent point in a Jordan curve was proved by Bishop in 
\cite{Bishop-thesis} (see also \cite{BCGJ}). The most difficult implication of Theorem \ref{teo-carleson}, i.e, the fact that the set $G$ is $1$-rectifiable and tangents to $\Gamma$ exist for $\HH^1$-a.e.\ $x\in G$,
was proved more recently by Ben Jaye and the last two authors of this paper \cite{JTV}.

It is natural to wonder about the existence of a suitable version of Carleson's conjecture in higher dimensions.
There are two natural questions: which coefficients should replace the coefficients $\ve(x,r)$ defined in
\rf{eqepsiloncoef}? and second, for which open sets $\Omega_1\subset\R^{n+1}$ should we expect to obtain
a characterization such as the one in Theorem \ref{teo-carleson}? In this theorem, the fact that
$\Omega_1$ is a Jordan domain ensures that
its boundary is connected, which plays an essential role in the proof. Indeed, the arguments in \cite{JTV} are based on the connectivity of the boundary and they do not extend to higher dimensions (this should not be a surprise, since typically the role of connectivity in the plane is much more relevant than in higher dimensions for many geometric problems). 

Regarding the coefficients $\ve(x,r)$, in \cite{AKN2} the authors show a very interesting connection
with the characteristic constants of spherical domains and the Friedland-Hayman inequality, which we proceed to describe. Given two disjoint open sets $\Omega_1,\Omega_2\subset \R^{n+1}$, consider the open subsets $\Sigma_1,\Sigma_2\subset \bS^n$ defined by
\begin{equation}\label{eqasigmai*}
\Sigma_i=\bigl\{r^{-1}(y-x): y\in\Omega_i\cap \partial B(x,r)\bigr\},
\end{equation}
and let $\alpha_i(x,r)= \alpha_{\Sigma_i}$, the characteristic constant of $\Sigma_i$. Analogously, set
$\lambda_i(x,r)=\lambda_{\Sigma_i}$.
In \cite{AKN2} it is remarked that, in the planar case $n=1$,
\begin{equation}\label{eqakn**}
\alpha_1(x,r) + \alpha_2(x,r) - 2 \gtrsim 
%(\alpha_1(x,r) - 1)^2 + (\alpha_2(x,r) - 1)^2 \approx 
\ve(x,r)^2.
\end{equation}
This is easy to check. Indeed, without loss of generality, assume $x=0$, $r=1$.
Taking into account that the characteristic of a domain decreases as its size increases, we have $\alpha_{I_i} = \alpha_{\Sigma_i}$ and also $\lambda_{I_i} = \lambda_{\Sigma_i}$,
for $I_i:=I_i(x,r)$ as in \rf{eqepsiloncoef}. 
Let $\gamma_i=\HH^1(I_i)/(2\pi)$. Since the first eigenfunction for $I_i$ is the function $u_i(\theta) =\sin((2\gamma_i)^{-1}\theta)$ (modulo a translation in the torus), we have $\alpha_i=\lambda_i^{1/2} = (2\gamma_i)^{-1}$.
%This shows that 
%$$\ve(x,r)^2 \approx \sum_{i=1}^2 \big|\pi- \HH^1(I_1)\big|^2 \approx \sum_{i=1}^2(\alpha_i(x,r) - 1)^2.$$
%To prove the remaining inequality in \rf{eqakn**}, 
Suppose, for example, that $\ve(x,r) = \big|\pi- \HH^1(I_1)\big|$ and write $\alpha_i =\alpha_i(x,r)$.
Let $\tilde\alpha_2$ the characteristic of
$\bS^1\setminus \overline{I_1}$. Since $I_2\subset \bS^1\setminus \overline{I_1}$, we have
$\alpha_2\geq \tilde\alpha_2$. Thus,
$$\alpha_1 + \alpha_2 -2 \geq \alpha_1 + \tilde\alpha_2 -2 = \frac1{2\gamma_1} + \frac1{2(1-\gamma_1)} - 2
= \frac{1-4\gamma_1(1-\gamma_1)}{2\gamma_1(1-\gamma_1)} = \frac{2(\frac12-\gamma_1)^2}{\gamma_1(1-\gamma_1)}
\approx \frac{\ve(x,r)^2}{\gamma_1(1-\gamma_1)},
$$
which completes the proof of \rf{eqakn**}, since $\gamma_1\in (0,1)$.
Further, in case that $I_1(x,r)$ and $I_2(x,r)$ are complementary arcs, arguing as above, one can deduce
$$\min\big(1,\alpha_1(x,r) + \alpha_2(x,r) - 2\big) %\approx (\alpha_1(x,r) - 1)^2 + (\alpha_2(x,r) - 1)^2 
\approx \ve(x,r)^2.$$
See also \cite{Bishop-conjectures} for a very related discussion.

In view of the preceding discussion, a possible generalization of Carleson's conjecture to higher dimensions
may consist in showing that, for a suitable  domain $\Omega_1\subset\R^{n+1}$ and $\Omega_2 = \R^{n+1}\setminus \overline{\Omega_1}$, it holds
\begin{equation}\label{equiv**}
\int_{0}^1 \frac{\min(1,\alpha_1(x,r) + \alpha_2(x,r) -2)}r\,dr <\infty \qquad\Leftrightarrow \qquad \text{$x$ is a tangent point of $\Omega_1$,}
\end{equation}
for every $x\in\partial\Omega_1$, up to a set of zero measure $\HH^n$.

In this paper we prove the implication $\Rightarrow$ in the above equivalence \rf{equiv**} for a very large class of domains. Further, we do not ask the domains $\Omega_1,\Omega_2$ to be complementary. 
The precise result is the
following.

\begin{theorema}\label{teoguai}
Let $\Omega_1,\Omega_2\subset \R^{n+1}$ be disjoint Wiener regular domains in $\R^{n+1}$.
Then, for $\HH^n$-a.e.\ tangent point $x$ for the pair $\Omega_1,\Omega_2$, 
it holds
\begin{equation}\label{eqint1}
\int_{0}^1 \frac{\min(1,\alpha_1(x,r) + \alpha_2(x,r) -2)}r\,dr <\infty.
\end{equation}
\end{theorema}

\vv
Recall that a domain is called Wiener regular if the Dirichlet problem for the Laplacian with continuous data is solvable in that domain. This is a property which is implied by the CDC, which is a strictly stronger condition (see Section \ref{secharm}).
We will prove Theorem \ref{teoguai} by integrating the Alt-Caffarelli-Friedman inequality
\rf{eqprec1} applied to the Green functions $u_i$ of $\Omega_i$ and using the fact that the harmonic measure of $\Omega_i$ satisfies $\omega_{\Omega_i}(B(x,r))\gtrsim r^n$ as $r\to 0$ at $\HH^n$-a.e.\ tangent point $x\in\partial\Omega_i$, which in turn implies that $J(x,0)>0$:
$$	\int_0^{r_0}\frac2r\bigl(\alpha_1(x,r) + \alpha_2(x,r)  - 2 \bigr)\,dr \leq \int_0^{r_0}\frac{\partial_rJ(x,r)}{J(x,r)}\,dr = \log\frac{J(x,r_0)}{J(x,0)}<\infty.
$$	
Putting together Theorem \ref{teo22} and Theorem \ref{teoguai}, for $\Omega_1,\Omega_2$ as
above and for $0<s\leq n$, denoting $\ve_s(x,r) = \ve_s(\Sigma_1,\Sigma_2)$, with $\Sigma_i$ as
in \rf{eqasigmai*}, we deduce that
$$\int_{0}^1 \ve_s(x,r)^2\,\frac{dr}r <\infty$$
for $\HH^n$-a.e.\ tangent point $x$ for the pair $\Omega_1,\Omega_2$.

Finally, we remark that in the companion paper \cite{FTV} we will prove the converse implication in \rf{equiv**}
when $\Omega_1\cup \Omega_2$ satisfies the CDC. 
The precise result is the following.

\begin{theorema}\label{teo-carleson2}
For $n\geq1$, let $\Omega_1,\Omega_2\subset\R^{n+1}$ be disjoint open sets. The set of points
$x\in\R^{n+1}$ such that 
$$\int_{0}^1 \ve_n(x,r)^2\,\frac{dr}r <\infty$$
is $n$-rectifiable. In case that the set $\Omega_1\cup\Omega_2$ satisfies the CDC, then
 $\HH^n$-a.e.\ $x\in\R^{n+1}$ such that
$$\int_{0}^1 \frac{\min(1,\alpha_1(x,r) + \alpha_2(x,r) -2)}r\,dr <\infty$$
is a tangent point for the pair $\Omega_1,\Omega_2$.
\end{theorema}

We insist on the fact that the CDC for $\Omega_1\cup\Omega_2$ is a quite mild lower regularity condition.
For example, it holds if\footnote{In fact, in \cite{Lewis} it is shown that this is also a necessary condition for the CDC.} there exist $s\in (n-1,n+1]$ and $c>0$ such that
$$\HH^s_\infty(B(x,r)\cap (\partial\Omega_1\cup\partial\Omega_2))\geq c\,r^s\quad\mbox{ for all $x\in
\partial\Omega_1\cup\partial\Omega_2$, $0<r\leq1$.}$$

Because of the discussion above, we think that this result can be considered as a natural extension
of Carleson's conjecture to higher dimensions. Observe that a pair $\Omega_1,\Omega_2$ consisting of
a Jordan domain in the plane and its complementary, as in \cite{JTV}, satisfies the assumptions in Theorem \ref{teo-carleson2}.
In fact, in the case of the plane, the above theorem applies to domains much more general than Jordan domains and so it is also new.

We remark that in \cite{AKN2} the authors propose another possible extension of Carleson's conjecture to
higher dimensions. Let
$$a(x,r)^2 := |\lambda_1(x,r) - n|^2 + |\lambda_2(x,r) - n|^2.$$
Assuming that $\Omega_2= \R^{n+1}\setminus \overline{\Omega_1}$,
they ask under what minimal assumptions on $\partial\Omega_1$ the set of those points $x\in\partial\Omega_1$ such that
\begin{equation}\label{eqAKN**}
\int_{0}^1 \frac{a(x,r)^2}r\,dr <\infty
\end{equation}
coincides with the rectifiable part of $\partial\Omega_1$, up to a set of $\HH^n$ measure zero.
As shown in \cite{AKN1}, one has 
\begin{equation}\label{eqAKN**a}
\min(1,a(x,r))^2\lesssim \min(1,\alpha_1(x,r) + \alpha_2(x,r) -2).
\end{equation}
Thus the condition \rf{eqint1} implies that 
 \begin{equation}\label{eqAKN***}
 \int_{0}^1 \frac{\min(1,a(x,r))^2}r\,dr <\infty,
\end{equation}
which is equivalent to \rf{eqAKN**} for ``reasonable'' domains, for example when both $\Omega_1$ and $\R^{n+1}\setminus\Omega_1$ are connected and have diameter larger than $1$.
By \rf{eqAKN**a} and Theorem \ref{teoguai}, the condition \rf{eqAKN***} holds for tangent points $x\in\partial\Omega_1$ in the case 
when both $\Omega_1$ and $\R^{n+1}\setminus\Omega_1$ are Wiener regular. However, 
the inequality converse to the one in \rf{eqAKN**a} does not hold in general, and so
it is not clear to
us if the condition \rf{eqAKN***} implies the existence of tangents for domains $\Omega_1,\Omega_2$ such
as the ones in Theorem \ref{teo-carleson2}.

\subsection{A counterexample}
For $n \geq 2$, let $u, v : B(0,10) \to \R$ be two nonnegative continuous functions satisfying 
\begin{equation}
	\mbox{$-\Delta u \leq 0$ in $\{u >0\}$, $-\Delta v \leq 0$ in $\{v>0\}$ and $uv=0$ in $B(0,10)$.} \nonumber%\label{count-1}	
\end{equation}
Define $\Gamma^*:= \{x \,:\, J(x,0^+) >0 \}$, where $J(0^+) := \lim_{r \to 0^+} J(x,r)$ and $J(x,r)$ is the ACF monotonicity formula defined in \eqref{eqACF2} relative to $u,v$.
In \cite{AKN2}, Allen, Kriventsov ans Neumayer prove the following theorem.
\begin{theorem}[{\cite[Theorem 1.2]{AKN2}}]\label{teo-AKN2-12}
	For $n\geq 2$, let $u, v$ as above. Then $\Gamma^* \cap B(0,1)$ is $(n-1)$-rectifiable.
\end{theorem}
\noindent
In the same work, see \cite[Problem 2.8]{AKN2}, the authors pose the converse question: suppose that $E=\partial \{u >0\} \cap \partial \{v>0\}$ is $(n-1)$-rectifiable. Is it true, then, that $J(x,0^+)>0$ for $\HH^{n-1}$-almost every $x \in E$ (under some minimal assumptions on $E$ and $u,v$)? In the last section of this paper, we provide an example which shows that, without further structural assumptions on the domains, the natural converse to Theorem \ref{teo-AKN2-12} is false. More precisely we construct two Wiener regular domains $\Omega_1,\Omega_2$ in $\R^2$ with $1$-rectifiable boundary and such that, if $E=\partial \Omega_1 \cap \partial \Omega_2$, then $\HH^{1}(E)>0$ and the limit $J(x,0^+)$ of the ACF functional associated to the Green functions of $\Omega_1,\Omega_2$ equals to $0$ for $\HH^1$-almost every point $x \in E$. See Section \ref{sec-counter} for the detailed construction. On the other hand, it is implicit in the proof of Theorem \ref{teoguai} that $J(x,0^+)>0$ whenever $x$ is a (true) tangent point of $E$ and $u$ and $v$ are the Green functions of disjoint Wiener regular domains (up to a set of $\HH^{n-1}$ measure zero).

\vv
{\bf Acknowledgements.} 
This work was initiated while the authors were in residence at the Hausdorff Institute of Mathematics in Spring 2022 during the program ``Interactions between geometric measure theory, singular integrals, and PDEs''. Other parts of this work took place during a two weeks  visit of M. Villa to Princeton University and during another one month visit of I. Fleschler to the Universitat Aut\`onoma de Barcelona.
I. Fleschler and M. Villa would like to thank G. De Philippis for some inspiring conversations in some early stages of this work. We also thank G. David for providing us some information in connection with the Friedland-Hayman inequality.
\vv

% ***************************************************************************

\section{Preliminaries}

\subsection{Miscellaneous notation}\label{secnot}
In the paper, constants denoted by $C$ or $c$ depend just on the dimension and perhaps other fixed
parameters, such as $a,\beta$ in Theorem \ref{teomain0}, for example. We will write $a\lesssim b$ if there is $C>0$ such that $a\leq Cb$ . We write $a\approx b$ if $a\lesssim b\lesssim a$.

Open balls in $\R^{n+1}$ centered in $x$ with radius $r>0$ are denoted by $B(x,r)$, and closed balls by 
$\bar B(x,r)$. For an open or closed ball $B\subset\R^{n+1}$ with radius $r$, we write $\rad(B)=r$.
 Similarly, open and closed balls in $\bM^n$ are denoted by $B_{\bM^n}(x,r)$ and  
$\bar B_{\bM^n}(x,r)$ respectively. For an open or closed ball $B\subset\bM^n$ with radius $r$, we write $\rad_{\bM^n}(B)=r$.
An open annulus in $\R^{n+1}$ centered in $x$ with inner radius $r_1$ and outer radius $r_2$ is denoted by 
$A(x,r_1,r_2)$, and the corresponding closed annulus by $\bar A(x,r_1,r_2)$. Similarly, the analogous annuli in $\bM^n$ 
are denoted by $A_{\bM^n}(x,r_1,r_2)$ and $\bar A_{\bM^n}(x,r_1,r_2)$, respectively.

We use the two notations $S(x,r)\equiv \partial B(x,r)$ for spheres
in $\R^{n+1}$ centered in $x$ with radius $r$, so that $\bS^n=S(0,1)$.

We call open sets in $\bS^n$ spherical domains, and the geodesic distance in $\bM^n$ is written as 
$\dist_{\bM^n}$. Notice that, assuming $\bS^n$ to be embedded in $\R^{n+1}$, the geodesic distance 
$\dist_{\bS^n}$ is comparable to the Euclidean distance in the ambient space $\R^{n+1}$.

Given a set $F\subset\R^{n+1}$, the notation $M_+(F)$ stands for the
set of (positive) Radon measures supported on $F$.
For $s>0$, the $s$-dimensional Hausdorff measure is denoted by $\HH^s$, and the $s$-dimensional Hausdorff
content by $\HH^s_\infty$. Recall that, for any set $E\subset \R^{n+1}$,
$$\HH^s_\infty(E) = \inf\Big\{\sum_i \diam(A_i)^s:E\subset\bigcup_i A_i\Big\}.$$
We say that $E$ is lower $s$-content regular if there exists some constant $c>0$ such that
$$\HH^s_\infty(B(x,r)\cap E)\geq c\,r^s\quad \text{ for all $x\in E$ and $0<r\leq\diam(E)$.}$$

The barycenter of a set $E\subset \R^n$ is defined by
$$x_E = \int_E x\,d\HH^n(x).$$
On the other hand, the $\bS^n$-barycenter of $E\subset\bS^n$ is defined by $x_E/|x_E|$, with $x_E$ as above. So this belongs to $\bS^n$ and it is defined only when $x_E\neq 0$.

\vv

\subsection{Capacities}\label{seccap}

The fundamental solution of the minus Laplacian in $\R^{n}$, for $n\geq3$, equals
$$\EE_n(x) = \frac{c_n}{|x|^{n-2}},$$
where $c_n= (n-2)\HH^{n-1}(\mathbb S^{n-1})$. In the plane,
the fundamental solution is
$$\EE_2(x) = \frac1{2\pi}\,\log \frac1{|x|}.$$
For a Radon measure $\mu$, we consider the potential defined by
$$U_{n-2}\mu(x) = \EE_n * \mu(x).$$
Given a set $F\subset\R^n$ (or, more generally, $F\subset \R^m$)
we define the capacity $\capp_{n-2}(F)$ by the identity
\begin{equation}\label{eqicapimu}
	\capp_{n-2}(F) = \frac 1{\inf_{\mu\in M_1(F)} I_{n-2}(\mu)},
\end{equation}
where the infimum is taken over all {\em probability} measures $\mu$ supported on $F$ and $I_{n-2}(\mu)$ is the energy
\begin{equation}\label{eqimu}
	I_{n-2}(\mu) = \iint \EE_n(x-y)\,d\mu(x)\,d\mu(y) = \int U_{n-2}\mu(x)\,d\mu(x).
\end{equation}
For $n\geq 3$ and $F\subset \R^n$, $\capp_{n-2}(F)$ is the Newtonian capacity of $F$, and for $n=2$ and $F\subset\R^2$, 
$\capp_0(F)$ is the Wiener capacity of $F$. Remark that for $n\geq3$, one also has
\begin{equation}\label{eqcap93}
\capp_{n-2}(F)=\sup\big\{\mu(F):\mu\in M_+(F),\,\|U_{n-2}\mu\|_{\infty,F}\leq1\big\},
\end{equation}
where $M_+(F)$ is the family of all Radon measures supported in $F$.
In the plane, the analogous identity with $n=2$ holds when $\diam F<1$. 

It is easy to check that the Newtonian capacity $\capp_{n-2}$ is
homogeneous of degree $n-2$ when $n\geq3$. The Wiener capacity is not homogeneous. However, the logarithmic capacity $\capp_L$, defined by
$$\capp_L(F) = e^{-\frac{2\pi}{\capp_0(F)}},$$
is $1$-homogeneous.

For $n\ge2$, let $\Omega\subset\R^n$ be an open set. We say that $\Omega$ satisfies the $n$-dimensional capacity density condition (CDC)
if there exists some constant $c>0$ such that, for every $x\in\partial\Omega$ and $r\in(0,\diam(\Omega))$, 
$$\capp_{n-2}(B(x,r)\setminus\Omega)\geq c\,r^{n-2}\quad \mbox{in the case $n\geq3$,}$$
and 
$$\capp_L(B(x,r)\setminus\Omega)\geq c\,r\quad \mbox{in the case $n=2$.}$$

For subsets $F\subset\bS^n$ and domains $\Omega\subset \bS^n$, we define $\capp_{n-2}(F)$ in the same way as
for $F\subset\R^n$ and domains $\Omega\subset \R^n$ (using the fundamental solution of the Laplacian in $\R^n$).
We define the $n$-dimensional CDC in $\bS^n$ by asking the following condition for every $x\in\partial\Omega$ and $r\in(0,\diam(\Omega))$:
$$\capp_{n-2}(B_{\bS^n}(x,r)\setminus\Omega)\geq c\,r^{n-2}\quad \mbox{in the case $n\geq3$,}$$
and 
$$\capp_L(B_{\bS^n}(x,r)\setminus\Omega)\geq c\,r\quad \mbox{in the case $n=2$.}$$

We will need the following auxiliary results.

\begin{lemma}\label{lemcap-contingut}
Let $E\subset\R^n$ or $E\subset\bS^n$ be compact and $n-2<s\leq n$. In the case $n>2$, we have
$$\capp_{n-2}(E) \gtrsim_{s,n} \HH_\infty^s(E)^{\frac{n-2}s}.$$
In the case $n=2$, we have
$$\capp_L(E) \gtrsim_s \HH_\infty^s(E)^{\frac1s}.$$
\end{lemma}

The proof of this result is an immediate consequence of Frostman's Lemma. See \cite[Chapter 8]{Mattila-gmt} for the case $n>2$, and 
\cite[Lemma 4]{Cufi-Tolsa-Verdera} for the case $n=2$, for example.

\begin{lemma}\label{lemmaprod1}
Let $\ell>0$, consider a compact set $F\subset\R^n$, and denote $\wt F= F\times [0,\ell]\subset\R^{n+1}$.
In the case $n\geq3$, we have
$$\capp_{n-1}(\wt F)\gtrsim \ell\,	\capp_{n-2}(F).$$
In the case $n=2$, given $C_1\geq1$, if $\diam(F)\leq C_1\,\ell$, we have
$$\capp_{1}(\wt F)\gtrsim \frac{\ell}{\log\dfrac{2C_1\ell}{\capp_L(F)}},$$
with the implicit constant depending on $C_1$.
\end{lemma}

\begin{proof}
By the $n-1$ homogeneity of $\capp_{n+1}(F)$ and of the inequalities stated in the lemma we can assume $\diam(F)<1$, so that
\rf{eqcap93} also holds in the case $n=2$.
Let $\mu\in	M_+(F)$ be such that $U_{n-2}\mu(x)\leq 1$ for all $x\in F$ and $\mu(F)=\capp_{n-2}(F)$ (it is well known that such measure, called equilibrium measure, exists for any compact set $F$). Denote $I_\ell= [0,\ell]$.
Consider the product measure $\wt\mu=\mu\times \HH^1|_{I_\ell}$, which is supported on $\wt F$.
For $x'=(x,x_{n+1})\in \wt F$, we have
\begin{equation}\label{eqsim1}
U_{n-1}\wt\mu (x') = \!\int\! \frac {c_n}{|x'-y'|^{n-1}} \,d\wt\mu(y') \approx \!\int_{y_{n+1}\in I_\ell}
\int_{y\in F} \frac {1}{|x-y|^{n-1} + |x_{n+1} - y_{n+1}|^{n-1}}  \,d\HH^1(y_{n+1})d\mu(y).
\end{equation}
For each $x,y\in F$, we have
\begin{multline*}
\int_{y_{n+1}\in I_\ell} \frac {1}{|x-y|^{n-1} + |x_{n+1} - y_{n+1}|^{n-1}}  \,d\HH^1(y_{n+1}) \\ = 
\int_{|x_{n+1}-y_{n+1}|\leq |x-y|}\!\!\!\cdots \,d\HH^1|_{I_\ell}(y_{n+1}) + \int_{|x-y|<|x_{n+1}-y_{n+1}|\leq \ell}\!\!\!\cdots\, d\HH^1|_{I_\ell}(y_{n+1}) =: T_1(x,y) + T_2(x,y).
\end{multline*}
Concerning $T_1(x,y)$, we have
$$T_1(x,y) \leq \int_{|x_{n+1}-y_{n+1}|\leq |x-y|}\frac {1}{|x-y|^{n-1}}  \,d\HH^1|_{I_\ell}(y_{n+1}) \leq \frac 1{|x-y|^{n-2}}.$$
Also,
$$T_2(x,y) \leq \int_{|x-y|<|x_{n+1}-y_{n+1}|\leq \ell} \frac {1}{|x_{n+1} - y_{n+1}|^{n-1}}  \,d\HH^1|_{I_\ell}(y_{n+1})
.$$
In the case $n\geq3$, it easily follows also that	
$$T_2(x,y)\lesssim \frac {1}{|x - y|^{n-2}}.$$	
Thus,	
$$U_{n-1}\wt\mu (x') \lesssim \int_{y\in F} \frac {1}{|x - y|^{n-2}}\,d\mu(y) \approx U_{n-2}\mu(x)\leq 1.$$
Consequently,
$$\capp_{n-1}(\wt F) \geq \frac{\wt \mu(F)}{\|U_{n-1}\wt \mu\|_{\infty,\wt F}} \gtrsim \wt\mu(\wt F) = \ell\,\mu(F)
\approx \ell\,\capp_{n-2}(F)\quad\mbox{ for $n\ge3$}.$$

In the case $n=2$, if $T_2(x,y)\neq 0$, we have
$$T_2(x,y) \leq \log\frac{\ell}{|x-y|}.$$
Since $|x-y|\leq C_1\ell$ for any $x,y\in F$, we can always write $T_2(x,y) \leq \log\frac{C_1\ell}{|x-y|},$
and so
$$T_1(x,y) + T_2(x,y) \leq 1+ \log\frac{C_1\ell}{|x-y|} \lesssim \log\frac{2C_1\ell}{|x-y|}.$$
Thus, 	
\begin{align*}
U_1\wt\mu (x') &\lesssim \int_{y\in F}  \log\frac{2C_1\ell}{|x-y|}\,d\mu(y) = \log(2C_1\ell)\,\mu(F) + 
\int_{y\in F}  \log\frac{1}{|x-y|}\,d\mu(y)\\
& = \log(2C_1\ell)\,\mu(F) + 2\pi\,U_0\mu(x)\leq \log(2C_1\ell)\,\mu(F) + 2\pi.
\end{align*}
Therefore,
$$\capp_1(F)  
\geq \frac{\wt \mu(F)}{\|U_{1}\wt \mu\|_{\infty,\wt F}} \gtrsim \frac{\ell\,\mu(F)}{\log(2C_1\ell)\,\mu(F) + 2\pi}
= \frac{\ell}{\log(2C_1\ell) + \frac{2\pi}{\mu(F)}}.
$$
Writing
$$\frac{2\pi}{\mu(F)} = \frac{2\pi}{\capp_0(F)} = \log\frac1{\capp_L(F)},$$
we obtain
$$\capp_1(F) \gtrsim \frac{\ell}{\log(2C_1\ell) + \log\frac1{\capp_L(F)}} = \frac{\ell}{\log\frac{2C_1\ell}{\capp_L(F)}}
.$$
\end{proof}
\vv

\begin{lemma}\label{lemmaprod2}
Let $\ell>0$, $F\subset\bS^n$ compact, and let $I_\ell$ be an interval of length $\ell$ contained in $[1/4,2]$. Denote
$$\wt F= \{x'\in \R^{n+1}: x' = tx \mbox{ for some $x\in F$ and $t\in I_\ell$}\}.$$
Then we have
$$\capp_{n-1}(\wt F)\gtrsim \ell\,	\capp_{n-2}(F) \quad \mbox{in the case $n\geq 3$,}$$
and, in the case $n=2$, if $\diam(F)\leq C_1\,\ell$,
$$\capp_{1}(\wt F)\gtrsim \frac{\ell}{\log\frac{2C_1\ell}{\capp_L(F)}}.$$
\end{lemma}

\begin{proof}
The proof is very similar to the one of Lemma \ref{lemmaprod1}.
We assume $\diam(F)<1$, so that
\rf{eqcap93} holds in the case $n=2$. We take
$\mu\in	M_+(F)$ be such that $U_{n-2}\mu(x)\leq 1$ for all $x\in F$ and $\mu(F)=\capp_{n-2}(F)$. 
Then we consider the ``product measure'' $\wt \mu$ defined by
$$\int f\,d\wt\mu = \int_{t\in I_\ell} \int_{x\in F}f(tx)\,d\mu(x)\,dt,\quad \text{ for $f\in C(\R^{n+1})$}.$$
Notice that $\wt\mu$ is supported on $\wt F$.

For $x'=sx\in \wt F$, with $x\in F$, $s\in I_\ell$, we have
$$U_{n-1}\wt\mu (x') = \int \frac {c_n}{|x'-y'|^{n-1}} \,d\wt\mu(y') = \int_{t\in I_\ell} 
\int_{y\in F}\frac {c_n}{|sx-ty|^{n-1}} \,d\mu(y)\, dt
$$
Now we claim that
$$|sx-ty|\gtrsim |x-y| + |s-t|.$$
Indeed, we have
$$|s-t| = ||sx| - |ty|| \leq |sx - ty|.$$
Also,
$$|x-y| = \frac1s\,|sx-sy| \leq \frac1s\,\big(|sx-ty| + |ty - sy|\big) \approx |sx-ty| + |t - s|\lesssim |sx - ty|.$$
Adding the last two inequalities, the claim follows.
Then we deduce
$$U_{n-1}\wt\mu (x') \lesssim\int_{t\in I_\ell} 
\int_{y\in F}\frac1{|x-y|^{n-1} + |s-t|^{n-1}} \,d\mu(y) dt
$$
Notice now the similarity between this inequality and \rf{eqsim1}. Then by the same arguments following \rf{eqsim1}
in the proof of Lemma \ref{lemmaprod1} we obtain the desired estimates.
\end{proof}

\vv

In the next lemma we show the equivalence of the statement of Theorem \ref{teomain0} with different
values of the
parameter $a$ in the case $n\geq3$.

\begin{lemma}\label{lemcompar}
Given $n\geq3$, let $\bM^n$ be either $\R^n$ or $\bS^n$ and let $\beta>0$.
Let $\Omega$ be a relatively open subset of $\bM^n$ and let $B$ be a geodesic ball in $\bM^n$ such that
$\HH^n(B)=\HH^n(\Omega)$
with radius $r_B$.
In the case $\bM^n=\bS^n$, suppose also that $\beta\leq \HH^n(\Omega)\leq \HH^n(\bS^n)-\beta$, and in the case
$\bM^n= \R^n$ just that $\HH^n(\Omega)\leq \beta$.
For $a,a'\in (0,1)$, we have
%\begin{multline*}
%\sup_{r\in (0,r_B)} \avint_{\partial_{\bM^n} \frac{r}{r_B}B} \frac{\capp_{n-2}(B_{\bM^n}(x,a(r_B-r))\setminus \Omega)}{(r_B-r)^{n-3}}\,d\HH^{n-1}(x)\\ \approx_{a,a'}
% \sup_{t\in (0,1)} \avint_{\partial_{\bM^n} ((1-t) B)} \frac{\capp_{n-2}(B_{\bM^n}(x,a'tr_B)\setminus \Omega)}{(t\,r_B)^{n-3}}\,d\HH^{n-1}(x).
%\end{multline*}
\begin{multline*}
\sup_{t\in (0,1)} \avint_{\partial_{\bM^n} ((1-t) B)} \frac{\capp_{n-2}(B_{\bM^n}(x,atr_B)\setminus \Omega)}{(t\,r_B)^{n-3}}\,d\HH^{n-1}(x)\\ \approx_{a,a'}
 \sup_{t\in (0,1)} \avint_{\partial_{\bM^n} ((1-t) B)} \frac{\capp_{n-2}(B_{\bM^n}(x,a'tr_B)\setminus \Omega)}{(t\,r_B)^{n-3}}\,d\HH^{n-1}(x).
\end{multline*}
\end{lemma}

\begin{proof}
To shorten notation, for any $a,t\in(0,1)$, write
$$I_{a,t} := \avint_{\partial_{\bM^n} ((1-t) B)} \frac{\capp_{n-2}(B_{\bM^n}(x,atr_B)\setminus \Omega)}{(t\,r_B)^{n-3}}\,d\HH^{n-1}(x).$$
Assume $a>a'$. Then it is clear that,  for every $t\in (0,1)$, $I_{a,t}\geq I_{a',t}$
since $\capp_{n-2}(B_{\bM^n}(x,atr_B)\setminus \Omega)\geq \capp_{n-2}(B_{\bM^n}(x,a'tr_B)\setminus \Omega)$.
So is suffices to show that $I_{a,t}\lesssim_{a,a'} \sup_{s\in(0,1)} I_{a',s}$.

Denote by $x_B$ the center of $B$. For a given $t\in (0,1)$ and $r=(1-t)r_B$, we consider a covering of $\partial_{\bM^n}B_{\bM^n}(x_B,r)$ by a family of $\bM^n$-balls $\{B_i\}_{i\in I_r}$ (where $I_r$ is just a suitable set of indices) centered in $\partial_{\bM^n}B_{\bM^n}(x_B,r)$ with radius $b(r_B-r) = btr_B$, with $b\ll a'$ to be chosen below, so that the family $\{B_i\}_{i\in I_r}$ has bounded overlap.
Then, denoting by $x_i$ the center of $B_i$, we split
\begin{align*}
\int_{\partial_{\bM^n} ((1-t) B)} \capp_{n-2} (B_{\bM^n}&(x,atr_B)\setminus \Omega)\,d\HH^{n-1}(x)\\
& \leq \sum_{i\in I_r} \int_{\partial_{\bM^n} ((1-t) B)\cap B_i} \capp_{n-2}(B_{\bM^n}(x,atr_B)\setminus \Omega)\,d\HH^{n-1}(x)\\
& \lesssim \sum_{i\in I_r} \sup_{x\in B_i}\capp_{n-2}(B_{\bM^n}(x,atr_B)\setminus \Omega)\,\rad_{\bM^n}(B_i)^{n-1}\\
& \leq \sum_{i\in I_r} \capp_{n-2}(B_{\bM^n}(x_i,(a+b)tr_B)\setminus \Omega)\,(tr_B)^{n-1}.
\end{align*}

Next we wish to cover the annulus $A_r=A_{\bM^n}(x_B,r-atr_B,r+atr_B)$ by another suitable family of balls. First we let
$N$ be the least integer such that $N\geq 2a/b$, and then we consider the radii
$$r_k= r + k\,\frac{atr_B}{N} \quad \mbox{ for $-N\leq k\leq N$.}$$
So we have $[r-atr_B,r+atr_B]= \bigcup_{k=-N}^{N-1} [r_k,r_{k+1}]$, and $r_{k+1} - r_k = \frac{atr_B}{N} \leq \frac{btr_B}2$. It is easy to check that there exist covering of $A_r$ by 
a family of $\bM^n$-balls $\Delta_j$, $j\in J_r$, with $\rad_{\bM^n}(\Delta_j) = btr_B$,
which are centered in $\bigcup_{k=-N}^N \partial_{\bM^n}B(x_B,r_k)$, and have bounded overlap. 
We write $j\in J_{r,k}$ if $\Delta_j$ is centered in $\partial_{\bM^n}B(x_B,r_k)$.
For $i\in I_r$, we also write
$\wt B_i = B_{\bM^n}(x_i,(a+b)tr_B)$ to shorten notation. By the subadditivity of $\capp_{n-2}$, then we have
\begin{align}\label{eqbidelta}
\sum_{i\in I_r} \capp_{n-2}(\wt B_i\setminus \Omega)\,(tr_B)^{n-1} & \leq
\sum_{i\in I_r} \sum_{j:\Delta_j\cap \wt B_i\neq\varnothing} \capp_{n-2}(\Delta_j\setminus \Omega)\,(tr_B)^{n-1}\\
& \leq C(a,b) \sum_{j\in J_r} \capp_{n-2}(\Delta_j\setminus \Omega)\,(tr_B)^{n-1},\nonumber
\end{align}
taking into account that for each $\Delta_j$ there is a bounded number (depending on $a$ and $b$) of balls $\wt B_i$ which intersect it. Using also that,  for any $x\in\Delta_j$, $\Delta_j\subset B_{\bM^n}(x,2btr_B)$,
we get
\begin{align*}
 \sum_{j\in J_r} \capp_{n-2}(\Delta_j\setminus \Omega)\,(tr_B)^{n-1} & \approx_{b} \sum_{k=-N}^N  
 \sum_{j\in J_{r,k}} \capp_{n-2}(\Delta_j\setminus \Omega)\,\rad_{\bM^n}(\Delta_j)^{n-1}\\
 & \lesssim \sum_{k=-N}^N \sum_{j\in J_{r,k}} \int_{\partial_{\bM^n}B(x_B,r_k)\cap\Delta_j}
 \capp_{n-2}(B_{\bM^n}(x,2btr_B)\setminus \Omega)\,d\HH^{n-1}(x)\\
 & \lesssim  \sum_{k=-N}^N  \int_{\partial_{\bM^n}B(x_B,r_k)}
 \capp_{n-2}(B_{\bM^n}(x,2btr_B)\setminus \Omega)\,d\HH^{n-1}(x)
\end{align*}
Writing $t_k = \frac{r_B-r_k}{r_B}$ and taking into account that $t_k\approx_{a} t$, it easily follows
that
\begin{align*}
I_{a,t} & =\avint_{\partial_{\bM^n} ((1-t) B)} \frac{\capp_{n-2}(B_{\bM^n}(x,atr_B)\setminus \Omega)}{(t\,r_B)^{n-3}}\,d\HH^{n-1}(x) \\
& \lesssim C(a,b) 
\sum_{k=-N}^N  \avint_{\partial_{\bM^n} ((1-t_k)B)}
 \frac{\capp_{n-2}(B_{\bM^n}(x,Cbt_kr_B)\setminus \Omega)}{(t_k\,r_B)^{n-3}}
 \,d\HH^{n-1}(x) \\ &\lesssim C'(a,b) \sup_{s\in (0,1)}I_{Cb,s}.
\end{align*}
So choosing $b$ so that $Cb\leq a'$, we deduce $I_{a,t}\lesssim_{a,a'} \sup_{s\in(0,1)} I_{a',s}$, as wished.
\end{proof}
\vv

The following lemma is the analogue of the preceding result for the case $n=2$.

\begin{lemma}\label{lemcompar'}
Let $\bM^2$ be either $\R^2$ or $\bS^2$, and let $\beta>0$.
Let $\Omega$ be a relatively open subset of $\bM^2$ and let $B$ be a geodesic ball in $\bM^2$  such that $\HH^2(B)=\HH^2(\Omega)$ with radius $r_B$.
In the case $\bM^2=\bS^2$, suppose also that $\beta\leq \HH^2(\Omega)\leq \HH^2(\bS^2)-\beta$, and in the case
$\bM^2= \R^2$ just that $\HH^2(\Omega)\leq \beta$.
For $a,a'\in (0,1)$, we have
\begin{multline*}
\sup_{t\in (0,1)} \avint_{\partial_{\bM^2} ((1-t) B)} \frac{t\,r_B}{
\log\frac{2ta r_B}{\capp_L(B_{\bM^2}(x,at r_B)\setminus \Omega)}}\,d\HH^{1}(x)\\
\approx \sup_{t\in (0,1)} \avint_{\partial_{\bM^2} ((1-t) B)} \frac{t\,r_B}{
\log\frac{2ta' r_B}{\capp_L(B_{\bM^2}(x,a't r_B)\setminus \Omega)}}\,d\HH^{1}(x).
\end{multline*}
\end{lemma}

\begin{proof}
The arguments are very similar to the ones for the preceding lemma with $n\geq3$.
The main change in the proof is that instead of using the subadditivity of $\capp_{n-2}$, we use the fact
that, by Theorem 5.1.4 from \cite{Ransford},
\begin{align*}\label{eqbidelta'}
\frac1{
\log\frac{2ta r_B}{\capp_L(B_{\bM^2}(x,at r_B)\setminus \Omega)}} \leq 
\sum_{j:\Delta_j\cap \wt B_i\neq\varnothing}
\frac1{
\log\frac{2ta r_B}{\capp_L(\Delta_j\setminus \Omega)}} \leq
\sum_{j:\Delta_j\cap \wt B_i\neq\varnothing}
\frac1{
\log\frac{2\,\rad_{\bS^2}(\Delta_j)}{\capp_L(\Delta_j\setminus \Omega)}}.
\end{align*}
We leave the details for the reader.
\end{proof}

\vv

\subsection{Harmonic measure}\label{secharm}

A bounded open set $\Omega\subset\R^{n+1}$ is Wiener regular if for any continuous $f:\pom\to\R$ there exists the solution 
of the Dirichlet problem for the Laplace equation. The Wiener criterion characterizes the Wiener regularity of any bounded open. For our purposes, it suffices to know that, in the case $n\geq 2$, if for any $x\in\partial\Omega$ there exists some $r_x>0$ and $c=c(x)>0$ such that 
\begin{equation}\label{eqwiener}
\capp_{n-1}(B(x,r)\setminus \Omega) \geq c\,r^{n-1}\quad \mbox{ for $0<r\leq r_x$,}
\end{equation}
then $\Omega$ is Wiener regular. In the case $n=1$, instead a sufficient condition for Wiener regularity is that $\capp_L(B(x,r)\setminus \Omega) \geq c\,r$ for $0<r\leq r_x$, with $c=c(x)>0$.

Recall that if $\Omega\subset \R^{n+1}$ is a bounded Wiener regular open set and $f:\overline\Omega\to\R$ 
is continuous and of class $C^2$ in $\Omega$, then
$$f(x)  = \int_{\partial\Omega} f\,d\omega_{\Omega}^x - \int_{\Omega} \Delta f(y)\,g_{\Omega}(x,y)\,dy,$$
where $g_\Omega$ stands for the Green function for $\Omega$.

The following result is well known\footnote{In some references, it is called ``Bourgain's lemma''.}. See Section 2.1 from \cite{Tolsa-fractional}, for example.

\vv
\begin{lemma}
	\label{lem2.1}
	For $n\geq1$ there exists $c(n)>0$ such that, given  an open set $\Omega\subset \R^{n+1}$ and a closed ball $B$ intersecting $\pom$, it holds 
	$$\omega^{x}(B)\geq c(n)\, \frac{\capp_{n-1}(\tfrac14 B\setminus\Omega)}{\rad(B)^{n-1}}\quad \mbox{  for all $x\in \tfrac14 B\cap \Omega$
	when $n\geq2$} ,$$
and	
$$ \omega^{x}(B)\geq c(n)\,
 \frac1{\log\dfrac{\rad(B)}{\capp_L(\frac14B\setminus\Omega)}}\qquad \mbox{ for all $x\in \tfrac14 B\cap \Omega$ when $n=1$.}
 $$
\end{lemma}

\vv

The following two lemmas are valid for the so called NTA domains. However, we only state them for
Lipschitz domains, which suffices for the purposes of this paper. For the proof, see \cite{Jerison-Kenig}.

\begin{lemma}\label{lem:wG}
Let $n\ge 1$ and $\Omega\subset\R^{n+1}$ be a Lipschitz domain.
Let $B$ be a closed ball centered in $\partial\Omega$. Then
\begin{equation}\label{eq:Green-lowerbound1}
 \omega^{p}(B)\approx  r(B)^{n-1}\, g_\Omega(p,y)\quad\mbox{
 for all $p\in \Omega\backslash  2B$ and $y\in B\cap\Omega$,}
 \end{equation}
 with the implicit constant just depending on $n$ and the Lipschitz character of $\Omega$. 
\end{lemma}

The following theorem states the so called ``change of pole formula". Again, this holds for NTA domains (see \cite{Jerison-Kenig}) but we only state for Lipschitz domains.

\begin{lemma}[Change of pole formula]\label{lempole}
For $n\geq 1$, let $\Omega\subset\R^{n+1}$ be a Lipschitz domain and let $B$ be a ball centered in $\partial\Omega$.
Let $p_1,p_2\in\Omega$ such that $\dist(p_i,B\cap \partial\Omega)\geq c_1^{-1}\,r(B)$ for $i=1,2$.
Then, for any Borel set $E\subset B\cap\partial\Omega$,
$$\frac{\omega^{p_1}(E)}{\omega^{p_1}(B)}\approx \frac{\omega^{p_2}(E)}{\omega^{p_2}(B)},$$
with the implicit constant depending only on $n$, $c_1$, and the Lipschitz character of $\Omega$. 
\end{lemma}

\vv

% ***************************************************************************

\section{Proof of Theorem \ref{teomain0} and beginning of the proof of Theorem \ref{teomain1} in the Wiener regular case
for $\bM^n=\bS^n$}
\label{seccontra}

In this section we assume that we are under the conditions of Theorems \ref{teomain0} and \ref{teomain1} and that $\bM^n=\bS^n$, and 
we fix  $\beta>0$ as in both theorems. In the proof we will allow all the implicit constants in the notation ``$\lesssim$'' to depend on $\beta$. We let $B$ be a ball with the same barycenter as $\Omega$ such that
$\HH^n(B)=\HH^n(\Omega)$.
We may assume that the barycenter of $\Omega$ is well defined, because otherwise, by 
Theorem \ref{teoAKN}, $\lambda_\Omega-\lambda_B\gtrsim \HH^n(\Omega\triangle\bS^n)^2\gtrsim1$ and then Theorems \ref{teomain0} and \ref{teomain1} are trivial.

% ***************************************************************************

%\subsection{The contradiction}
As in Theorem \ref{teoAKN}, we let  $u_\Omega$ and $u_B$ be the corresponding eigenfunctions normalized so that
they are positive and $\|u_\Omega\|_{L^2(\bS^n)} = \|u_B\|_{L^2(\bS^n)}=1$. Let $\alpha_\Omega,\alpha_B$ be 
the respective characteristic constants of $\Omega$ and $B$, so that $\lambda_\Omega = \alpha_\Omega \,(\alpha_\Omega + n-1)$ and $\lambda_B = \alpha_B\, (\alpha_B + n-1)$. Let $\wt u_\Omega$ and $\wt u_B$ be the $\alpha_B$-homogeneous extensions of $u_\Omega$ and $u_B$, respectively.
That is,
$$\wt u_\Omega(y) = |y|^{\alpha_B}\,u_\Omega(|y|^{-1}y),\qquad \wt u_B(y) = |y|^{\alpha_B}\,u_B(|y|^{-1}y).$$
Also, denote by $\wt \Omega$ and $\wt B$ the following truncated conical domains generated by $\Omega$ and $B$:
%$$\wt\Omega_0 = \{y\in \R^{n+1}:|y|^{-1}y\in \Omega\},\qquad \wt B_0 = \{y\in \R^{n+1}:|y|^{-1}y\in B\},$$
%and set
%$$\wt\Omega = \{y\in \R^{n+1}:|y|^{-1}y\in \Omega\},\qquad \wt B = \{y\in A(0,1/4,1):|y|^{-1}y\in B\}.$$
$$\wt\Omega = \{y\in A(0,1/4,1):|y|^{-1}y\in \Omega\},\qquad \wt B = \{y\in A(0,1/4,1):|y|^{-1}y\in B\},$$
and the enlarged versions
$$\wt\Omega' = \{y\in A(0,1/8,2):|y|^{-1}y\in \Omega\},\qquad \wt B' = \{y\in A(0,1/8,2):|y|^{-1}y\in B\}.$$
We also set
$$S=\partial A(0,1/4,1).$$
Notice that $S$ contains some part of the boundary of $\wt B$.
To prove Theorem \ref{teomain1}, first we will assume that $\wt\Omega$ and $\wt\Omega'$ are Wiener regular domains in $\R^{n+1}$. Later on we will deduce the general result from the Wiener regular case.

From the identity in spherical coordinates
$$\Delta f = \partial_{rr} f + \frac nr\,\partial_r f + \frac1{r^2}\,\Delta_{\bS^n}f$$
for a homogeneous function $f$ defined in $\R^{n+1}$,
it is immediate that $\Delta \wt u_{B} = 0$ in $\wt B'$, and so in $\wt B$. On the other hand, $\Delta \wt u_{\Omega} \neq 0$ in $\wt\Omega'$ in general, unless
$\alpha_\Omega=\alpha_B$. Instead, we have
\begin{align}\label{eqlaplace3}
\Delta \wt u_\Omega(r,\theta) & = \partial_{rr} \wt u_\Omega(r,\theta) + \frac nr\,\partial_r \wt u_\Omega(r,\theta) + \frac1{r^2}\,\Delta_{\bS^n}  u_\Omega(\theta)\\
& = r^{\alpha_B-2}\big[\big(\alpha_B(\alpha_B-1) + n \alpha_B\big)u_\Omega(\theta) + \Delta_{\bS^n}u(\theta)\notag\\
& = r^{\alpha_B-2}\,(\lambda_B- \lambda_\Omega)\,u_\Omega(\theta) = r^{-2}\,(\lambda_B- \lambda_\Omega)\,u_{\wt \Omega}(r,\theta).\notag
\end{align}

%$$\wt\Omega = \wt\Omega_0 \cap A(0,1/4,1)\qquad \wt B = \wt B_0 \cap A(0,1/4,1).$$
We will need to use the following auxiliary result below. This may be known to experts but we provide a detailed proof for completeness.

\begin{lemma}[Basic estimates] \label{lemubom}
We have
\begin{equation}\label{equb30}
\wt u_B(y) \approx \dist(y,\partial\wt B' \cap A(0,1/8,2))\quad \mbox{ for all $y\in \wt B'$}
\end{equation}
%Also,
%$$\wt u_B(y) \lesssim \dist(y,\partial\wt B')\quad \mbox{ for all $y\in \wt B'$}$$
and
\begin{equation}\label{equb31}
\wt u_\Omega(y) \lesssim 1\quad \mbox{ for all $y\in \wt \Omega'$.}
\end{equation}
Further, $\omega_{\wt B'}$ and $\omega_{\wt B}$ are mutually absolutely continuous with $\HH^n|_{\partial \wt B'}$ and $\HH^n|_{\partial \wt B}$ respectively, and for $x\in A(0,1/2,3/4)\cap \wt\Omega$ such that $|x|^{-1}x\in \frac12 B$, we have
%$$\frac{d\omega_{\wt B'}^x}{d\HH^n|_{\partial \wt B'}} \lesssim 1\quad \mbox{ in $\partial \wt B'\cap A(0,1/4,1)$}$$
\begin{equation}\label{equb32}
\frac{d\omega_{\wt B'}^x}{d\HH^n|_{\partial \wt B'}} \lesssim 1\quad \mbox{ in $\partial \wt B'$}
\end{equation}
and
\begin{equation}\label{equb33}
\frac{d\omega_{\wt B}^x}{d\HH^n|_{\partial \wt B}} \approx 1\quad \mbox{ in $\partial \wt B\cap A(0,1/3,4/5)$.}
\end{equation}
\end{lemma}

Remark that the implicit constants involved in the estimates of this lemma depend on the constant $\beta$.

\begin{proof}
First we will show that 
\begin{equation}\label{equb34}
\|\wt u_B\|_{\infty,\wt B'}\approx 1.
\end{equation} 
By homogeneity and Cauchy-Schwarz, we have
$$\|\wt u_B\|_{\infty,\wt B'}\approx \|u_B\|_{\infty,B}
\gtrsim  \|u_B\|_{L^2(B)} =1.$$
To prove $\|u_B\|_{\infty, B}\lesssim 1$, let $y\in B$, consider the (Euclidean) ball $B_y=B(y,1/2)$ and
extend $\wt u_B$ by $0$ in $\R^{n+1}\setminus \wt B'$. Notice that $B_y$ may intersect $\partial\wt B'\cap A(0,1/8,2)$. It is immediate to check that $\wt u_B$
is subharmonic in $B_y$. Therefore, using also the  homogeneity of $\wt u_B$, we get
$$u_B(y) = \wt u_B(y) \leq \avint_{B_y} \wt u_B\,dx \leq \left(\avint_{B_y} |\wt u_B|^2\,dx\right)^{1/2}
\lesssim  \|u_B\|_{L^2(B)}\lesssim 1,$$
which completes the proof of \rf{equb34}.

Remark that essentially the same argument used to prove that $\|u_B\|_{\infty, B}\lesssim 1$ shows that
$$\|\wt u_\Omega\|_{\infty, \wt \Omega'}\lesssim 1,$$
proving \rf{equb31} (in this case we have to argue with the $\alpha_\Omega$-homogeneous extension of $u_\Omega$ in place of $\wt u_\Omega$).

Next we deal with \rf{equb30}. Again by homogeneity it suffices to show that
\begin{equation}\label{equb35}
u_B(y) \approx \dist(y,\partial_{\bS^n} B)\quad \mbox{ for all $y\in B$}.
\end{equation}
First we will prove that $u_B(y) \lesssim \dist(y,\partial_{\bS^n} B)$. Since $\| u_B\|_{\infty, B}\lesssim 1,$ we can assume that $\dist(y,\partial_{\bS^n} B)\leq \frac1{10}\,\rad_{\bS^n}(B)$.
For a fixed $y_0\in B$ satisfying this condition, let $\xi\in\partial_{\bS^n} B$ be such that $\dist(y_0,\partial_{\bS^n} B) = \dist(y_0,\xi)$ and consider
a ball $V_\xi\subset \R^{n+1}\setminus \wt B'$ such that $\xi\in\partial V_\xi$ with $c\leq \rad(V_\xi)\leq 1/10$, so that $V_\xi$ is outer tangent to $\wt B'$ in $\xi$. Let $g_V$ be  the
Green function of the open set $\R^{n+1}\setminus \overline{V_\xi}$. Notice that the functions
$\wt u_B$ and $g_V(x_B,\cdot)$ are both harmonic in the domain $\wt B_0=\wt B'\setminus B(x_B,\frac1{10}\rad_{\bS^n}(B))$. It is easy to check that
$g_V(x_B,y)\approx 1$ for $y\in\partial \wt B_0 \cap \partial A(0,1/8,2)$ and $y\in\partial B(x_B,\frac1{10}\rad_{\bS^n}(B))$. So using also \rf{equb34} we deduce that 
$$\wt u_B(y) \lesssim g_V(x_B,y) \,\|\wt u_B\|_{\infty,\wt B'}  \lesssim g_V(x_B,y) 
\quad\mbox{ for all $y\in \partial \wt B_0$.}
$$
By the maximum principle, the same estimate holds for all $y\in \wt B_0$, and so
$$\wt u_B(y) \lesssim g_V(x_B,y) \approx |y-\xi| 
\quad\mbox{ for all $y\in \wt B_0$,}
$$
by the smoothness of $g_V$ in $\xi$. In particular, 
\begin{equation}\label{eq381}
\wt u_B(y_0)  \lesssim |y_0-\xi| = \dist(y_0,\partial_{\bS^n} B).
\end{equation}

Now we will show that
\begin{equation}\label{equb36}
	u_B(y) \gtrsim \dist(y,\partial_{\bS^n} B)\quad \mbox{ for all $y\in B$}.
\end{equation}
To this end, observe first that the condition \rf{eq381} and the fact that $\|u_B\|_{\infty,B}\approx1$ imply that the $\max_{B} u_B$ is attained at a point $x_0\in B$ such that $\dist(x_0,\partial_{\bS^n}B)\approx1$. By a Harnack chain argument, this implies that
\begin{equation}\label{equb37}
\wt u_B(y)\approx 1\quad \mbox{ for all $y\in \wt B'$ such that $\dist(y,\partial \wt B')\approx1$.}
\end{equation}
To prove \rf{equb36}, let $y_0$ and $\xi$ be as above, and consider 
a ball $U_\xi\subset \wt B'$ such that $\xi\in\partial U_\xi$ with $c\leq \rad(U_\xi)\leq 1/10$, so that $U_\xi$ is inner tangent to $\wt B'$ in $\xi$. 
Let $g_U$ be  the
Green function of the ball $U_\xi$ and let $x_U$ be the center of $U_\xi$. Notice that the functions
$\wt u_B$ and $g_U(x_U,\cdot)$ are both harmonic in $\wt B_1= U_\xi \setminus \bar B(x_U,\frac1{10}\rad(U_\xi))$. It is easy to check that, for $y\in\partial B(x_U,\frac1{10}\rad(U_\xi))$,
$g_U(x_U,y)\approx 1$   while $\wt u_B(y)\approx 1$, by
\rf{equb37}. So we deduce that 
$$\wt u_B(y) \gtrsim g_U(x_U,y)  
\quad\mbox{ for all $y\in \partial \wt B_1$.}
$$
By the maximum principle, the same estimate holds for all $y\in \wt B_1$, and so
$$\wt u_B(y) \gtrsim g_U(x_U,y) \approx |y-\xi| 
\quad\mbox{ for all $y\in \wt B_1$,}
$$
by the smoothness of $g_U$ in $\xi$.  In particular, $\wt u_B(y_0)  \gtrsim |y_0-\xi| = \dist(y_0,\partial_{\bS^n} B)$, which concludes the proof of \rf{equb36} and \rf{equb35}.

Let us turn our attention to the behavior of the harmonic measures
$\omega_{\wt B'}$ and $\omega_{\wt B}$. The fact that they are mutually absolutely continuous with $\HH^n|_{\partial \wt B'}$ and $\HH^n|_{\partial \wt B}$, respectively, is due to the fact that
 both $\wt B'$ and $\wt B$ are Lipschitz domains, and in fact, piecewise $C^\infty$ domains.
To prove \rf{equb32} for $x$ as in the lemma (i.e.,  $x\in A(0,1/2,3/4)\cap \wt\Omega$ such that $|x|^{-1}x\in \frac12 B$), we consider an arbitrary point $\xi\in \partial \wt B'$ and we assume that
$\partial \wt B'$ is smooth in a neighborhood of $\xi$ (this happens for $\HH^n$-a.e.\ $\xi\in\partial \wt B'$).  As above, let $V_\xi\subset \R^{n+1}\setminus \wt B'$ be a ball such that $\xi\in\partial V_\xi$ with $c\leq \rad(V_\xi)\leq 1/10$, so that $V_\xi$ is (outer) tangent to $\wt B'$ in $\xi$,
and  let $g_V$ be  the
Green function of $\R^{n+1}\setminus \overline{V_\xi}$. Also, let $g_{\wt B'}$ be the Green function
of $\wt B'$. Then the function $g_V(x,\cdot) - g_{\wt B'}(x,\cdot)$ is harmonic in $\wt B'$ and
it is non-negative in $\partial\wt B'$. So 
$$g_{\wt B'}(x,y)\le g_V(x,y)\quad\mbox{ for all $y\in \wt B'$,}$$
by the maximum principle. Consequently,
$$\frac{d\omega_{\wt B'}^x}{d\HH^n|_{\partial \wt B'}}(\xi) = \partial_\nu  g_{\wt B'}(x,\xi)
\leq \partial_\nu  g_{V}(x,\xi) = \frac{d\omega_{\R^{n+1}\setminus \overline{V_\xi}}^x}{d\HH^n|_{\partial V_\xi}}(\xi)\approx 1,$$
where $\partial_\nu$ stands for the normal derivative in the inner direction. 

The same argument as above shows that 
$$\frac{d\omega_{\wt B}^x}{d\HH^n|_{\partial \wt B}} \lesssim 1\quad
 \mbox{ in $\partial \wt B$.}$$
To prove the converse estimate for $\xi\in\partial \wt B\cap A(0,1/3,4/5)$, consider again
a ball $U_\xi\subset \wt B'$ such that $\xi\in\partial U_\xi$ with $c\leq \rad(U_\xi)\leq 1/10$, so that $U_\xi$ is (inner) tangent to $\wt B$ in $\xi$. 
Let $g_{\wt B}$ be the Green function of $\wt B$ and  $g_U$ be  the
Green function of $U_\xi$, and $x_U$ the center of $U_\xi$.
Then the function $g_{\wt B}(x_U,\cdot) - g_U(x_U,\cdot)$ is harmonic in $U_\xi$ and
it is non-negative in $\partial U_\xi$. So 
$$g_{\wt B}(x_U,y) \geq g_U(x_U,y)  \quad\mbox{ for all $y\in U_\xi$,}$$
by the maximum principle. Consequently,
$$\frac{d\omega_{\wt B}^{x_U}}{d\HH^n|_{\partial \wt B}}(\xi) = \partial_\nu  g_{\wt B}(x_U,\xi)
\geq \partial_\nu  g_{U}(x_U,\xi) = \frac{d\omega_{U_\xi}^{x_U}}{d\HH^n|_{\partial U_\xi}}(\xi)\approx 1.$$
By a Harnack chain argument, it follows that
$$\frac{d\omega_{\wt B}^{x_U}}{d\HH^n|_{\partial \wt B}}(\xi) \approx 
\frac{d\omega_{\wt B}^{x}}{d\HH^n|_{\partial \wt B}}(\xi),$$
and so the proof of \rf{equb33} is concluded.	
\end{proof}

\vv

%\begin{claim}
%It suffices to prove Theorem \ref{teomain1} under the assumption that $\wt\Omega'$ is Wiener regular. 
%\end{claim}

%We defer the proof of this statement till the end of the current subsection, and from now on we assume that $\wt\Omega'$ is
%Wiener regular. It is immediate to check that this implies that $\wt\Omega$ is Wiener regular too.

\begin{claim}[Extension of $u_B$ and $\wt u_B$]\label{claim2}
There exists an extension $u_{B}^{ex}$ of $u_B|_B$ to $\bS^n$ which belongs to $C^2(\bS^n)$ and such that $\|\nabla^j u_B^{ex}\|_{\infty,\bS^n}\lesssim 1$ for $j=0,1,2$ and $\|\Delta \wt u_B^{ex}\|_{\infty,A(0,1/4,2)}\lesssim 1$, where $\wt u_B^{ex}$ is the 
$\alpha_B$-homogeneous extension of $u_B^{ex}$ to $\R^{n+1}$. Further, we may construct $u_B^{ex}$ so that it is supported in an $\bS^n$-ball $B_0$ concentric with $B$ such that $\HH^n(\bS^n\setminus B_0)\geq \beta/2$.
\end{claim}

\begin{proof}
The arguments are quite standard.
Since the function $\wt u_B$ is harmonic in $\wt B'$  and it vanishes identically in $\partial\wt B'\setminus \partial A(0,1/8,2)$, which is 
a $C^\infty$ portion of the boundary, it follows that 
$\wt u_B\in C^2(\overline{V})$, for $V=\wt B'\cap A(0,1/4,3/2)$  and that
$\|\wt u_B\|_{C^2(\overline{V})}\lesssim \|\wt u_B\|_{C(\overline{\wt B'})}$. See for example Corollary 6.7 from \cite{Gilbarg-Trudinger}. 
By Lemma \ref{lemubom}, $\|\wt u_B\|_{C(\overline{\wt B'})}\lesssim1$. 
By a suitable reflection, one can construct a function $f\in C^2(\bar A(0,1/2,5/4))$ which coincides with $\wt u_B$ in $V$
and satisfies $\|f\|_{C^2(\bar A(0,1/2,5/4))}\lesssim\|\wt u_B\|_{C^2(\overline{V})}$.
Multiplying $f$ by a suitable bump $C^\infty$ function, we can assume that $f|_{\bS^n}$ is supported in an $\bS^n$-ball $B_0$ concentric with $B$ such that $\HH^n(\bS^n\setminus B_0)\geq \beta/2$. Then we set $u_B^{ex}:= f|_{\bS^n}$ and we let $\wt u_B^{ex}$ be its
$\alpha_B$-homogeneous extension.
\end{proof}
\vv

%Above, $\Delta \wt u_B^{ex}$ should be understood in the Sobolev sense.
Recall that we assume that $u_B$ and $u_\Omega$ vanish respectively in $\bS^n\setminus B$ and $\bS^n\setminus \Omega$.
Notice that the estimate \rf{eqAKN} from Theorem \ref{teoAKN}
may not hold for the extension $u_B^{ex}$. Instead, we have
\begin{equation}\label{eqAKN2}
\lambda_\Omega - \lambda_B \geq C(\beta)\,\Big(\HH^n(\Omega \triangle B)^2 + \int_{B} |u_\Omega - u_B^{ex}|^2\,d\HH^n\Big).
\end{equation}

Consider the function $v:\bS^n\to\R$ defined by $v= u_B^{ex} - u_\Omega$ and let $\wt v$ be its $\alpha_B$-homogeneous extension, so that
$\wt v= \wt u_B^{ex} - \wt u_\Omega$. Notice that in $\wt \Omega'$, by \rf{eqlaplace3} we have
\begin{equation}\label{eqlapla}
\Delta \wt v = \Delta \wt u_B^{ex} - \Delta \wt u_\Omega = \Delta \wt u_B^{ex} - (\lambda_B-\lambda_\Omega)\,|y|^{-2}\,\wt u_\Omega.
\end{equation}
Notice also that $\supp \Delta\wt u_B^{ex}\subset  (\wt B')^c$. 

Recall that $S=\partial A(0,1/4,1)$. We can write, for $x\in\wt \Omega$,
\begin{align}\label{eqv1234}
\wt v(x) & = \int_{\partial\wt \Omega} \wt v\,d\omega_{\wt\Omega}^x - \int_{\wt\Omega} \Delta\wt v(y)\,g_{\wt\Omega}(x,y)\,dy\\
& = \int_{\partial\wt \Omega \cap \wt B\setminus S} \wt v\,d\omega_{\wt\Omega}^x 
+ \int_{\partial\wt \Omega \setminus  (\wt B\cup S)} \wt v\,d\omega_{\wt\Omega}^x 
+ \int_{\partial\wt \Omega\cap S} \wt v\,d\omega_{\wt\Omega}^x
- \int_{\wt\Omega} \Delta\wt v(y)\,g_{\wt\Omega}(x,y)\,dy\nonumber\\
& =: \wt v_1(x) + \wt v_2(x) + \wt v_3(x) + \wt v_4(x),\nonumber
\end{align}
where $g_{\wt\Omega}(x,y)$ is the Green function of $\wt\Omega$.
For the points $x\in A(0,1/2,3/4)\cap \wt\Omega$ such that $|x|^{-1}x\in \frac9{10} B$, we intend to estimate
$\wt v(x)$ from below and to relate this to the terms appearing on the right hand side of the inequalities
\rf{eqmain00}, \rf{eqmain01}, and \rf{eqmain1}, so that then we can apply the inequality \rf{eqAKN2} to prove Theorem \ref{teomain0} and Theorem~\ref{teomain1}.
To this end, we will estimate $|\wt v_2(x) +\wt v_3(x) +\wt v_4(x)|$ from above and 
$\wt v_1(x)$ from below.
\vv

\begin{lemma}\label{lemv1234}
Let $x\in A(0,1/2,3/4)\cap \wt\Omega$ and also that $|x|^{-1}x\in \frac9{10} B$, and let $\wt v_i$, for $1\leq i \leq 4$, be the functions in \rf{eqv1234}. Suppose that $|\lambda_B-\lambda_\Omega|\leq 1$.
Then we have
$$\big|\wt v_2(x) + \wt v_3(x) + \wt v_4(x)\big| \lesssim (\lambda_\Omega - \lambda_B)^{1/2}.$$
\end{lemma}

\begin{proof}
{\bf Estimate of $|\wt v_3(x)|$}. 
 By the maximum principle and taking into account that $\frac{d\omega_{A(0,1/4,1)}^x}{d\HH^n}\approx 1$ on $S$,  for any set 
$F\subset \partial\wt \Omega\cap S$, we have
$$\omega_{\wt \Omega}^x (F) \leq \omega_{A(0,1/4,1)}^x (F) \approx \HH^n(F).$$
Therefore, $\omega_{\wt \Omega}^x|_S$ is absolutely continuous with respect to surface measure on $S$, and
$$\frac{d\omega_{\wt \Omega}^x}{d\HH^n}(y)\lesssim 1\quad \mbox{ for $\HH^n$-a.e.\ $y\in \partial\wt \Omega\cap S$.}$$
Hence, using also the $\alpha_B$-homogeneity of $\wt v$,
$$|\wt v_3(x)| \lesssim \int_{\partial\wt \Omega\cap S} |\wt v|\,d\HH^n \lesssim \int_{\Omega} |u_B^{ex}-u_\Omega|\,d\HH^n.$$
Then, by Cauchy-Schwarz and \rf{eqAKN2}, we deduce
$$
\int_{\Omega\cap B} |u_B^{ex}-u_\Omega|\,d\HH^n\leq \bigg(\int_{\Omega\cap B}  |u_B-u_\Omega|^2\,d\HH^n\bigg)^{1/2} \lesssim (\lambda_\Omega - \lambda_B)^{1/2}.$$
On the other hand, from \rf{equb30}, \rf{equb31}, and again \rf{eqAKN2}, we have
$$\int_{\Omega\setminus B} |u_B^{ex}-u_\Omega|\,d\HH^n\leq \big(\|u_B^{ex}\|_\infty + \|u_\Omega\|_\infty\big)\,\HH^n(\Omega\setminus B)\lesssim
\HH^n(\Omega\setminus B)\lesssim (\lambda_\Omega - \lambda_B)^{1/2}.$$
Adding the two preceding estimates we derive
$$|\wt v_3(x)| \lesssim (\lambda_\Omega - \lambda_B)^{1/2}.$$

{\bf Estimate of $|\wt v_4(x)|$}. To this end, we use the trivial estimate 
$g_{\wt\Omega}(x,y)\leq |x-y|^{1-n}$  and the fact that by 
\rf{eqlaplace3}
%\rf{eqlapla} 
and Claim \ref{claim2},
$$\|\Delta \wt u_{\Omega}\|_{\infty,\wt \Omega}\lesssim |\lambda_\Omega-\lambda_B|\leq 1 \quad \text{ and }\quad
\|\Delta \wt u_B^{ex}\|_{\infty,\wt\Omega\setminus B}\lesssim 1.$$ 
Taking also into account that $\supp\Delta \wt u_B^{ex}\subset (\wt B)^c$, we get
\begin{align*}
|\wt v_4(x)| &\leq 
\int_{\wt\Omega \cap\wt B} |\lambda_B-\lambda_\Omega|\,|\wt u_\Omega(y)|\,\frac1{|y|^2\,|x-y|^{n-1}}
\,dy + 
\int_{\wt\Omega \setminus \wt B} \frac{|\Delta \wt u_B^{ex}(y)| + |\Delta \wt u_{\wt\Omega}(y)|}{|x-y|^{n-1}}
\,dy \\
& \lesssim |\lambda_B-\lambda_\Omega|\,\|\wt u_\Omega\|_\infty
\int_{\wt\Omega \cap\wt B} \frac1{|x-y|^{n-1}}\,dy + 
\int_{\wt\Omega \setminus \wt B} \,\frac1{|x-y|^{n-1}}
\,dy.
\end{align*}
To estimate the first summand we use the fact that $\|\wt u_\Omega\|_\infty\lesssim1$ and a trivial bound for the integral, while for second one we take into account that $\HH^{n+1}(\wt\Omega \setminus \wt B)
\lesssim \HH^{n}(\Omega \setminus B)\lesssim |\lambda_B-\lambda_\Omega|^{1/2}$ and the fact that 
$\dist(x,\wt B^c) \gtrsim 1$, since $|x|^{-1}x\in \frac9{10} B$. Using also 
the assumption that $|\lambda_B-\lambda_\Omega|\leq 1$, we obtain 
$$|\wt v_4(x)|\lesssim |\lambda_B-\lambda_\Omega| + |\lambda_B-\lambda_\Omega|^{1/2}\lesssim |\lambda_B-\lambda_\Omega|^{1/2}.$$

{\bf Estimate of $|\wt v_2(x)|$.} 
Recall that
$$ \wt v_2(x)= \int_{\partial\wt \Omega \setminus  (\wt B\cup S)} \wt v\,d\omega_{\wt\Omega}^x
= \int_{\partial\wt \Omega \setminus  (\wt B\cup S)}( 
\wt u_B^{ex} - \wt u_\Omega)\,\,d\omega_{\wt\Omega}^x
.$$ 
Since $u_\Omega$ vanishes on $\poms$, we can write
$$|\wt v_2(x)| \leq \int_0^\infty \omega_{\wt\Omega}^x\big(\big\{y\in \partial\wt \Omega \setminus  (\wt B\cup S): |\wt u_B^{ex}(y)|>t\big\}
\big)\,dt.$$
Since $\|\nabla \wt u_B^{ex}\|_{\infty,A(0,1/4,1)}\approx 
\|\nabla u_B^{ex}\|_{\infty,\bS^n}+\|u_B^{ex}\|_{\infty,\bS^n}\lesssim 1$ and $\wt u_B^{ex}$ vanishes on $\partial\wt B\setminus S$, we have 
$$|\wt u_B^{ex}(y)|\lesssim \dist(y,\partial\wt B\setminus S) \approx \dist_{\bS^n}(|y|^{-1}y,\partial B)
= \delta_B(|y|^{-1}y).$$
Therefore, recalling also that $\wt u_B^{ex}$ vanishes in $\bS^n\setminus B_0$ (with $B_0$ defined in Claim
\ref{claim2}), 
for some $c>0$ we have
\begin{align*}
|\wt v_2(x)| & \leq \int_0^\infty \omega_{\wt\Omega}^x\big(\big\{y\in \partial\wt \Omega \cap \wt B_0\setminus  (\wt B\cup S):\delta_B(|y|^{-1}y)\geq ct\big\}
\big)\,dt \\
& \approx 
\int_0^{t_0} \omega_{\wt\Omega}^x\big(\big\{y\in \partial\wt \Omega \setminus  (\wt B\cup S):\delta_B(|y|^{-1}y)\geq t\big\}
\big)\,dt,
\end{align*}
where $t_0= \dist_{\bS^n}(\partial B,\partial B_0)$.
For each $t\in(0,t_0)$, take the $\bS^n$-ball $B_t=\{z\in\bS^n:\dist_{\bS^n}(z,B)< t\}$ and the associated truncated conical regions
$$\wt B_t = \{y\in A(0,1/4,1):|y|^{-1}y\in B_t\},\qquad \wt B'_t = \{y\in A(0,1/8,2):|y|^{-1}y\in B_t\}.$$
Notice that $B\subset B_t\subset B_0$ and that
$$\big\{y\in \partial\wt \Omega \setminus  (\wt B\cup S):\dist_{\bS^n}(|y|^{-1}y,\partial B)\geq t\big\}
= \partial\wt \Omega \setminus  (\wt B_t\cup S).$$
By the maximum principle, since $\wt\Omega\cap \wt B_t\subset \wt\Omega$ and $\wt\Omega\cap\partial(\wt\Omega\cap \wt B_t) = \wt\Omega\cap \partial\wt B_t$, for $0<t\leq t_0$
we have
$$\omega_{\wt\Omega}^x\big(\partial\wt \Omega \setminus  (\wt B_t\cup S)\big) \leq 
\omega_{\wt \Omega \cap  \wt B_t}^x\big(\partial \wt B_t \cap \wt \Omega \big).$$
By the maximum principle and standard estimates (taking into account that the ball $B_0$ in Claim~\ref{claim2} does not degenerate)
we deduce that, for any set $F\subset \partial \wt B_t \cap \wt \Omega  \setminus S$, 
$$\omega_{\wt \Omega \cap  \wt B_t}^x(F)\leq \omega_{\wt B_t'}^x(F) \approx \HH^n(F).$$
Therefore, for each $t\in (0,t_0)$,
$$\omega_{\wt\Omega}^x\big(\big\{y\in \partial\wt \Omega \setminus  (\wt B\cup S):\delta_B(|y|^{-1}y)\geq t\big\}
\big) \leq \omega_{\wt \Omega \cap  \wt B_t}^x\big(\partial \wt B_t \cap \wt \Omega \big) \lesssim 
\HH^n\big(\partial \wt B_t \cap \wt \Omega\big).$$
Consequently, applying Fubini and using the conical property of $\wt\Omega$, we obtain
$$|\wt v_2(x)| \lesssim \int_0^{t_0} \HH^n\big(\partial \wt B_t \cap \wt \Omega \big)\,dt \approx \HH^{n+1}(\wt \Omega\cap \wt B_0\setminus
\wt B) \lesssim \HH^n(\Omega\setminus B) \lesssim |\lambda_B-\lambda_\Omega|^{1/2},$$
using also \rf{eqAKN2} for the last estimate.
\end{proof}
\vv

{\bf Estimate of $\wt v_1(x)$}. Next we turn our attention to the function $\wt v_1$.
Since $\wt u_\Omega$ vanishes in $\partial\wt\Omega\setminus S$ and $\wt u_B^{ex} = \wt u_B$ in $\wt B$, for $x\in A(0,1/2,3/4)\cap \wt\Omega$ we have
$$\wt v_1(x) = \int_{\partial\wt \Omega \cap (\wt B\setminus S)} \big(\wt u_B^{ex} - \wt u_\Omega)\,d\omega_{\wt\Omega}^x =
\int_{\partial\wt \Omega \setminus S} \wt u_B\,d\omega_{\wt\Omega}^x.$$
Observe first that, by the maximum principle, if we let $\Omega_0=\Omega\cap B$ and $\wt\Omega_0=\wt\Omega\cap\wt B$, 
then it holds
\begin{equation}\label{eqvf0}
\wt v_1(x) \geq \int_{\partial\wt \Omega_0  \setminus S} \wt u_B\,d\omega_{\wt\Omega_0}^x =: f_0(x)\quad \mbox{ for all $x\in \wt \Omega_0$}.
\end{equation}
We extend $f_0$ to the whole $\wt B$ by letting
\begin{equation}\label{eqvf0'}
f_0(x) = \wt u_B(x)\quad \mbox{ for all $x\in \wt B\setminus\wt \Omega_0$}.
\end{equation}
We claim that $f_0$ is superharmonic in $\wt B_0$. Indeed, first notice that, by the maximum principle, 
$$f_0(x) \leq \wt u_B(x)\quad \mbox{ for all $x\in \wt \Omega_0$}.$$
So $f_0-\wt u_B$ is continuous in $\wt B$ (although it may happen that $f_0$ does not extend continuously to $\partial\wt B\cap S$), it vanishes in $\wt B\setminus\wt\Omega$, and it is harmonic and negative in $\wt\Omega\cap\wt B=\wt \Omega_0$. 
This implies that $f_0-\wt u_B$ is superharmonic in $\wt B$, and so the same happens for $f_0 = (f_0-\wt u_B) + \wt u_B$. 
\vv

\begin{lemma}\label{lemhomo}
Suppose that $|\lambda_B-\lambda_\Omega|\leq 1$ and that $\wt\Omega$ is Wiener regular. 
Denote $\wt {\frac9{10} B} = \{x\in \wt B: |x|^{-1}x\in  \frac9{10} B\}.$
Then we have
$$\lambda_\Omega - \lambda_B \geq C(\beta) \left(\int_{\wt { \frac9{10} B}\cap A(0,1/2,3/4)} f_0(x)\,d\HH^{n+1}(x)\right)^2.$$
\end{lemma}

\begin{proof}
By the Allen-Kriventsov-Neumayer theorem and the $\alpha_B$-homogeneity of $\wt v$, we have
 $$\lambda_\Omega - \lambda_B \gtrsim \int_{ \frac9{10}B} |u_B- u_\Omega|^2\,d\HH^n 
 \approx_\beta \int_{\wt{ \frac9{10}B}} |\wt u_B - \wt u_\Omega|^2\,d\HH^{n+1},$$
 since $\alpha_B\approx_\beta 1$. 
By Lemma \ref{lemv1234} and \rf{eqvf0}, for all $x\in \wt { \frac9{10} B}\cap A(0,1/2,3/4)\cap \wt\Omega$, we have
\begin{align*}
\wt u_B(x) - \wt u_\Omega(x) & = \wt u_B^{ex}(x) - \wt u_\Omega(x) = \wt v(x) \\
& \geq \wt v_1(x) - |\wt v_2(x) + \wt v_3(x) + \wt v_4(x)|\geq 
f_0(x) - C(\lambda_\Omega - \lambda_B)^{1/2}.
\end{align*}
On the other hand, for  $x\in \wt { \frac9{10} B}\cap A(0,1/2,3/4)\setminus \wt\Omega$, by definition
$$\wt u_B(x) - \wt u_\Omega(x) = \wt u_B(x) = f_0(x).$$
Therefore, by Cauchy-Schwarz,
\begin{align*}
 \lambda_\Omega - \lambda_B & \gtrsim_\beta \int_{\wt{ \frac9{10}B}\cap A(0,1/2,3/4)} |f_0(x) - C(\lambda_\Omega - \lambda_B)^{1/2}|^2\,d\HH^{n+1}(x) \\
& \gtrsim_\beta \left(\int_{\wt{ \frac9{10}B}\cap A(0,1/2,3/4)} |f_0(x) - C(\lambda_\Omega - \lambda_B)^{1/2}|\,d\HH^{n+1}(x)\right)^2 \\
& \geq \left(\int_{\wt{ \frac9{10}B}\cap A(0,1/2,3/4)} f_0(x)\,d\HH^{n+1}(x)\right)^2   - C|\lambda_\Omega - \lambda_B|,
 \end{align*}
which proves the lemma.
\end{proof}
\vv

%Estimating $\wt v_1(x)$ from below is the most difficult part of the argument for the proof of Theorem~\ref{teomain1}.

\begin{lemma}\label{lemcapomega}
There exists some constant $c_2\in (0,1/4)$ depending only on $n$ such that the following holds.
Let $\Delta_z$ be an $\bS^n$-ball centered in $z\in\bS^n$ such that  $\Delta_z\subset B$, with
$\dist_{\bS^n}(\Delta_z,\partial_{\bS^n}B)\geq \rho$. Let $B_z$ be an Euclidean ball centered in the segment 
$L_z = [0,z] \cap A(0,1/2,3/4)$ with $\frac18\rad_{\bS^n}(\Delta_z)\leq \rad(B_z) \leq \rad_{\bS^n}(\Delta_z)$ such that $S\cap B_z\neq\varnothing$.
Then, in the case $n\geq3$ we have 
\begin{equation}\label{eqlem34a}
f_0(y) \gtrsim \frac{\rho\,\capp_{n-2}(c_2 \Delta_z\!\setminus  \Omega)}{\rad_{\bS^n}(\Delta_z)^{n-2}}  \quad \text{ for all $y\in \tfrac14B_z\cap\wt\Omega$,}
\end{equation}
and in the case $n=2$,
\begin{equation}\label{eqlem34b}
f_0(y) \gtrsim \frac{\rho}{\log\dfrac{\rad_{\bS^n}(\Delta_z)}{\capp_L(c_2\,\Delta_z\!\setminus\Omega)}} \quad \text{ for all $y\in c_2 B_z\cap\wt\Omega$.}
\end{equation}
\end{lemma}

\begin{proof}
Suppose $n\geq3$.
By Lemma \ref{lem2.1} and Lemma \ref{lemmaprod1}, for all $y\in \frac14B_z\cap \wt\Omega_0$, we have 
$$\omega_{\wt\Omega_0}^y(B_z) \gtrsim \frac{\capp_{n-1}(\tfrac 14 B_z\setminus \wt \Omega_0)}{\rad(B_z)^{n-1}} \gtrsim \frac{\capp_{n-2}(c_2\Delta_z\setminus \Omega_0)}{\rad(B_z)^{n-2}},$$
for a suitable constant $c_2$ depending just on $n$.
Then, taking into account that 
$\wt u_B(z)\approx \dist(z,\partial\wt B)\gtrsim \rho$ for all $z\in B_z$ and that $B_z\cap S=\varnothing$ (recall that $S=\partial A(0,1/4,1)$), 
from the definition of $f_0$ in \rf{eqvf0} we infer
$$f_0(y) \geq\int_{\partial\wt \Omega_0  \cap B_z} \wt u_B\,d\omega_{\wt\Omega_0}^y\gtrsim \rho\,\omega_{\wt\Omega_0}^y(B_z) \gtrsim \frac{\rho\,\capp_{n-2}(c_2\Delta_z\setminus \Omega_0)}{\rad(B_z)^{n-2}}\quad \mbox{ for all $y\in\frac14B_z\cap \wt\Omega_0$.}$$
On the other hand,  for $y\in \tfrac14B_z\setminus \wt\Omega_0$, we also have $f_0(y) = \wt u_B(y) \approx \dist(z,\wt B)
\geq \rho$. 

The arguments for the case $n=2$ are analogous.
\end{proof}
\vv

\begin{proof}[\bf Proof of Theorem \ref{teomain0} assuming $\wt\Omega$ to be Wiener regular, for $\bM^n=\bS^n$]
Suppose $n\geq3$. We can assume that $|\lambda_B-\lambda_\Omega|\leq1$, because otherwise the inequality
\rf{eqmain00} is immediate, since the right hand side of \rf{eqmain00} is bounded above by some absolute constant.
Recall that we denote $r_B=\rad_{\bS^n}(B)$. We will prove the theorem with $a=c_2/4$, with $c_2$ as in \rf{eqlem34a}
and \rf{eqlem34b}.

We will show first that
\begin{equation}\label{eqcaaa1}
\lambda_\Omega - \lambda_B \geq C(\beta)\,
\bigg(\sup_{s \in (0,c_2/4)} \avint_{\partial_{\bS^n} (s B)} \capp_{n-2}(B_{\bS^n}(x,\tfrac{c_2}4r_B)\setminus \Omega)\,d\HH^{n-1}(y)\bigg)^2.
\end{equation}
To this end, observe that, from \rf{eqlem34a} in Lemma \ref{lemcapomega} applied to $\Delta_z= \frac12\,B$ and to a ball $B_z$ as in that lemma (so that $\rad(B_z) \approx \rad_{\bS^n}(\Delta_z) \approx\rad_{\bS^n}(B)$), we get 
$$f_0(y) \gtrsim \frac{\rad_{\bS^n}(B)\,\capp_{n-2}(\tfrac{c_2}2 B\setminus  \Omega)}{\rad_{\bS^n}(B)^{n-2}}  =
\frac{\capp_{n-2}(\tfrac{c_2}2 B\setminus  \Omega)}{\rad_{\bS^n}(B)^{n-3}}
 \quad \text{ for all $y\in \tfrac14B_z\cap\wt\Omega$.}$$
The same inequality holds for $y\in \tfrac14B_z\setminus\wt\Omega$, since for these points it holds 
$f_0(y) = \wt u_B(y) \approx \rad_{\bS^n}(B)$, by \rf{equb30}.
Consequently, by Lemma \ref{lemhomo}, choosing $B_z$ so that 
$\HH^{n+1}(\wt {\tfrac9{10} B}\cap A(0,1/2,3/4) \cap B_z)\approx \HH^{n+1}(\wt B)\approx r_B^{n+1}\approx_\beta1,$
we deduce
\begin{equation}\label{eqcaaa2}
\lambda_\Omega - \lambda_B \geq C(\beta) \left(\int_{\wt { \frac9{10} B}\cap A(0,1/2,3/4) \cap B_z} f_0(x)\,d\HH^{n+1}(x)\right)^2 \gtrsim \left(r_B^{\;4}\capp_{n-2}(\tfrac{c_2}2 B\setminus  \Omega)
\right)^2.
\end{equation}
Notice that, for any $ s \in (0,c_2/4)$ and $y\in \partial_{\bS^n} ( s B)$, it holds $\capp_{n-2}(\tfrac{c_2}2 B\setminus  \Omega)
\geq \capp_{n-2}(B_{\bS^n}(y,\tfrac{c_2}4r_B)\setminus \Omega)$. Thus,
$$\capp_{n-2}(\tfrac{c_2}2 B\setminus  \Omega) \geq
\avint_{\partial_{\bS^n} (s B)} \capp_{n-2}(B_{\bS^n}(y,\tfrac{c_2}4r_B)\setminus \Omega)\,d\HH^{n-1}(y).
$$
Together with \rf{eqcaaa2}, this gives \rf{eqcaaa1}.

We consider now the case $s = 1-t\in (c_2/4,1)$ (here $t$ is the parameter appearing in \rf{eqmain00} and \rf{eqmain01}).
Let $B_s=sB$ and
$$\wt B_{s} = \{y\in A(0,1/4,1):|y|^{-1}y\in sB\}$$
and denote 
$$\Gamma_{s} = \partial \wt B_{s} \cap A(0,1/4,1).$$
Notice that $\wt B_s$ is a Lipschitz domain with Lipschitz character depending at most on $n$ and $\beta$,
since $s\geq c_2/4$.

Let $\vphi_0$ be a smooth bump function which equals $1$ in $A(0,1/3,4/5)$ is supported on $A(0,1/4,1)$.
We define $f_1:\wt B_s\to \R$ by
\begin{equation}\label{eqdeff1}
f_1(x) = \int_{\partial\wt B_s  \setminus S} \vphi_0\,f_0\,d\omega_{\wt B_s}^x = \int_{\Gamma_s} f_0\,d\omega_{\wt B_s}^x.
\end{equation}
Since the boundary data $\vphi_0f_0$ is continuous in $\partial\wt B_s$,
the function
$f_1$ is harmonic in $\wt B_s$ and continuous in $\overline{\wt B_s }$. These are the advantages of $f_1$ over $f_0$.
Moreover, since  $f_0$ is superharmonic in $\wt B_s$ and $f_1$ harmonic in $\wt B_s$, $f_0\geq f_1$ in $\Gamma_s$, and 
$$\liminf_{y\to z} (f_0-f_1)(z) = \liminf_{y\to z} f_0(z) \geq 0\quad \mbox{ for all $z\in \partial\wt B_s\setminus\Gamma_s$},$$
by the maximum principle then it follows that
$f_0(x) \geq f_1(x)$ for all $x\in \wt B_s$ (see Theorem 3.1.5 from \cite{AG}, for example).

By the same arguments as the ones for \rf{equb33}, we deduce that for $x\in A(0,1/2,3/4) \cap \wt B_{s/2}$ 
$$\frac{d\omega_{\wt B_s}^x}{d\HH^n|_{\partial \wt B_s}} \approx 1\quad \mbox{ in $\partial \wt B_s\cap A(0,1/3,4/5)$.}$$
Therefore, by \rf{eqdeff1}, for these points $x$,
$$f_1(x) \gtrsim 
\int_{\partial \wt B_s\cap A(0,1/3,4/5)} f_0(y)\,d\HH^{n}(y).$$
On the other hand, for every $y\in\partial \wt B_s\cap A(0,1/3,4/5)$,
by Lemma \ref{lemcapomega} applied to an $\bS^n$-ball $\Delta_z$ centered in $z = |y|^{-1}y$ with
geodesic radius $\frac12\delta_B(z) = \frac{1-s}2\,r_B = \rho$ , we have
$$f_0(y) \gtrsim \frac{\rho\,\capp_{n-2}(c_2 \Delta_z\!\setminus  \Omega)}{\rad_{\bS^n}(\Delta_z)^{n-2}}
\approx \frac{\capp_{n-2}(c_2 \Delta_z\!\setminus  \Omega)}{((1-s)r_B)^{n-3}} = \frac{\capp_{n-2}(B_{\bS^n}(|y|^{-1}y,c_2 (1-s)r_B)\setminus  \Omega)}{((1-s)r_B)^{n-3}}
.$$
Thus,
\begin{align*}
\int_{\partial \wt B_s\cap A(0,1/3,4/5)} f_0(y)\,d\HH^{n}(y) & \gtrsim \int_{\partial \wt B_s\cap A(0,1/3,4/5)}
\frac{\capp_{n-2}(B_{\bS^n}(|y|^{-1}y,c_2 (1-s)r_B)\setminus  \Omega)}{((1-s)r_B)^{n-3}}\,d\HH^{n}(y)\\
& \approx 
\int_{\partial_{\bS^n} (sB)}
\frac{\capp_{n-2}(B_{\bS^n}(y,c_2 (1-s)r_B)\setminus  \Omega)}{((1-s)r_B)^{n-3}}\,d\HH^{n-1}(y)
.
\end{align*}
So, for all $x\in A(0,1/2,3/4) \cap \wt B_{s/2}$,
$$f_0(x) \geq f_1(x)\gtrsim \int_{\partial_{\bS^n} (sB)}
\frac{\capp_{n-2}(B_{\bS^n}(y,c_2 (1-s)r_B)\setminus  \Omega)}{((1-s)r_B)^{n-3}}\,d\HH^{n-1}(y).
$$
By Lemma \ref{lemhomo}, this implies
$$\lambda_\Omega - \lambda_B \gtrsim C(\beta) \left(\int_{\partial_{\bS^n} (sB)}
\frac{\capp_{n-2}(B_{\bS^n}(y,c_2 (1-s)r_B)\setminus  \Omega)}{((1-s)r_B)^{n-3}}\,d\HH^{n-1}(y)\right)^2,$$
for all $s\geq c_2/4$.
Together with \rf{eqcaaa1}, this proves the theorem with $a=c_2/4$ and $n\geq 3$. 
The arguments for the case $n=2$ are almost the same. The only essential change is that above we have to use the estimate \rf{eqlem34b} instead of \rf{eqlem34a}.
Thanks to Lemmas \ref{lemcompar} and
\ref{lemcompar'}, this implies that the theorem is valid for any $a\in(0,1)$.
\end{proof}

\vv

% ***************************************************************************

\section{End of the proof of Theorem \ref{teomain1} in the Wiener regular case for
$\bM^n=\bS^n$}\label{secmainlemma}

Notice that the case $s=n$ in Theorem \ref{teomain1} is an immediate consequence of Theorem \ref{teoAKN}. So we assume that $0<s<n$.
Recall that, for $n\geq 3$,
$$
T_{c_0}(\Omega,B,a) = \big\{x\in B\setminus \Omega:\capp_{n-2}(B_{\bM^n}(x,a\,\delta_B(x))\setminus \Omega)\geq c_0\,\delta_B(x)^{n-2}\big\},
$$
and, in the case $n=2$, 
$$
 T_{c_0}(\Omega,B,a) = \big\{x\in B\setminus \Omega:\capp_L(B_{\bM^n}(x,a\,\delta_B(x))\setminus \Omega)\geq c_0\,\delta_B(x)\big\}.
$$
To shorten notation we write
$$T_{c_0} = T_{c_0}(\Omega,B,a).$$
We will assume $a=1/2$, but all the arguments below work with arbitrary $a\in (0,1)$. We denote
$$\gamma_s(\Omega) = \int_{T_{c_0}} \delta_B(y)^{n-s}\,d\HH^{s}_\infty(y).$$
So to prove Theorem \ref{teomain1} we have to show that, for $0<s\leq n$, 
\begin{equation}\label{eqgamd8}
\lambda_\Omega - \lambda_B \geq C(a,s,\beta,c_0)\,\gamma_s(\Omega)^2.
\end{equation}
To this end, we distinguish two cases:
\begin{itemize}
\item {\bf Case 1.} There exists some point $y_0\in T_{c_0}$ such that
$\delta_B(y_0) \geq \frac1{4} \,\rad_{\bS^n}(B)$.
\item {\bf Case 2.} Such point $y_0$ does not exist.
\end{itemize}

% *****************************************************************************************
\vv
\subsection{Case 1.}
Let $y_0\in T_{c_0}$ be as described above.
 By covering 
$B_{\bS^n}(y_0,\frac12\delta_B(y_0))$ with a finite number of $\bS^n$-balls with radius $\frac1{2c_2}\delta_B(y_0)$, we find an $\bS^n$-ball $\Delta_z = B_{\bS^n}(z,\frac1{2c_2}\delta_B(y_0))$
 such that $\capp_{n-2}(\Delta_z\setminus\Omega)\approx \rad_{\bS^n}(\Delta_z)^{n-2}$ in the case $n\geq3$ and $\capp_L(\Delta_z\setminus\Omega)\approx \rad_{\bS^n}(\Delta_z)^{n-2}$ in the case $n=2$.
Then, by Lemma \ref{lemcapomega}, in the case $n\geq 3$ the function $f_0$ (defined in \rf{eqvf0} and \rf{eqvf0'})
satisfies
$$f_0(y) \gtrsim \frac{\rad_{\bS^n}(\Delta_z)\,\capp_{n-2}(c_2 \Delta_z\!\setminus  \Omega)}{\rad_{\bS^n}(\Delta_z)^{n-2}} \gtrsim c_0 \quad \text{ for all $y\in \tfrac14B_z\cap\wt\Omega$,}$$
for an Euclidean ball $B_z\subset \wt {\frac9{10}B}$ centered in the segment 
$L_z = [0,z] \cap A(0,1/2,3/4)$ with $\approx \rad(B_z)\approx \rad_{\bS^n}(\Delta_z)\approx 1$.
The same estimate holds in the case $n=2$, using that $\capp_L(\Delta_z\setminus\Omega)\approx \rad_{\bS^n}(\Delta_z)$.
Then, by Lemma \ref{lemhomo}, we deduce that either $|\lambda_B-\lambda_\Omega|> 1$ or 
$$\lambda_\Omega - \lambda_B \geq C(\beta) \left(\int_{\wt { \frac9{10} B}\cap A(0,1/2,3/4)} f_0(x)\,d\HH^{n+1}(x)\right)^2 
\geq C'(\beta) \left(\int_{B_z} c_0\,d\HH^{n+1}(x)\right)^2 \gtrsim_\beta c_0^2,$$
which implies \rf{eqgamd8} in any case, as $\gamma_s(\Omega)\lesssim1$.

% *****************************************************************************************
\vv
\subsection{Case 2.}

From now on we will assume that $x\in A(0,1/2,3/4)\cap \wt\Omega$ and also that $|x|^{-1}x\in \frac12 B$. The two conditions and the fact we are in Case 2 ensure that 
$$\dist(x,\R^{n+1}\setminus\wt B) \gtrsim 1\quad \text{ and }\quad \dist(x,\wt T_{c_0}) \gtrsim 1,$$
where 
$$\wt T_{c_0} = \{y\in A(0,1/4,1):|y|^{-1}y\in T_{c_0}\}.$$

\begin{mlemma}\label{mainlemma1}
Under the assumptions of Case 2,
for all
 $x\in A(0,1/2,3/4)\cap \wt\Omega$ such that $|x|^{-1}x\in \frac12 B$, we have
$$f_0(x)\gtrsim \gamma_s(\Omega).$$
\end{mlemma}

%By the definition of $f_0$, it also holds $f_0(x)\gtrsim 1\gtrsim c_0\,\gamma_s(\Omega).$ 
By 
Lemma \ref{lemhomo} and the preceding result, then we deduce that \rf{eqgamd8} also holds in Case 2 and Theorem \ref{teomain1}
follows.
So to conclude the proof of Theorem \ref{teomain1} it just remains to prove Main Lemma~\ref{mainlemma1}.

\vv
Let $M\geq 10$ be some constant to be fixed below. For each $y\in T_{c_0}$, let $y'\in\partial B$ be such that $\delta_{B}(y) = \dist_{\bS^n}(y,y')$ and let $\Delta_{y}$ be an $\bS^n$-ball
 centered in $y'$ with $\bS^n$-radius $2M\delta_{B}(y)$ (in case that $2M\delta_{B}(y)>\pi$, then $\Delta_{y}=\bS^n$).
By Vitali's covering theorem, there exists a subfamily 
$\{\Delta_i'\}_{i\in I_0} \subset \{\Delta_y\}_{y\in T_{c_0}}$ such that $T_{c_0}\subset \bigcup_{i\in I_0}5\Delta_i'$ and
the balls $\Delta_i'$, $i\in I_0$, are pairwise disjoint.
We denote $\Delta_i= \frac1M\Delta_i'$.
Notice that the $\bS^n$-balls $\Delta_i$, $i\in I_0$, satisfy $M\Delta_i\cap M\Delta_j=\varnothing$ for all $i\neq j$. Moreover,  for each ball $\Delta_i$ there exists some $y_i\in T_{c_0}\cap \frac12\overline\Delta_i$ such that $\delta_{B}(y_i)=\frac12\rad_{\bS^n}(\Delta_i)$.
\vv

\begin{lemma}\label{leminteg}
	For $0<s<n$, we have
	$$\gamma_s(\Omega) \lesssim_M \sum_{i\in I_0} \rad_{\bS^n}(\Delta_i)^n.$$
\end{lemma}

\begin{proof}
	Changing variables, we get
	\begin{align}\label{eqgammadd}
	\gamma_s(\Omega) & = \int_{T_{c_0}} \delta_B(y)^{n-s}\,d\HH^{s}_\infty(y)
	= \int_0^\infty \HH^{s}_\infty (\{y\in T_{c_0}: \delta_B(y) > t^{\frac1{n-s}}\})\,dt\\
	&= (n-s)\int_0^\infty t^{n-s-1}\,\HH^{s}_\infty (\{y\in T_{c_0}: \delta_B(y) > t\})\,dt\nonumber
	.
	\end{align}
	Since the balls $5\Delta_i'$ are centered in $\partial B$ and they cover $T_{c_0}$, 
	the condition $\delta_B(y) > t$ for $y\in T_{c_0}$ implies that $y$ belongs to some ball $5\Delta_i'$ with $\rad_{\bS^n}(5\Delta_i')>t$. Therefore,
	$$\HH^{s}_\infty (\{y\in T_{c_0}: \delta_B(y) > t\}) \lesssim \sum_{i\in I_0:\rad_{\bS^n}(5\Delta_i')> t} \rad_{\bS^n}(5\Delta_i')^{s}.
	$$
	Thus, by Fubini,
	\begin{align*}
		\gamma_s(\Omega) &\lesssim \int_0^\infty t^{n-s-1}\!\sum_{i\in I_0:\rad_{\bS^n}(5\Delta_i')> t} \rad_{\bS^n}(5\Delta_i')^{s}\,dt\\
		& = \sum_{i\in I_0} \rad_{\bS^n}(5\Delta_i')^{s} \int_{t< \rad_{\bS^n}(5\Delta_i')}t^{n-s-1}\,dt \approx\sum_{i\in I_0} \rad_{\bS^n}(5\Delta_i')^{n} 
		\approx_M \sum_{i\in I_0} \rad_{\bS^n}(\Delta_i)^{n}.
	\end{align*}
\end{proof}

\vv
Now we need to define
a family of separated balls centered in 
$$L:=\partial \wt B\setminus S.$$
For each $i\in I_0$, let $z_i\in\bS^n$ be the center of $\Delta_i$ and consider the segment
$$L_i = [0,z_i] \cap A(0,1/2,3/4).$$
Next we choose a maximal family of points $z_{i,j}\in L_i$, $j\in J_i$, which are $(M\rad_{\bS^n}(\Delta_i))$-separated, so that $\#J_i\approx
(M\rad_{\bS^n}(\Delta_i))^{-1}$ for each $i\in I_0$. It is easy to check that the balls $B_{i,j}:=B(z_{i,j},\rad_{\bS^n}(\Delta_i))$, with $i\in I_0$, $j\in J_i$, 
satisfy
$c_3MB_{i,j}\cap c_3MB_{i',j'} = \varnothing$ for a suitable absolute constant $c_3>0$ if $(i,j)\neq (i',j')$.
Notice that
\begin{equation}\label{eqsumbij}
\sum_{i\in I_0,j\in J_i} \rad(B_{i,j})^{n+1} = \sum_{i\in I_0} \#J_i\,\rad_{\bS^n}(\Delta_{i})^{n+1}\approx_M \sum_{i\in I_0} \rad_{\bS^n}(\Delta_{i})^n\gtrsim_M \gamma_s(\Omega),
\end{equation}
by Lemma \ref{leminteg}. 

To simplify notation, we denote 
$$\{B_k\}_{k\in K} = \{B_{i,j}\}_{i\in I_0,j\in J_i}.$$
We need now to construct an auxiliary Lipschitz graph $\Gamma$ which will play an important role in the proof of Main Lemma
\ref{mainlemma1}.
First observe that, from the fact that $\capp_{n-2}( B(y_i,\frac12\delta_B(y_i))\cap \bS^n\setminus \Omega) \geq c_0r^{n-2}$ for each $i\in I_0$ in the case $n	\geq3$ and the analogous estimate for the logarithmic capacity in the case $n=2$ (since $y_i\in T_{c_0}$), by Lemma \ref{lemmaprod2} we deduce that, for all $k\in K$,
$$\capp_{n-1}( \{y\in B_k\setminus \wt\Omega:\dist(y,\partial\wt B)\geq c_4\,\rad(B_k)\}\big)
 \gtrsim \rad(B_k)^{n-1},$$
 for a suitable constant $c_4\in (0,1/4)$. By covering $\{y\in \partial\wt\Omega \cap B_k:\dist(y,\partial\wt B)\geq c_4\,\rad(B_k)\}$
 with balls with radius $c_4\,\rad(B_k)/10$, we infer that there exists some point
 $z_k\in  B_k\setminus \wt\Omega$ such that $\dist(z_k,\partial\wt B)\geq c_4\,\rad(B_k)$ and so that 
 the ball $B_k':=B(z_k,c_4\rad(B_k)/4)$ satisfies
 \begin{equation}\label{eqbkprima}
 4B_k'\subset \wt B\quad \text{ and }\quad
 \capp_{n-1}\big(  \tfrac14B_k'\setminus \wt\Omega \big) \gtrsim \rad(B_k)^{n-1}.
 \end{equation}

For each $k\in K$, we consider a  Lipschitz function $\vphi_k:L \to \wt B$ supported in $\frac12c_2MB_k$ such that ${\rm Lip}( \vphi_k)\lesssim (M\,\rad(B_k))^{-1}$ and so that $z_k$ belongs to the graph of $\vphi_k$.
It is easy to check that such a function exists.
Then we define $\vphi:L\to \wt B$ as follows:
$$\vphi(y) = \left\{\begin{array}{ll}
\vphi_k(y)& \quad \mbox{if $y\in c_2MB_k\cap L$ fo some $k\in K$,}\\
&\\
0& \quad\mbox{otherwise.}
\end{array}\right.
$$
We assume that the balls $c_2MB_k$, $k\in K$, are pairwise disjoint (otherwise we could replace $c_2$ by a smaller constant in the construction above), and then it is clear that the definition is correct. It is also clear that $\vphi$ is a Lipschitz function whose graph is contained in $\overline{\wt B}$. Further, because the functions $\vphi_k$ are $cM$-Lipschitz and because of the smoothness of $L$, it follows easily that 
\begin{equation}\label{eqremark*}
\mbox{$\vphi|_F$ is $c'M^{-1}$-Lipschitz for any set $F\subset L$ with $\diam(F)$ small enough (depending on $M$).}
\end{equation}

We denote by $\Gamma$ the graph of $\vphi$, and we let $F_\Gamma$ be the closed set comprised between $L$ and $\Gamma$. Next we consider the
domain
$$\wt \Omega_\Gamma = \wt B\setminus F_\Gamma.$$
Notice that $\wt \Omega_\Gamma$ is a Lipschitz domain contained in $\wt B$ whose boundary equals $\Gamma \cup (S\cap \partial\wt B)$.

Recall that
$$f_0(y) =\int_{\partial\wt \Omega_0  \setminus S} \wt u_B\,d\omega_{\wt\Omega_0}^y\quad \mbox{ for all $y\in\wt\Omega_0$}$$
and $$f_0(x) = \wt u_B(x)\quad \mbox{ for all $x\in \wt B\setminus\wt \Omega_0$},$$
where $\Omega_0=\Omega\cap B$ and $\wt\Omega_0=\wt\Omega\cap\wt B$.
Now we define $f_1:\wt\Omega_\Gamma\to \R$ by
\begin{equation}\label{eqdeff1**}
f_1(x) = \int_{\partial\wt \Omega_\Gamma  \setminus S} f_0\,d\omega_{\wt\Omega_\Gamma}^x = \int_{\Gamma} f_0\,d\omega_{\wt\Omega_\Gamma}^x.
\end{equation}
Remark that $f_1$ is harmonic in $\wt\Omega_\Gamma$ and continuous in $\overline{\wt\Omega_\Gamma}$, except perhaps in $\Gamma\cap S$,
 because the boundary data $f_0\,\chi_{\Gamma}$ is continuous in $\partial\wt\Omega_\Gamma$.
Since  $f_0$ is superharmonic in $\wt\Omega_\Gamma$ and $f_1$ harmonic in $\wt\Omega_\Gamma$, $f_0$ and $f_1$ have the same boundary values in $\Gamma$, and 
$$\liminf_{y\to z} (f_0-f_1)(z) = \liminf_{y\to z} f_0(z) \geq 0\quad \mbox{ for all $z\in \partial\wt\Omega_\Gamma\setminus\Gamma$},$$
by the maximum principle then it follows that
\begin{equation}\label{eqvf11}
f_0(x) \geq f_1(x) \quad \mbox{ for all $x\in \wt \Omega_\Gamma$}
\end{equation}
(see Theorem 3.1.5 from \cite{AG}).
\vv

\begin{lemma}
For each $k\in K$, let $\Gamma_k$ be the following subset of $\Gamma$:
$$\Gamma_k = \big\{y\in\vphi_k(c_2MB_k\cap L): \dist(y,L) \geq \tfrac1{10} \rad(B_k)\big\}.$$
Then 
$$f_0(y) = f_1(y) \gtrsim \rad(B_k)\quad \text{ for all $y\in\Gamma\cap \tfrac14B_k'$.}$$
\end{lemma}

\begin{proof}
The fact that $f_0(y)=f_1(y)$ for $y\in\Gamma\cap \tfrac14B_k'$ is an immediate consequence of the definition of $f_1(y)$.

By \rf{eqbkprima} and Lemma \ref{lem2.1}, for all $y\in \frac14B_k'\cap \wt\Omega_0$, we have 
$$\omega_{\wt\Omega_0}^y(2B_k') \gtrsim \frac{\capp_{n-1}(\tfrac 14 B_k'\setminus \wt \Omega_0)}{\rad(B_k')^{n-1}} \gtrsim 
 1.$$
Then, taking into account that 
$\wt u_B(z)\approx \dist(z,\partial\wt B)\approx\rad(B_k')\approx\rad(B_k)$ for all $z\in 2B_k'$, we infer
$$f_0(y) \geq\int_{\partial\wt \Omega_0  \cap 2B_k'} \wt u_B\,d\omega_{\wt\Omega_0}^y\gtrsim \rad(B_k')\,\omega_{\wt\Omega_0}^y(2B_k') \gtrsim \rad(B_k)\quad \mbox{ for all $y\in\frac14B_k'\cap \wt\Omega_0$.}$$
On the other hand,  for $y\in \tfrac14B_k'\setminus \wt\Omega_0$, we also have $f_0(y) = \wt u_B(y) \approx \dist(z,\wt B)
\approx \rad(B_k)$. 
\end{proof}
\vv

From \rf{eqvf11}, the definition of $f_1$ in \rf{eqdeff1**}, 
the preceding lemma, and the doubling property of $\omega_{\wt\Omega_\Gamma}$ (since $\wt\Omega_\Gamma$ is a Lipschitz domain) we deduce that
$$f_0(x) \geq f_1(x) \gtrsim \sum_{k\in K} \rad(B_k)\,\omega_{\wt\Omega_\Gamma}^x(\tfrac14B_k') \approx \sum_{k\in K} \rad(B_k)\,\omega_{\wt\Omega_\Gamma}^x(B_k\cap\Gamma)
\quad \mbox{ for all $x\in \wt \Omega_\Gamma$}.$$
Hence the Main Lemma \ref{mainlemma1} is a consequence of the following result and 
\rf{eqsumbij}.

\begin{lemma}\label{finallemma}
For all
 $x\in A(0,1/2,3/4)\cap \wt\Omega$ such that $|x|^{-1}x\in \frac12 B$, we have
$$\sum_{k\in K} \rad(B_k)\,\omega_{\wt\Omega_\Gamma}^x(B_k\cap\Gamma)
\gtrsim_M \sum_{k\in K} \rad(B_k)^{n+1}.$$
\end{lemma}
\vv

To prove this, we will need the following auxiliary lemma.

\begin{lemma}\label{lem-KT}
Given any $\alpha\in(0,1/2)$, suppose that the constant $M$ above is chosen large enough, depending on $\alpha$. Then there exists some constant $C=C(\alpha)$ such that, for all
 $x\in A(0,1/2,3/4)\cap \wt\Omega$ such that $|x|^{-1}x\in \frac12 B$, all
$y\in \Gamma\cap A(0,1/3,4/5)$, and $0<r\leq R\leq 2$, we have
\begin{equation}\label{eqreif1}
C^{-1}\bigg(\frac Rr\bigg)^{n-\alpha} \leq \frac{\omega_{\wt\Omega_\Gamma^x}(B(y,R))}{\omega_{\wt\Omega_\Gamma^x}(B(y,r))}
\leq C\bigg(\frac Rr\bigg)^{n+\alpha}.
\end{equation}
\end{lemma}

This lemma follows from the local Reifenberg flatness of $\partial\wt\Omega_\Gamma$ far away from $S$, by \rf{eqremark*} and taking $M$ large enough. The arguments are essentially the same as the ones for Theorem 4.1 from \cite{Kenig-Toro-Duke} and Lemma 3.7 from \cite{Prats-Tolsa-CVPDE}, and so we omit the proof.
From now on, we assume $M$ large enough so that we can take $a=1/4$ in \rf{eqreif1}.

\vv

\begin{proof}[Proof of Lemma \ref{finallemma}]
Consider the function $h_0:\partial\wt B\to\R$ given by
$$h_0(y) = \sum_{k\in K} \rad(B_k)\,\chi_{B_k\cap L}(y)$$
and let $h:\wt B\to\R$ be its harmonic extension to $\wt B$:
$$h(y) = \int h_0(z)\,d\omega_{\wt B}^y(z) = \sum_{k\in K}  \rad(B_k)\,\omega_{\wt B}^y(B_k).$$ 
%By Lemma *** and a simple Harnack chain argument, it is clear that 
%\begin{equation}\label{eqhbk8}
%h(y) \gtrsim  \rad(B_k) \quad \mbox{ for all $y\in \Gamma\cap B_k'$.}
%\end{equation}
 By Lemma \ref{lemubom}, we have
 $\frac{d\omega_{\wt B}^x}{d\HH^n}\approx 1$ on $\wt B\cap B_k$ for each $k$ and for
  $x\in A(0,1/2,3/4)\cap \wt\Omega$ such that $|x|^{-1}x\in \frac12 B$. Using also the fact that $h$ is harmonic in $\wt B$, that
  $x\in \wt B\cap \wt\Omega_\Gamma$, and that $h$ vanishes in
  $L\setminus \bigcup_{k\in K} \frac12c_2B_k = \Gamma\setminus \bigcup_{k\in K} \frac12c_2B_k
  $, we derive
\begin{align*}
\sum_{k\in K} \rad(B_k)^{n+1} & \approx \sum_{k\in K} \rad(B_k)\,\omega_{\wt B}^x(B_k\cap L) = \int_L h\,d\omega_{\wt B}^x = \int_\Gamma h\,d\omega_{\wt \Omega_\Gamma}^x\\
& = \sum_{k\in K}\int_{\frac12c_2 B_k\cap \Gamma} h(y)\,d\omega_{\wt \Omega_\Gamma}^x(y) 
 = \sum_{j,k\in K}  \rad(B_j)\,\int_{\frac12c_2B_k\cap \Gamma} \omega_{\wt B}^y(B_j)\,d\omega_{\wt \Omega_\Gamma}^x(y)\\
 & \lesssim \sum_{j,k\in K}  \rad(B_j)\,  \sup_{y\in\frac12c_2B_k\cap \Gamma}\omega_{\wt B}^y(B_j)\,\,\omega_{\wt \Omega_\Gamma}^x(B_k).
\end{align*} 

For all $j,k\in K$, we denote
$$D(B_j,B_k) =\dist(B_j,B_k) + \rad(B_j) +  \rad(B_k).$$ 
\vv

\begin{claim}\label{claimfinal}
For all $j,k\in K$, we have
\begin{equation}\label{eqybjk7}
 \sup_{y\in\frac12c_2B_k\cap \Gamma}\omega_{\wt B}^y(B_j)\lesssim \frac{\rad(B_j)^n\,\rad(B_k)}{D(B_j,B_k)^{n+1}}.
\end{equation}	
\end{claim}

We defer the proof of the claim to the end of the proof of the lemma. Assuming this for the moment, we get
\begin{equation}\label{eqsumbjbk}
\sum_{k\in K} \rad(B_k)^{n+1} \lesssim 
\sum_{j,k\in K} \,  \frac{\rad(B_j)^{n+1}\,\rad(B_k)}{D(B_j,B_k)^{n+1}}\,\omega_{\wt \Omega_\Gamma}^x(B_k) =:S_0.
\end{equation}
To conclude the proof of the lemma we will bound the sum $S_0$ on right hand side above by
$$\sum_{k\in K} \rad(B_k)\,\omega_{\wt\Omega_\Gamma}^x(B_k).$$
To this end, we split the sum $S_0$ as follows:
$$S_0 =  \Bigg(\sum_{\substack{j,k\in K:\\ \rad(B_j)\leq\rad(B_k)}} \!  
 + \sum_{\substack{j,k\in K:\\ \rad(B_j)>\rad(B_k)}}\Bigg)\,\, \frac{\rad(B_j)^{n+1}\,\rad(B_k)}{D(B_j,B_k)^{n+1}}\,\omega_{\wt \Omega_\Gamma}^x(B_k)
 =:S_{0,1} + S_{0,2}.$$
\vv

\subsubsection*{Estimate of $S_{0,1}$}
Using that $\rad(B_j)\leq\rad(B_k)$ in this sum, we write
\begin{align*}
S_{0,1} & \leq \sum_{\substack{j,k\in K:\\ \rad(B_j)\leq\rad(B_k)}}
\frac{\rad(B_j)^{n}\,\rad(B_k)^2}{D(B_j,B_k)^{n+1}}\,\omega_{\wt \Omega_\Gamma}^x(B_k)\\
& = \sum_{k\in K} \rad(B_k)^2\,\omega_{\wt \Omega_\Gamma}^x(B_k) \,\,\sum_{i\geq 0} 
\sum_{\substack{j\in K:\, \rad(B_j)\leq\rad(B_k),\\ 2^i\rad(B_k)\leq D(B_j,B_k) < 2^{i+1}\rad(B_k)}}\!\!
\frac{\rad(B_j)^{n}\,}{D(B_j,B_k)^{n+1}}.
\end{align*}
Notice that, for a fixed $i\geq0$, the conditions $\rad(B_j)\leq\rad(B_k)$, $2^i\rad(B_k)\leq D(B_j,B_k) < 2^{i+1}\rad(B_k)$ imply that
$B_j\subset 2^{i+2}B_k$. Then we get
\begin{align}\label{eqal289}
\sum_{\substack{j\in K:\, \rad(B_j)\leq\rad(B_k),\\ 2^i\rad(B_k)\leq D(B_j,B_k) < 2^{i+1}\rad(B_k)}}\!\!\!
\frac{\rad(B_j)^{n}}{D(B_j,B_k)^{n+1}} & \lesssim 
\sum_{\substack{j\in K:\, B_j\subset 2^{i+2} B_k}}\!\!
\frac{\HH^n(L\cap B_j)}{(2^i\rad(B_k))^{n+1}} \\
&\lesssim \frac{\HH^n(L\cap 2^{i+2} B_k)}{(2^i\rad(B_k))^{n+1}}\approx \frac1{2^i\rad(B_k)}.\nonumber
\end{align}
Therefore,
$$S_{0,1} \lesssim \sum_{k\in K} \rad(B_k)^2\,\omega_{\wt \Omega_\Gamma}^x(B_k) \,\,\sum_{i\geq 0}\frac1{2^i\rad(B_k)}
\approx \sum_{k\in K} \rad(B_k)\,\omega_{\wt \Omega_\Gamma}^x(B_k) 
.$$

\vv
\subsubsection*{Estimate of $S_{0,2}$}
Now we assume that $\rad(B_j)>\rad(B_k)$.
Let $A\geq 1$ be the minimal constant such that $B_k\subset AB_j$. Notice that $\rad(AB_j)\approx D(B_j,B_k)$.
By Lemma \ref{lem-KT}, using also the doubling property of $\omega_{\wt \Omega_\Gamma}^x$ (since $\wt \Omega_\Gamma$ is a Lipschitz domain), we have
\begin{align*}
\omega_{\wt \Omega_\Gamma}^x(B_j) &\gtrsim \bigg(\frac{\rad(B_j)}{\rad(AB_j)}\bigg)^{n+a}\,\omega_{\wt \Omega_\Gamma}^x(AB_j)
 \\
&\gtrsim \bigg(\frac{\rad(B_j)}{\rad(AB_j)}\bigg)^{n+\alpha} \bigg(\frac{\rad(AB_j)}{\rad(B_k)}\bigg)^{n-\alpha}
\omega_{\wt \Omega_\Gamma}^x(B_k)\approx  \frac{\rad(B_j)^{n+\alpha}}{\rad(B_k)^{n-\alpha}\,D(B_j,B_k)^{2\alpha}}\,\omega_{\wt \Omega_\Gamma}^x(B_k).
\end{align*}
Plugging this estimate into the definition of $S_{0,2}$, we obtain
\begin{align*}
S_{0,2} & \lesssim
\sum_{\substack{j,k\in K:\\ \rad(B_j)>\rad(B_k)}}\!\! \frac{\rad(B_j)^{n+1}\,\rad(B_k)}{D(B_j,B_k)^{n+1}}\,
\frac{\rad(B_k)^{n-\alpha}\,D(B_j,B_k)^{2\alpha}}{\rad(B_j)^{n+\alpha}}\,\omega_{\wt \Omega_\Gamma}^x(B_j)\\
& \approx
\sum_{\substack{j,k\in K:\\ \rad(B_j)>\rad(B_k)}}\!\! \frac{\rad(B_j)^{1-\alpha}\,\rad(B_k)^{n+1-\alpha}}{D(B_j,B_k)^{n+1-2\alpha}}\,
\,\omega_{\wt \Omega_\Gamma}^x(B_j).
\end{align*}
Using that $\rad(B_k)<\rad(B_j)$, we get
$$S_{0,2} \lesssim \sum_{j\in K} \rad(B_j)^{2-2\alpha}\,\omega_{\wt \Omega_\Gamma}^x(B_j) \,\,\sum_{i\geq 0} 
\sum_{\substack{k\in K:\, \rad(B_k)<\rad(B_j),\\ 2^i\rad(B_j)\leq D(B_j,B_k) < 2^{i+1}\rad(B_j)}}\!\!
\frac{\rad(B_k)^{n}\,}{D(B_j,B_k)^{n+1-2\alpha}}.
$$
Arguing as in \rf{eqal289}, we deduce that for each $i\geq0$
$$\sum_{\substack{k\in K:\, \rad(B_k)<\rad(B_j),\\ 2^i\rad(B_j)\leq D(B_j,B_k) < 2^{i+1}\rad(B_j)}}\!\!
\frac{\rad(B_k)^{n}\,}{D(B_j,B_k)^{n+1-2\alpha}}\lesssim \frac1{\big(2^i\rad(B_j)\big)^{1-2\alpha}}.
$$
Thus, since $\alpha<1/2$,
$$S_{0,2} 
\lesssim \sum_{j\in K} \rad(B_j)^{2-2\alpha}\,\omega_{\wt \Omega_\Gamma}^x(B_j)\,\,\sum_{i\geq 0} \frac1{\big(2^i\rad(B_j)\big)^{1-2\alpha}}
\approx \sum_{j\in K} \rad(B_j)\,\omega_{\wt \Omega_\Gamma}^x(B_j).$$

\vv
Gathering the estimates obtained for $S_{0,1}$ and $S_{0,2}$, we obtain
$$
S_0\lesssim \sum_{k\in K} \rad(B_k)\,\omega_{\wt \Omega_\Gamma}^x(B_k),$$
as wished. Together with \rf{eqsumbjbk}, this concludes the proof of the lemma, modulo  Claim \ref{claimfinal}.
\end{proof}

\vv

\begin{proof}[Proof of Claim \ref{claimfinal}]
In the case $j=k$, we have 
$$\frac{\rad(B_j)^n\,\rad(B_k)}{D(B_j,B_k)^{n+1}}\approx1$$
and so the estimate \rf{eqybjk7} is trivial.

So we assume that $j\neq k$. Let $y\in\frac12c_2B_k\cap \Gamma$, and let $\wh B_j$ be a ball concentric with $B_j$ such that $2\wh B_j\cap \frac12c_2B_k = \varnothing$ and $3\wh B_j\cap \frac12c_2B_k \neq \varnothing$. Clearly, by construction, we have $B_j\subset\wh B_j$ and $\rad(\wh B_j)\approx D(B_j,B_k)$. By
Lemma \ref{lempole} and \ref{lemubom}, we have
\begin{equation}\label{eqfrac84}
\frac{\omega_{\wt B}^y(B_j)}{\omega_{\wt B}^y(\wh B_j)}
 \approx \frac{\omega_{\wt B}^x(B_j)}{\omega_{\wt B}^x(\wh B_j)} 
  \approx \frac{\HH^n(B_j\cap \wt B) }{\HH^n(\wh B_j\cap \wt B)} 
  \approx \bigg(\frac{\rad(B_j)}{\rad(\wh B_j)}\bigg)^n,
\end{equation}
where we took $x\in A(0,1/2,3/4)\cap \wt\Omega$ such that $|x|^{-1}x\in \frac12 B$.
On the other hand, by Lemma \ref{lem:wG}, if we let $p\in\wt B$ be a corkscrew point for $\wh B_j$ (that is, $p\in\wh B_j$ with $\dist(p,\partial\wt B)\approx\rad(\wh B_j)$), we have
$$
\omega_{\wt B}^y(\wh B_j) \approx g_{\wt B}(y,p) \,\rad(\wh B_j)^{n-1} \approx \omega_{\wt B}^p(B(y,2\dist(y,\partial\wt B)))\,\frac{\rad(\wh B_j)^{n-1}}{\dist(y,\partial\wt B)^{n-1}}.$$
By Lemmas \ref{lempole} and \ref{lemubom}, we also have
$$ \omega_{\wt B}^p(B(y,2\dist(y,\partial\wt B))) \approx \frac{ \omega_{\wt B}^x(B(y,2\dist(y,\partial\wt B)))}{ \omega_{\wt B}^x(\wh B_j)} 
\approx \bigg(\frac{ \dist(y,\partial\wt B)}{\rad(\wh B_j)}\bigg)^n.$$
Therefore,
$$\omega_{\wt B}^y(\wh B_j)
\approx  \bigg(\frac{ \dist(y,\partial\wt B)}{\rad(\wh B_j)}\bigg)^n\,\frac{\rad(\wh B_j)^{n-1}}{\dist(y,\partial\wt B)^{n-1}} =  \frac{ \dist(y,\partial\wt B)}{\rad(\wh B_j)}.$$
Plugging this estimate into \rf{eqfrac84}, we obtain
$$\omega_{\wt B}^y(B_j)\approx \bigg(\frac{\rad(B_j)}{\rad(\wh B_j)}\bigg)^n\,\frac{ \dist(y,\partial\wt B)}{\rad(\wh B_j)} \approx  \frac{ \rad(B_j)^n\,\dist(y,\partial\wt B)}{D(B_j,B_k)^{n+1}} \lesssim 
\frac{ \rad(B_j)^n\,\rad(B_k)}{D(B_j,B_k)^{n+1}}
.$$
\end{proof}

\vv

% ***************************************************************************

\section{The proof of Theorems \ref{teomain0} and \ref{teomain1} in the non-Wiener regular case}

In Sections \ref{seccontra} and \ref{secmainlemma} we have proved and Theorem \ref{teomain0} and Theorem \ref{teomain1} under the assumption that
$\wt\Omega$ and $\wt\Omega'$ are Wiener regular open sets in $\R^{n+1}$. In this section we will prove both theorems
in full generality by an approximating argument. To this end, for each $k\geq1$, we denote
$$S_k = \big\{x\in \Omega:\dist_{\bS^n}(x,\partial_{\bS^n}\Omega)\leq 2^{-k}\big\},\qquad
\Omega_k=\Omega\setminus S_k.$$
Clearly, we have
$\HH^n(\Omega\triangle \Omega_k) = \HH^n(\Omega\setminus \Omega_k) \to0$ as $k\to\infty$.
Since we are assuming that the barycenter $x_\Omega$ of $\Omega$ is well defined, it follows easily that the barycenter $x_{\Omega_k}$ 
of $\Omega_k$ is also well defined for $k$ large enough and $x_{\Omega_k}\to x_\Omega$ as $k\to\infty$.
We denote by $B$ and $B_k$ the $\bS^n$-balls with the same barycenter and $\HH^n$-measure as $\Omega$ and $\Omega_k$, respectively, so that
$B_k\to B$ in Hausdorff distance as $k\to\infty$.
We also set
$$\wt\Omega_k = \{y\in A(0,1/4,1):|y|^{-1}y\in \Omega_k\},\qquad \wt B_k = \{y\in A(0,1/4,1):|y|^{-1}y\in B_k\},$$
and 
$$\wt\Omega_k' = \{y\in A(0,1/8,2):|y|^{-1}y\in \Omega_k\},\qquad \wt B_k' = \{y\in A(0,1/8,2):|y|^{-1}y\in B_k\}.$$

\vv
\begin{lemma}
For each $k\geq1$, if $\Omega_k\neq \varnothing$, then $\wt\Omega_k$ and $\wt\Omega_k'$ are Wiener regular.
\end{lemma}

\begin{proof}
We will show that $\wt\Omega_k$ is Wiener regular. The same arguments are valid for $\wt\Omega_k'$. By the discussion
at the beginning of Section \ref{secharm} (see \rf{eqwiener}), it is enough to show that for all $x\in\partial\wt\Omega_k$ and any $r>0$ small enough,
\begin{equation}\label{eqcapfac1}
\capp_{n-1}(B(x,r)\setminus\wt\Omega_k)\gtrsim r^{n-1}.
\end{equation}
Clearly this holds for all $x\in \partial A(0,1/4,1) \cap \partial\wt\Omega_k$ because $B(0,1)^c \cup \bar B(0,1/4)\subset \wt\Omega_k^c$. So consider the case when $x\in \partial\wt\Omega_k \setminus \partial A(0,1/4,1) = \partial\wt\Omega_k \cap A(0,1/4,1)$.
Let $\xi=|x|^{-1}x$, so that $\xi\in\partial_{\bS^n}\Omega_k$. Take $\eta\in \partial_{\bS^n}\Omega$ such that
$\dist_{\bS^n}(\xi,\eta) = \dist_{\bS^n}(\xi,\partial_{\bS^n}\Omega)$. Notice that $\dist_{\bS^n}(\xi,\eta)= 2^{-k}$, by the definition of $\Omega_k$. In fact, $\bar B_{\bS^n}(\eta,2^{-k}) \subset S_k \subset \bS^n\setminus\Omega_k$. Since $\xi\in
\partial_{\bS^n}\bar B_{\bS^n}(\eta,2^{-k})$, we deduce that, for $0<r< 2^{-k}$, $B_{\bS^n}(\xi,r)\setminus\Omega_k$ contains an $\bS^n$-ball with radius comparable to $r$. From this fact we deduce that for the same range $0<r< 2^{-k}$, $B(x,r)\setminus \wt\Omega_k$ contains a ball with radius comparable to $r$, which yields \rf{eqcapfac1}.
\end{proof}

\vv

\begin{lemma}\label{lemaprox37}
We have
$$\lim_{k\to\infty}\lambda(\Omega_k) = \lambda(\Omega)\qquad \text{and}\qquad \lim_{k\to\infty}\lambda(B_k) = \lambda(B).$$
\end{lemma}

\begin{proof}
Let $U$ be the connected component of $\Omega$ such that $\lambda(\Omega) = \lambda(U)$, and let $U_k = U\setminus S_k = U\cap \Omega_k.$
Let $u_k\in W_0^{1,2}(U_k)$ be the Dirichlet eigenfunction associated with $\lambda(U_k)$, normalized so that $u_k\geq0$ and $\|u_k\|_{L^2(U_k)}=1$.
We assume that $u_k$ is extended by $0$ to the whole $U$. For any subsequence of $\{u_k\}_k$, there exists another subsubsequence 
$\{u_{k_j}\}_j$ which converges weakly in $W_0^{1,2}(U)$ and strongly in $L^2(U)$ (by the compact embedding $W^{1,2}(U)\subset L^2(U)$) to some function $u\geq 0$ such that $\|u\|_{L^2(U)}=1$. Since $\{\lambda(U_{k_j})\}_j$ is a non-increasing sequence and 
$\lambda(U_{k_j})\geq \lambda(U)$ for any $j$, it also holds
that $\lambda(U_{k_j})$ converges to some $\alpha\in [\lambda(U),\infty)$.
Passing to the limit the equation 
$$-\Delta u_{k_j} = \lambda_{k_j}\,u_{k_j},$$ 
it follows that, in the weak sense,
$$-\Delta u = \alpha\,u\quad \mbox{ in $U$}.$$
Since $u$ is a Dirichlet eigenfunction for $U$, which is  connected, and moreover $u\geq0$, it follows that $u$ is a Dirichlet eigenfunction corresponding to the first eigenvalue for $U$ and
$\alpha = \lambda(U) = \lambda(\Omega)$.
Consequently,
$$\lim_{k\to\infty}\lambda(U_k) = \lambda(\Omega).$$
Since $U_k\subset\Omega_k$, this implies that
$$\limsup_{k\to\infty}\lambda(\Omega_k) \leq \limsup_{k\to\infty}\lambda(U_k) = \lambda(\Omega).$$
On the other hand, since $\Omega_k\subset\Omega$, we have $\lambda(\Omega_k) \geq \lambda(\Omega)$ for any $k$, and thus 
$$\lim_{k\to\infty}\lambda(\Omega_k) = \lambda(\Omega).$$

The fact that $\lim_{k\to\infty}\lambda(B_k) = \lambda(B)$ can be proved by analogous arguments. Indeed, since $\HH^n(B_k) = 
\HH^n(\Omega_k)\leq \HH^n(\Omega) = \HH^n(B)$, we can assume that $B_k\subset B$ and that they are concentric. Then the same arguments used for $\Omega_k$ work in this case.
\end{proof}
\vv

\begin{proof}[\bf Proof of Theorem \ref{teomain0} for $\bM^n=\bS^n$]
Suppose $n\geq3$.
We consider $\Omega_k$ and $B_k$ as above, so that 
$$\lambda(\Omega) - \lambda(B)  = \lim_{k\to\infty} \big(\lambda(\Omega_k) - \lambda(B_k)\big).$$
We have to show that, for a given $a\in(0,1)$, for any $t\in (0,1)$,
\begin{equation}\label{eqobjec9}
\lambda(\Omega) - \lambda(B) \geq C(a,\beta)\,\bigg(
 \avint_{\partial_{\bM^n} ((1-t) B)} \frac{\capp_{n-2}(B_{\bM^n}(y,atr_{B})\setminus \Omega)}{(t\,r_B)^{n-3}}\,d\HH^{n-1}(y)\bigg)^2.
\end{equation}
Choose $a'$ such that $a<a'<1$ and fix some $t\in (0,1)$.
Denote by $T_k$ the translation $T_k(y) = y + (x_{B_k}-x_B)$.
Since $\HH^{n-1}|_{\partial_{\bM^n} ((1-t) B_k)}$ coincides with the image measure $T_k\#
\HH^{n-1}|_{\partial_{\bM^n} ((1-t) B)}$ times $c_k= \frac{\HH^{n-1}(\partial_{\bM^n} ((1-t) B_k))}{\HH^{n-1}(\partial_{\bM^n} ((1-t) B))}$, we have
\begin{multline*}
\avint_{\partial_{\bM^n} ((1-t) B)} \capp_{n-2}(B_{\bM^n}(y,atr_{B})\setminus \Omega)\,d\HH^{n-1}(y)\\
 = \avint_{\partial_{\bM^n} ((1-t) B_k)} \capp_{n-2}(B_{\bM^n}(y+(x_{B}-x_{B_k}),atr_{B})\setminus \Omega)\,d\HH^{n-1}(y).
\end{multline*}
Since $B_k$ tends to $B$ in Hausdorff distance, it also holds that $x_{B_k}\to x_B$ and 
$r_{B_k}\to r_B$. So for $k$ large enough,
$B_{\bM^n}(y+(x_{B}-x_{B_k}),atr_{B}) \subset B_{\bM^n}(y,a'tr_{B_k}).$
Consequently, using also that $\Omega_k\subset\Omega$,
$$\capp_{n-2}(B_{\bM^n}(y+(x_{B}-x_{B_k}),atr_{B})\setminus \Omega)\leq\capp_{n-2}(B_{\bM^n}(y,a'tr_{B_k})\setminus \Omega_k).$$
 Therefore, taking into account that $r_{B_k}\approx r_B$ for $k$ large enough and applying
 Theorem \ref{teomain0} for $\Omega_k$ and $B_k$,
 we deduce
\begin{align*}
\avint_{\partial_{\bM^n} ((1-t) B)} &\frac{\capp_{n-2}(B_{\bM^n}(y,atr_{B})\setminus \Omega)}{(t\,r_B)^{n-3}}\,d\HH^{n-1}(y) \\ & \lesssim 
\avint_{\partial_{\bM^n} ((1-t) B_k)} \frac{\capp_{n-2}(B_{\bM^n}(y,a'tr_{B_k})\setminus \Omega_k)}{(t\,r_B)^{n-3}}\,d\HH^{n-1}(y) \leq C(a',\beta) \big(\lambda(\Omega_k)-\lambda(B_k)\big)
\end{align*}
Letting $k\to\infty$, we get \rf{eqobjec9}.

The arguments for the case $n=2$ are analogous.
\end{proof}
\vv

For the proof of Theorem \ref{teomain1}, we denote
$$\gamma_{s,c_0}(\Omega,B,a) = \int_{T_{c_0}(\Omega,B,a)} \delta_B(y)^{n-s}\,d\HH^{s}_\infty(y).$$

\begin{lemma}
For $d\in (0,n)$ and
any $a'\in (a,1)$ and $c_0'\in (0,c_0)$, we have
$$\liminf_{k\to\infty}\gamma_{s,c_0'}(\Omega_k,B_k,a') \geq \gamma_{s,c_0}(\Omega,B,a).$$
\end{lemma}

\begin{proof}Recall that the condition $x\in T_{c_0}(\Omega,B,a)$ means that $x\in B\setminus \Omega$ and in the case $n\geq3$, $\capp_{n-2}(B_{\bS^n}(x,a\,\delta_B(x))\setminus \Omega)\geq c_0\,\delta_B(x)^{n-2}$, while in the case $n=2$ we ask  $\capp_L(B_{\bS^n}(x,a\,\delta_B(x))\setminus \Omega)\geq c_0\,\delta_B(x)$.
Consider an arbitrary constant $\tau\in (1/2,1)$, and let $\tau B$ be the $\bS^n$-ball concentric with $B$ with radius $\tau\,\rad_{\bS^n}(B)$. Since $B_k$ tends to $B$ in Hausdorff distance (by the arguments at the beginning of this section), we deduce that $\tau B\subset B_k$ for $k$ large enough (depending on $\tau$). Therefore,
\begin{equation}\label{eqinc931}
\tau B\setminus \Omega \subset B_k\setminus \Omega_k\quad \mbox{ for $k\geq k_0(\tau)$.}
\end{equation}

Let us check that
\begin{equation}\label{eqinc932}
\tau B \cap T_{c_0}(\Omega,B,a) \subset  T_{c_0'}(\Omega_k,B_k,a') \quad \mbox{ for $k\geq k_0(\tau,a,a'c_0,c_0')$.}
\end{equation}
Indeed, for $x\in \tau B\setminus \Omega$, 
we have $a\,\delta_B(x) \leq a'\,\delta_{B_k}(x)$  
for $k$ large enough depending on $\tau,a,a'$. 
Similarly, $c_0\delta_B(x)^{n-2} \geq c_0'\,\delta_{B_k}(x)^{n-2}$ in the case $n\geq 3$ and 
$c_0\delta_B(x) \geq c_0'\,\delta_{B_k}(x)$ in the case $n=2$ for $k$ large enough depending on $c_0, c_0'$.
Consequently, in the case $n\geq3$, for $x\in \tau B \cap T_{c_0}(\Omega,B,a)$,
$$\capp_{n-2}(B(x,a'\,\delta_{B_k}(x))\cap \bS^n\setminus \Omega)\geq \capp_{n-2}(B(x,a\,\delta_{B}(x))\cap \bS^n\setminus \Omega)\geq  c_0\,\delta_B(x)^{n-2} \geq  c_0'\,\delta_{B_k}(x)^{n-2}.$$
An analogous estimate holds in the case $n=2$. Together with \rf{eqinc931}, this yields \rf{eqinc932}.

 We split, using the identity \rf{eqgammadd},
\begin{align*}
\gamma_{s,c_0}(\Omega,B,a) &= (n-s)\int_0^\infty t^{n-s-1}\,\HH^s_\infty\big(\{y\in T_{c_0}(\Omega,B,a):\delta_{B}(y)>t\}\big)
\,dt\\
& \leq (n-s)
\int_0^\infty t^{n-s-1}\,\HH^s_\infty\big(\{y\in B\setminus \tau B:\delta_{B}(y)>t\}\big)
\,dt\\
&\quad + (n-s)
\int_0^\infty t^{n-s-1}\,\HH^s_\infty\big(\{y\in \tau B \cap T_{c_0}(\Omega,B,a):\delta_{B}(y)>t\}\big)\,dt =: I_1+ I_2.
\end{align*}
Concerning the integral $I_1$, all points $y\in B\setminus \tau B$ satisfy $\delta_B(y)\leq (1-\tau)\,\rad_{\bS^n}(B)$.
Hence,
$$I_1\leq (n-s)\int_0^{C(1-\tau)} t^{n-s-1}\,\HH^s_\infty(B\setminus \tau B)\,dt \lesssim (1-\tau)^{n-s}.$$
Regarding $I_2$, by \rf{eqinc932} and the fact that $\delta_B(y) \leq a^{-1}a'\,\delta_{B_k}(y)$ for $y\in\tau B\setminus\Omega$ and $k\geq k_0(\tau,a,a'c_0,c_0')$, we have
$$\HH^s_\infty\big(\{y\in \tau B \cap T_{c_0}(\Omega,B,a):\delta_{B}(y)>t\}\big)
\leq \HH^s_\infty\big(\{y\in  T_{c_0'}(\Omega_k,B_k,a'):\delta_{B_k}(y)>(a')^{-1}at\}\big).
$$
Therefore, changing variables,
\begin{align*}
I_2 & \leq (n-s)\int_0^\infty t^{n-s-1}\,\HH^s_\infty\big(\{y\in  T_{c_0'}(\Omega_k,B_k,a'):\delta_{B_k}(y)>(a')^{-1}at\}\big)\,dt\\
& = (a^{-1}a')^{n-s}\,\gamma_{s,c_0'}(\Omega_k,B_k,a').
\end{align*}
So
$$\gamma_{s,c_0}(\Omega,B,a) \leq C(1-\tau) + (a^{-1}a')^{n-s}\,\gamma_{s,c_0'}(\Omega_k,B_k,a')\quad \mbox{ for $k\geq k_0(\tau,a,a'c_0,c_0')$.}
$$
Since we can take $\tau\in (1/2,1)$ arbitrarily close to $1$  and $a'>a$ as close to $a$ as wished, the lemma follows.
\end{proof}
\vv

\begin{proof}[\bf Proof of Theorem \ref{teomain1} for $\bM^n=\bS^n$]
From the preceding lemmas and the validity of Theorem \ref{teomain1} for the sets $\Omega_k$, we deduce
$$\gamma_{s,c_0}(\Omega,B,a)
\leq \liminf_{k\to\infty}\gamma_{s,c_0'}(\Omega_k,B_k,a') \lesssim_{\beta,c_0',a'} \liminf_{k\to\infty} \big(\lambda(\Omega_k) - \lambda(B_k)\big) = \lambda(\Omega) - \lambda(B),$$
for any $a'\in (a,1)$ and $c_0'\in (0,c_0)$, so that Theorem \ref{teomain1} holds for $\Omega$.
\end{proof}
\vv

% ***************************************************************************

\section{The proof of Theorem \ref{teomain0} and Theorem \ref{teomain1} for $\bM^n=\R^n$ and $\bM^n = \bH^n$.}

The proofs of Theorems \ref{teomain0} and \ref{teomain1} for Euclidean and hyperbolic domains are very similar to the one for spherical domains. For this reason, we just sketch the main ideas of the proofs.

\subsection{Domains in Euclidean space}
Remark that in this case, i.e. $\bM^n=\R^n$, one can state scale invariant versions of Theorem \ref{teomain0} and Theorem \ref{teomain1}, and so we may assume that $B=B(0,1)$ in the proofs, with $B$ attaining the infimum in
\rf{eqmain00}, \rf{eqmain01}, or \rf{eqmain1} modulo a factor of $1/2$.

For $\Omega,B\subset\R^{n+1}$ as in Theorems \ref{teomain0} and \ref{teomain1}, we let
$\alpha_\Omega=\lambda_\Omega^{1/2}$, $\alpha_B=\lambda_B^{1/2}$. We also let
$\wt u_B$ and $\wt u_\Omega$ be the respective extensions of $u_B$ and $u_\Omega$ defined by
$$\wt u_B(x,t) = u_B(x)  \,e^{\alpha_B t},\qquad \wt u_\Omega(x,t) = u_\Omega(x)  \,e^{\alpha_Bt},$$
for $(x,t)\in\R^{n+1}$. This type of extension replaces the $\alpha_B$-homogeneous extension used for the case of spherical domains.
We consider the cylindrical domains $\wt\Omega$, $\wt\Omega'$, $\wt B$, $\wt B'$ in $\R^{n+1}$ defined by:
$$\wt B = B \times (0,1),\qquad \wt B' = B \times (-1,2),\qquad
\wt\Omega = \Omega \times (0,1),\qquad \wt\Omega' = \Omega \times (-1,2).
$$
It is immediate to check that 
$\Delta \wt u_B =0$ in $\wt B'$, while in $\wt\Omega'$ we have
$$\Delta \wt u_\Omega(x,t) =   e^{\alpha_B t}\,\Delta_{\R^n} u_\Omega(x) + \lambda_B\, e^{\alpha_B t}\,u_\Omega(x)
= (\lambda_B - \lambda_\Omega) \,\wt u_\Omega(x,t).$$ 
We should compare this identity with \rf{eqlaplace3}.

As in Claim \ref{claim2}, there exists an extension $u_{B}^{ex}$ of $u_B|_B$ to $\R^n$ which belongs to $C^2(\R^n)$ and such that $\|\nabla^j u_B^{ex}\|_{\infty,\R^n}\lesssim 1$ for $j=0,1,2$  so that it is supported in the $\R^n$-ball $B_0=2B$.
Then we consider the function $v:\R^n\to\R$ defined by $v= u_B^{ex} - u_\Omega$ and we let 
$$\wt v(x,t) = v(x)\,\,e^{\alpha_B t},$$
 so that
$\wt v= \wt u_B^{ex} - \wt u_\Omega$. 
%Notice that in $\wt \Omega'$, by \rf{eqlaplace3} we have
%\begin{equation}\label{eqlapla}
%\Delta \wt v = \Delta \wt u_B^{ex} - \Delta \wt u_\Omega = \Delta \wt u_B^{ex} - (\lambda_B-\lambda_\Omega)\,|y|^{-2}\,\wt u_\Omega.
%\end{equation}
%Notice also that $\supp \Delta\wt u_B^{ex}\subset  (\wt B')^c$. 
As in \rf{eqv1234}, denoting $S=\R^n\times \{0,1\}$, we write, for $x\in\wt \Omega$,
\begin{align*}%\label{eqv1234''}
\wt v(x) & = \int_{\partial\wt \Omega} \wt v\,d\omega_{\wt\Omega}^x - \int_{\wt\Omega} \Delta\wt v(y)\,g_{\wt\Omega}(x,y)\,dy\\
& = \int_{\partial\wt \Omega \cap \wt B\setminus S} \wt v\,d\omega_{\wt\Omega}^x 
+ \int_{\partial\wt \Omega \setminus  (\wt B\cup S)} \wt v\,d\omega_{\wt\Omega}^x 
+ \int_{\partial\wt \Omega\cap S} \wt v\,d\omega_{\wt\Omega}^x
- \int_{\wt\Omega} \Delta\wt v(y)\,g_{\wt\Omega}(x,y)\,dy\nonumber\\
& =: \wt v_1(x) + \wt v_2(x) + \wt v_3(x) + \wt v_4(x),\nonumber
\end{align*}
where $g_{\wt\Omega}(x,y)$ is the Green function of $\wt\Omega$.
Then we continue in the same way as in the proof of Theorem \ref{teomain0} and Theorem \ref{teomain1}.
We estimate $|\wt v_2(x) + \wt v_3(x) + \wt v_4(x)|$ from above as in Lemma~\ref{lemv1234} and
$\wt v_1(x)$ from below following arguments very similar to the ones for $\bM^n = \bS^n$, first considering the Wiener regular case, and later deducing the general result from the Wiener regular one.
We leave the details for the reader. 

\vv

\subsection{Domains in hyperbolic space}
As for the Euclidean space, the proof of Theorems \ref{teomain0} and \ref{teomain1} when $\bM^n= \bH^n$ is very similar to the case $\bM^n = \bS^n$. For this reason, we only sketch the arguments and highlight the main modifications required.

Before getting started, some preliminary remarks: we consider the half-space model of the hyperbolic space $\bH^n$, which is the upper half space $\R^n_+$ with the metric $\frac{dx_1^2+\cdots dx_n^2}{x_n^2}$. Unlike the Euclidean case, and similarly to the spherical one, Theorems \ref{teomain0} and \ref{teomain1} are not scale invariant in this geometry. Hence the constants in the estimates naturally depend on the volume of $\Omega$.
% (which goes to infinity as $\Omega$ approaches the boundary $\R^{n-1}$ of $\bH^n$). 
On this, see also \cite[Remark 1.5]{AKN1}. %Finally, for $\Omega \subset \bH^n$, its barycenter is defined by $x_\Omega := \mathrm{argmin}_{x_0 \in \bH^n} \int_{\Omega} d(x, x_0)^2 d\mathrm{vol}(x)$. This is well defined. See Remark 3.3 in \cite{AKN1} and also \cite{BDS}. 

Let $\Omega, B \subset \bH^n$ be as in Theorems \ref{teomain0} and \ref{teomain1} (so in particular, $\HH^n(\Omega)=\HH^n(B) <\beta$),
with $\HH^n$ computed in the hyperbolic metric.  We assume  that $B=B_x$ attains the infimum in
\rf{eqmain00}, \rf{eqmain01}, or \rf{eqmain1} modulo a factor of $1/2$.
Since the Laplace-Beltrami operator is invariant by isometries, by applying a suitable isometry to $\Omega$ we can assume that  $B$ is centered at the point $(0,\ldots,0,1)\in\R^{n}_+$. From the fact that $\HH^n(B)=\mathrm{vol}(B) <\beta$, it follows
that $4B$ is far away from $\partial\R^n_+$ in the Euclidean metric. So inside $4B$ the Euclidean metric and the hyperbolic metric are comparable,
and the Hausdorff measures $\HH^s$ are comparable when computed with both metrics (with the comparability constant depending on $\beta$ in both cases).

Let $\lambda_\Omega, \lambda_B$ be the respective first Dirichlet eigenvalues of $\Omega$ and $B$, let $u_\Omega, u_B$ be the corresponding eigenfunctions, and put $\alpha_\Omega= \lambda_\Omega^{1/2}$ and $\alpha_B = \lambda_B^{1/2}$. For $(x,t) \in \bH^n \times \R$, we set
\begin{equation*}
	\wt u_\Omega(x,t) := u_\Omega(x) e^{t \alpha_B}\, \mbox{ and } \, \wt u_B(x,t) := u_B(x) e^{t \alpha_B}.
\end{equation*}
These are the extensions of $u_\Omega$ and $u_B$, respectively, analogous to the ones defined at the beginning of Section \ref{seccontra} for the spherical case.

We also consider the cylindrical domains in $\bH^n \times \R$ given by
\begin{equation*}
	\wt B = B \times (0, 1), \,\,\,\,  \wt B' = B \times (-1,2), \,\,\,\,  \wt \Omega = \Omega \times  (0,1) \, \mbox{ and } \, \wt \Omega' = \Omega \times (-1,2).
\end{equation*} 
It is easy to check that $\Delta_{\bH^n \times \R}\wt u_B=0$, where $\Delta_{\bH^n \times \R}$ is the Laplace-Beltrami operator in $\bH^n \times \R$. In local coordinates, this is equivalent to saying that
\begin{equation*}
	L\wt u_B = 0 \, \mbox{ in } \, \wt B',
\end{equation*}
where $L$ is a uniformly elliptic second order operator in divergence form given by $L u=  - \mathrm{div} (A \nabla u)$, where $A$ is the diagonal matrix with smooth coefficients defined by 
\begin{equation}
	a_{ij} = x_n^{2-n}\delta_{i,j}\, \mbox{ if }\, 1 \leq i,j \leq n, \,\, \mbox{ and } \, a_{ij} =x_n^{-n} \delta_{ij} \,\mbox{ otherwise.} 
\end{equation}
Moreover, note that
\begin{align}\label{hyp-eq2}
	L\wt u_\Omega(x,t) & = x_n^{-n}\left(  \Delta_{\bH^n} \wt u_\Omega (x,t) + (\partial_t^2 u_{\Omega}(x) e^{t \alpha_B})\right) 
	 = x_n^{-n} (-\lambda_\Omega + \lambda_B)\, \wt u_\Omega(x,t),
\end{align}
  which is the analog of \eqref{eqlaplace3} for the current setting.  Remark that  the multiplicative term $x_n^{-n}$ is bounded in $4B$ and so $L$ is uniformly elliptic on $4B$.
  
  The proof continues very similarly to the spherical case. We just highlight some of the differences. 
  \begin{itemize}
  	\item As far as the definition of Newtonian capacity is concerned, no changes is needed since the fundamental solution to $L$ is comparable to that of $-\Delta$ in $4B$. See equation (2.14) in \cite{AGMT}. In particular, an open set is Wiener regular for the Laplacian if and only if it is Wiener regular for uniformly elliptic operator such as $L$ (see \cite[Theorem 6.21]{Prats-Tolsa-Notes}).
  	\item Lemmas from Section \ref{secharm}: Lemma \ref{lem2.1} for elliptic uniformly elliptic operators can be found in \cite[Lemma 11.21]{HKM}. For Lemma \ref{lem:wG} see \cite[Lemma 2.8]{AGMT}. Finally, we couldn't find a reference for Lemma \ref{lempole}, but the proof in the elliptic case is essentially the same to that for harmonic measure. A key point for these estimates is that we only need to apply them for subdomains
	of $\wt B'':=4B\times [-4,4]$, where $L$ is uniformly elliptic. Further, we can redefine the matrix $A$ away from $\wt B''$ so that it becomes uniformly elliptic in the whole space if necessary.
	
  	\item About Lemma \ref{lemubom}: the bounds $\|\wt u_B\|_{\infty, \wt B'}, \|\wt u_\Omega\|_{\infty, \wt \Omega'\cap \wt B''} \lesssim 1 $ follows from the fact that $\alpha_B \approx_\beta 1$ and the mean value property on balls for subsolutions, see \cite[Theorem 8.17]{Gilbarg-Trudinger}). 
  	The argument to prove estimate 
  	\eqref{equb30} relies on a) estimates on the Green functions, b) the maximum principle and c) Harnack inequality. All these are available: see \cite[Lemma 2.6]{AGMT} for the properties of the Green function associated to the operator $L$; see \cite[Theorem 8.16]{Gilbarg-Trudinger} for b) and \cite[Section 8.8]{Gilbarg-Trudinger} for c). As far as the behaviour of elliptic measures $\omega_{L, \wt B'}, \omega_{L, \wt B}$ is concerned, that is, the second part of Lemma \ref{lemubom}, it follows again from the estimates on the Green functions connected to $L$. See \cite[Section 2.4]{AGMT} for more information on this. 
  	 
 \item Claim 1 in Section \ref{seccontra} also holds in this case. More specifically, there exists an extension $u^{ex}_B$ of $u_B|_B$ to $\bH^n$ which belongs to $C^2(\bH^n)$ and such that $\|\nabla^j u^{ex}_B\|_{\infty, \bH^n} \lesssim 1$ for $j=0,1,2$ and so that it is supported on an $\bH^n$-ball $B_0=2B$. The implicit constant in the previous bound might depend on $\beta$.
 \item The representation formula used in the important decomposition \eqref{eqv1234} in the current case, and assuming Wiener regularity of $\Omega$ and so of $\wt \Omega$, reads as 
 \begin{equation*}
 	\varphi(x) = \int_{\partial \wt \Omega} \varphi d \omega_L^x  + \int_{\wt \Omega} A^*(y) \nabla_y g_{\wt \Omega}(x,y) \cdot \nabla_y \varphi(y)\, dy
 \end{equation*}
See \cite{AGMT}, equation (2.6). In our case, the matrix $A$ is symmetric, and thus $A^*=A$. Here $g_{\wt \Omega}(\cdot, \cdot)$ is the Green function of $\wt \Omega$ associated to $L$. 
We consider a smooth function $\psi:\R^{n+1}\to\R$ which equals $1$ on $\wt B'$ and is supported on $\wt B''$.
Defining $v=u_B^{ex} - u_\Omega$, and extending it to $\bH^n \times \R$ by $\wt v(x,t) := v(x) \, e^{t \alpha_B}$, we write, for $x \in \wt \Omega$,
\begin{equation}\label{hyp-eq1}
	(\psi\,\wt v) (x) = \int_{\partial \wt \Omega} \psi \wt v \, d \omega_{L, \wt \Omega}^x - \int_{\wt \Omega} A(y) \, \nabla_y g_{\wt \Omega}(x,y) \cdot \nabla_y (\psi \wt v)(y) \, dy.
\end{equation}
Denoting $S= \bH^n \times \{0,1\}$ we split \eqref{hyp-eq1} as
\begin{align*}
	(\psi\,\wt v) (x) & = \int_{\partial \wt \Omega \cap \wt B \setminus S} \psi\,\wt v \, d \omega_{L, \wt \Omega}^x  + \int_{\partial \Omega \setminus (\wt B \cup S)} \psi\,\wt v \, d \omega_{L,\wt \Omega}^x + \int_{\partial\wt \Omega\cap S} \psi\,\wt v\,d\omega_{L,\wt\Omega}^x\\
	&\quad + \int_{\wt \Omega} A(y) \, \nabla_y g_{\wt \Omega}(x,y) \cdot \nabla_y (\psi\,\wt v)(y) \, dy \\
	& =: \wt v_1(x) + \wt v_2 (x) + \wt v_3(x) + \wt v_4(x). 
\end{align*}
\item The bounds for the analogous functions $\wt v_i$, $i=2,3,4$, in Lemma \ref{lemv1234}, rest on the behaviour of harmonic measure, the maximum principle and \eqref{eqlaplace3}.  Notice that the multiplication by $\psi$ ensures that the integrands above are supported in $\wt B''$, where $L$ is uniformly elliptic. Then the estimates for $\wt v_2$ and $\wt v_3$ are similar to the ones in Lemma \ref{lemv1234}. Regarding $\wt v_4$, this can be written as
$$\wt v_4(x)= \int A \, \nabla_y (\psi\, g_{\wt \Omega})(x,y) \cdot \nabla \wt v  \, dy +
\int A \, \nabla \psi \cdot\big(v \nabla_y g_{\wt \Omega})(x,y)- g(x,y)\,\nabla \wt v\big) \, dy.$$
The first integral on the right hand side equals $-\int L(\wt v)\psi\, g_{\wt \Omega})(x,y)  \, dy$ and this is estimated in the same ways as $\wt v_4$ in Lemma \ref{lemv1234}. For the second integral on the right hand side above we use the fact that $\supp\psi\cap \supp\wt v\subset (\Omega\setminus B)\times (-4,4)$. Then we obtain the same conclusion as in Lemma \ref{lemv1234}.
 The same can be said about the lower bound for $\wt v_1$. The proofs of Theorem \ref{teomain0} and \ref{teomain1}, then, continue in the same way, modulo technical adjustments which we omit. The extension to the non Wiener regular case is also very similar and we leave the details to the reader.
  \end{itemize}

% ***************************************************************************

\section{Proof of Theorem \ref{teo22}}

Without loss of generality, we assume that $\HH^n(\Omega_1)\leq \HH^n(\Omega_2)$. For simplicity, we assume that 
$a=1/2$ in the definition of $\ve_s(\Omega_1,\Omega_2)$, but minor modifications yield the result for any $a\in(0,1)$.
Let $\bar \alpha_i$ be the characteristic of the spherical cap $B_i\subset\bS^n$ with the same $\HH^n$ measure as $\Omega_i$.
The Friedland-Hayman \cite{FH} inequality ensures that
 $\alpha_1+\alpha_2-2\geq 0$.
Then we write
\begin{equation}\label{eqcadena}
\alpha_1+\alpha_2-2 = (\alpha_1-\bar\alpha_1) + (\alpha_2-\bar\alpha_2) + (\bar \alpha_1+\bar\alpha_2-2)\geq 0.
\end{equation}
By Sperner's inequality \cite{Sperner}, among all the open subsets with a fixed measure $\HH^n$ on $\bS^n$, the one that minimizes the characteristic constant is a spherical cap with that measure $\HH^n$. Hence, 
$$\alpha_i\geq
\bar\alpha_i.$$
So the three summands on the right hand side of \rf{eqcadena} are non-negative.
Further, if one of the caps $B_i$ differs from a hemisphere by a surface measure $h_0$,
that is, 
$$h_0=\max_i\Big|\HH^n(B_i)- \frac12\HH^n(\bS^n)\Big|,$$
then 
\begin{equation}\label{eqfac748}
\bar \alpha_1+\bar \alpha_2-2\geq c\,h_0^2.
\end{equation}
See Corollary 12.4 from \cite{CS}, for example.
So to prove Theorem \ref{teo22} we can assume that, for $i=1,2$,
\begin{equation}\label{eqfac749}
\Big|\HH^n(B_i)- \frac12\HH^n(\bS^n)\Big| \leq \frac1{100}\,\HH^n(\bS^n),
\end{equation}
because otherwise 
$$\alpha_1 + \alpha_2 - 2\gtrsim1 $$
and the statement in the theorem is trivial. Observe that \rf{eqfac749} implies that
\begin{equation}\label{eqfac750}
\beta\leq\HH^n(\Omega_i)=\HH^n(B_i)\leq \HH^n(\bS^n)-\beta
\end{equation}
for a suitable absolute constant $\beta$.
From this estimate it follows that $\bar \alpha_{i}\approx 1$ for $i=1,2$. So we can assume that $|\alpha_i-\bar \alpha_i|\leq \frac12\,
\bar \alpha_{i}$ because otherwise the theorem follows trivially from \rf{eqcadena}. So $\alpha_i\approx \bar \alpha_{i}\approx 1$.
Then from the identity $\lambda_i = \alpha_i(\alpha_i+n-1)$ (where $\lambda_i\equiv \lambda_{\Omega_i}$) it follows immediately that
\begin{equation}\label{eqfac751} 
\alpha_i - \bar \alpha_{i} \approx \lambda_i - \lambda_{B_i}.
\end{equation}

For $i=1,2$, let $x_i\in \bS^n$ be the barycenter of $\Omega_i$.
We assume that the barycenter exists because otherwise this means that $\int_{\Omega_i} y\,d\HH^n(y)=0$, while
$\left|\int_{B_i} y\,d\HH^n(y)\right|\gtrsim1$ for the spherical cap $B_i$ because of \rf {eqfac750} (independently of the choice of its center in $\bS^n$).
Hence, by the Allen-Kriventsov-Neumayer theorem, using \rf{eqfac750} and \rf{eqfac751}, we would obtain
$$\alpha_i - \bar\alpha_{i} \approx \lambda_i - \lambda_{B_i} \gtrsim \HH^n(\Omega_i\triangle B_i)^2 \geq 
\bigg|\int_{\Omega_i} y \,d\HH^n(y) - 
\int_{B_i} y \,d\HH^n(y)\bigg|^2 \gtrsim1,$$
which would yield the conclusion of the theorem. The same argument shows that, in fact, we have
$$\left|\int_{\Omega_i} y\,d\HH^n(y)\right|\approx 1.$$

From now on we assume that the spherical caps $B_i$, $i=1,2$, are centered in the barycenters $x_i$ of the $\Omega_i$'s.
In the case when $\HH^n(\Omega_i) + \HH^n(\Omega_2)=\HH^n(\bS^n)$ it follows easily that 
 the barycenters of $\Omega_1$ and $\Omega_2$ are opposite points in $\bS^n$, and $B_1$ and $B_2$ are complementary balls in 
 $\bS^n$.
 When the preceding condition does not hold, we need to be a little more careful.
 
 Let 
 $$\theta_0 = \HH^n(\bS^n) - \HH^n(\Omega_1) - \HH^n(\Omega_2),$$
 and suppose that $\theta_0>0$.
 For $i=1,2$,
 $$y_i = \frac1{\HH^n(\Omega_i)} \int_{\Omega_i} y\,d\HH^n(y),$$
 so that $x_i=y_i/|y_i|$. Also, let 
 $$y_2^* = \frac1{\HH^n(\bS^n\setminus\Omega_1)} \int_{\bS^n\setminus\Omega_1} y\,d\HH^n(y).$$
 Notice that $y_2^*=y_2$ if $\theta_0=0$ and that
  $$\HH^n(\Omega_1)\,y_1 + \HH^n(\bS^n\setminus\Omega_1)\,y_2^* = \int_{\bS^n} y\,d\HH^n(y)=0.$$
  So the barycenter $x_1$ and the point $x_2^*=y_2^*/|y_2^*|$ are antipodal points in $\bS^n$.
 We also have
 \begin{align*}
 |y_2 - y_2^*| & \leq \bigg|\frac1{\HH^n(\Omega_2)} - \frac1{\HH^n(\bS^n\setminus\Omega_1)}\bigg|\int_{\bS^n\setminus\Omega_1} |y|\,d\HH^n(y) \\
 & \quad+ \frac1{\HH^n(\bS^n\setminus\Omega_1)} 
 \bigg|\int_{\Omega_2} y\,d\HH^n(y) -  
 \int_{\bS^n\setminus\Omega_1} y\,d\HH^n(y)\bigg|\lesssim \theta_0 + \theta_0\approx\theta_0.
\end{align*}
Also, let $B_2^*$ be a spherical cap centered in $x_2^*$ with measure $\HH^n(B_2^*) = \HH^n(\bS^n\setminus \Omega_1)$.
Since 
$$|\rad_{\bS^n}(B_2) - \rad_{\bS^n}(B_2^*)|\approx |\HH^n(B_2)- \HH^n(B_2^*)|\lesssim \theta_0,$$
using also \rf{eqfac750}, we infer that
\begin{equation}\label{eqs1s2}
\dist_H(\partial_{\bS^n}B_2, \partial_{\bS^n}B_1) =
\dist_H(\partial_{\bS^n}B_2, \partial_{\bS^n}B_2^* ) \lesssim \theta_0.
\end{equation}
On the other hand, from the definition of $\theta_0$, it follows that
$$
h_0=\max_i\big|\HH^n(\Omega_i) - \frac12\,\HH^n(\bS^n)\big|\geq \frac12\,\theta_0.
$$
Then, by \rf{eqfac748},
\begin{equation}\label{eqgvn4}
\bar \alpha_1+\bar \alpha_2-2\gtrsim \theta_0^2\gtrsim \dist_H(\partial_{\bS^n}B_2, \partial_{\bS^n}B_1)^2.
\end{equation}

To estimate $\ve_s(\Omega_1,\Omega_2)$ we denote by $L_1$ be the $n$-plane that contains $\partial_{\bS^n}B_1$, and we let $L_0$ 
 be the $n$-plane through the origin parallel to $L_1$. Then we choose $H$ to be the half-space that contains $B_1$ and whose boundary is $L_0$ (notice that $B_1$ is contained in a hemisphere because of the assumption
$\HH^n(\Omega_1)\leq \HH^n(\Omega_2)$). First we consider the case $0<s<n$. In this case, we have
\begin{align}\label{eqb1b2}
\ve_s(\Omega_1,\Omega_2) & \leq  \int_{V_{c_0}(0,1,H,a)}\dist(y,L_0)^{n-s}\,d\HH_\infty^{s}(y)\\
& \leq \sum_{i=1}^2\int_{V_{c_0}(0,1,H,a)\cap B_i}\dist(y,L_0)^{n-s}\,d\HH_\infty^{s}(y) + 
\int_{\bS^n\setminus(B_1\cup B_2)}\dist(y,L_0)^{n-s}\,d\HH_\infty^{s}(y).\nonumber
\end{align}
First we deal with the last integral on the right hand side:
\begin{align}\label{eqkdc12}
\int_{\bS^n\setminus(B_1\cup B_2)}\dist(y,L_0)^{n-s}\,d\HH_\infty^{s}(y) & \leq \max_{\bS^n\setminus(B_1\cup B_2)}\,\dist(y,L_0)^{n-s}
\HH_\infty^{s}(\bS^n\setminus(B_1\cup B_2)) \\ &\lesssim \max_{i=1,2} \,\dist_H(\partial_{\bS^n} B_i, \bS^n\cap L_0)^{n-s}\,\HH_\infty^{s}(\bS^n\setminus(B_1\cup B_2)).\nonumber
\end{align}
Regarding $\dist_H(\partial_{\bS^n} B_i, \bS^n\cap L_0)$,  
for $i=1$, since $L_0$ is parallel to $L_1$ and $L_0$ splits $\bS^n$ in two hemispheres, we  have
\begin{equation}\label{eqdhj3}
\dist_H(\partial_{\bS^n} B_1, \bS^n\cap L_0) \lesssim |\HH^n(B_1) - \tfrac12\HH^n(\bS^n)|\leq h_0.
\end{equation}
Also, for $i=2$,
$$\dist_H(\partial_{\bS^n} B_2, \bS^n\cap L_0) \leq \dist_H(\partial_{\bS^n} B_2, \partial_{\bS^n} B_2^*) + 
\dist_H(\partial_{\bS^n} B_2^*,\bS^n\cap L_0).$$
Since $\partial_{\bS^n} B_2^*=\partial_{\bS^n} B_1$, from \rf{eqs1s2} and \rf{eqdhj3} we get
$$
\dist_H(\partial_{\bS^n} B_2, \bS^n\cap L_0) \lesssim \theta_0 + \dist_H(\partial_{\bS^n} B_1,\bS^n\cap L_0)\lesssim\theta_0 + h_0\lesssim h_0.
$$
So, for a suitable $C_2>0$,
\begin{equation}\label{eqdhj4} 
\dist_H(\partial_{\bS^n} B_i, \bS^n\cap L_0) \le C_2\,h_0\quad \text{ for both $i=1,2$.}
\end{equation}
Plugging this estimate into \rf{eqkdc12}, we obtain
\begin{equation}\label{eqkdc1210}
\int_{\bS^n\setminus(B_1\cup B_2)}\dist(y,L_0)^{n-s}\,d\HH_\infty^{s}(y)\lesssim h_0^{n-s}\,\HH^s_\infty(\bS^n\setminus(B_1\cup B_2)).
\end{equation}
To estimate $\HH^s_\infty(\bS^n\setminus(B_1\cup B_2))$, notice that by \rf{eqdhj4}, $\bS^n\setminus(B_1\cup B_2)\subset {\mathcal U}_{C_2 h_0}(\bS^n\cap L_0)$. So $\bS^n\setminus(B_1\cup B_2)$ can be covered by a family $I$ of
of $\bS^n$-balls of radius $2C_2h_0$ with bounded overlap, centered in $\bS^n\cap L_0$. Using that 
$\HH^{n-1}(\bS^n\cap L_0) = \HH^{n-1}(\bS^{n-1}) \approx 1$, it follows easily that $\# I\lesssim h_0^{1-n}$.
Hence, 
\begin{equation}\label{eqkdc121aa}
\HH^s_\infty(\bS^n\setminus(B_1\cup B_2)) \leq 
\HH^s_\infty({\mathcal U}_{C h_0}(\bS^n\cap L_0))\lesssim
h_0^s\,\,\# I \lesssim h_0^{s+1-n}.
\end{equation}
Thus, plugging this estimate into \rf{eqkdc1210}, we obtain
\begin{equation}\label{eqkdc121}
\int_{\bS^n\setminus(B_1\cup B_2)}\dist(y,L_0)^{n-s}\,d\HH_\infty^{s}(y)\lesssim h_0^{n-s}\,h_0^{s+1-n} = h_0.
\end{equation}

We turn our attention to the first summands on the right hand side of \rf{eqb1b2}.
We denote 
$$\delta_H(y) =\dist_{\bS^n}(y,\bS^n\cap L_0)$$
and
$$V_i = \{y\in V_{c_0}(0,1,H,a) \cap B_i: \delta_{H}(y) >10C_2 h_0\}.$$
Recall we are assuming $a=1/2$.
Then we split
\begin{align*}
\int_{V_{c_0}(0,1,H,a)\cap B_i} \dist(y,L_0)^{n-s}\,d\HH_\infty^{s}(y) &\lesssim \int_{V_{c_0}(0,1,H,a)\cap B_i\setminus V_i}
\dist(y,L_0)^{n-s}\,d\HH_\infty^{s}(y)\\
&\quad + 
\int_{V_i}\dist(y,L_0)^{n-s}\,d\HH_\infty^{s}(y) \\
& =: I_1 + I_2 .
\end{align*}
To deal with $I_1$ we take into account the fact that $\dist(y,L_0)\lesssim h_0$ in the domain of integration, 
and thus together with \rf{eqkdc121aa} this gives
$$
I_1\lesssim h_0^{n-s}\,\HH^{s}_\infty({\mathcal U}_{C h_0}(\bS^n\cap L_0)) \lesssim h_0.$$
%For $I_3$ we use \rf{eqdhj4}, so that we obtain
%$$I_1 + I_3\lesssim  h_0.$$

To estimate $I_2$ we take into account that for $y\in V_i$,
$$\dist(y,L_0) \approx \delta_H(y) \approx \dist(y,\partial_{\bS^n} B_i).$$
Thus, 
$$I_2\lesssim 
\int_{V_i}\dist(y,\partial_{\bS^n} B_i)^{n-s}\,d\HH_\infty^{s}(y).$$ 
Next we plan to apply 
 Theorem \ref{teomain1}, under the assumption \rf{eqfac749}.
Taking into account \rf{eqfac751}, this implies that, for a given $c_5>0$,
\begin{equation}\label{eqteo99}
\left(\int_{T_{i}}\dist(y,\partial_{\bS^n} B_i)^{n-s}\,d\HH_\infty^{s}(y) \right)^2\lesssim 
\lambda_i - \lambda_{B_i}\approx
\alpha_i - \bar\alpha_i,
\end{equation}
with $T_{i}$ defined by
%\footnote{In the definition of $T_i$ we choose the constant $3/4$ instead of $1/2$ as in $T_{c_0}$ in
%Theorem \ref{teomain1}. It is immediate to check that Theorem \ref{teomain1} is also valid with this choice.}
$$T_{i}  =  \big\{y\in B_i\setminus\Omega_i:\capp_{n-2}(B(y,\tfrac34\delta_{B_i}(y))\cap \bS^n\setminus \Omega_i)\geq c_5\,\delta_{B_i}(y)^{n-2}\big\},
$$
and, in the case $n=2$, 
$$
T_{i}  = \big\{y\in B_i\setminus\Omega_i:\capp_L(B(y,\tfrac34\delta_{B_i}(y))\cap \bS^n\setminus \Omega_i)\geq c_5\,\delta_{B_i}(y)\big\}.
$$
We claim that, for a suitable constant $c_5$,
$$V_i\subset T_{i}.$$
Indeed, for $y\in V_i\cap B_i$, by definition we have $\delta_{H}(y) >10C_2 h_0$ and consequently, by \rf{eqdhj4},
$$|\delta_{B_i}(y) - \delta_{H}(y)|\leq 2\,\dist_H(\partial_{\bS^n} B_i, \bS^n\cap L_0) \leq 2 C_2h_0\leq \frac15 \,\delta_{H}(y),$$
which readily implies that
$\frac34\,\delta_{B_i}(y) \geq \frac12\,\delta_H(y)$, or equivalently, $B(y,\frac12\,\delta_H(y))\subset B(y,\frac34\,\delta_{B_i}(y))$.
Then we derive
$$\capp_{n-2}(B(y,\tfrac34\delta_{B_i}(y))\cap \bS^n\setminus \Omega_i) \geq \capp_{n-2}(B(y,\tfrac12\delta_{H}(y))\cap \bS^n\setminus \Omega_i) \geq c_0\,\delta_{H}(y)^{n-1}\approx \delta_{B_i}(y)^{n-1},$$
as well as the analogous estimate for logarithmic capacity in the case $n=2$,
which yields our claim. Then by \rf{eqteo99}, we have
$$I_2^2\lesssim\left(\int_{V_{i}}\dist(y,\partial_{\bS^n} B_i)^{n-s}\,d\HH_\infty^{s}(y) \right)^2\leq 
\left(\int_{T_{i}}\dist(y,\partial_{\bS^n} B_i)^{n-s}\,d\HH_\infty^{s}(y) \right)^2\lesssim \alpha_i - \bar\alpha_i.
$$
This estimate, together with the one obtained for $I_1$, gives
$$\bigg(\int_{V_{c_0}(0,1,H,a)\cap B_i} \dist(y,L_0)^{n-s}\,d\HH_\infty^{s}(y)\bigg)^2 \lesssim \alpha_i - \bar\alpha_i + h_0^2.$$
Plugging this inequality into \rf{eqb1b2} and using also \rf{eqkdc12} and \rf{eqfac748}, we obtain
\begin{equation}\label{eqfafa7}
\ve_s(\Omega_1,\Omega_2)^2 \lesssim \sum_{i=1}^2(\alpha_i-\bar\alpha_i) + h_0^2 
\lesssim \sum_{i=1}^2(\alpha_i-\bar\alpha_i)
  + (\bar \alpha_1+\bar\alpha_2-2) = \alpha_1+\alpha_2-2.
\end{equation}
  
\vv
In the case $s=n$, the arguments are easier. We write $S_1 = \bS^n\cap H$, $S_2= \bS^n\setminus H$, and then
similarly to \rf{eqb1b2}, we write
\begin{align*}
\ve_n(\Omega_1,\Omega_2) & \leq  \HH^n(S_1\setminus \Omega_1) +\HH^n(S_2\setminus \Omega_2) \leq \sum_{i=1}^2 \HH^n(B_i\setminus \Omega_i)  + 
\sum_{i=1}^2 \HH^n(S_i\triangle B_i)
.%\nonumber
\end{align*}
By the theorem of Allen-Kriventsov-Neumayer, recalling the assumption \rf{eqfac749},
$$\HH^n(B_i\setminus \Omega_i)^2 \lesssim \alpha_i - \bar \alpha_i.$$
On the other hand,
$\HH^n(S_i\triangle B_i)\lesssim h_0$ by \rf{eqdhj4}. Thus, again the same estimate as in \rf{eqfafa7} holds, with $s=n$.
\qed
\vv

% ******************************************************************************************************

\section{Existence of tangents implies finiteness of the Carleson square function}

In this section we prove Theorem \ref{teoguai}. Given two disjoint Wiener regular domains $\Omega_1,\Omega_2\subset\R^{n+1}$, denote by $g_i$ the Green function of the domain $\Omega_i$, for $i=1,2$.
Let $p_1\in\Omega_1$, $p_2\in\Omega_2$, and consider the functions 
$$u_1(y) = g_1(y,p_1),\qquad u_2(y) = g_2(y,p_2).$$
We extend $u_i$ by $0$ in $\Omega_i^c$, and abusing notation we still denote by $u_i$ such extension. The Wiener regularity of $\Omega_i$ ensures that $u_i$ is continuous away from $p_i$.

Let $d=\frac16\,\min_i \dist(p_i,\partial\Omega_1\cup\partial\Omega_2)$.
For all $x\in\R^{n+1}\setminus (\Omega_1\cup\Omega_2)$ and all $r\in (0,d)$, by the ACF monotonicity formula, we have
$$\frac{\partial_rJ(x,r)}{J(x,r)}\geq \frac2r\bigl(\alpha_1(x,r) + \alpha_2(x,r)  - 2 \bigr),$$
with $J(x,r)$ and $\alpha_i(x,r)=\alpha_i$ as in \rf{eqACF2} and \rf{eqprec1}. Integrating on $r$, for any $\rho\in (0,d)$ we derive
$$\int_\rho^d  \frac{\alpha_1(x,r) + \alpha_2(x,r) -2}r\,dr \leq
\log\frac{J(x,d)}{J(x,\rho)}.$$
Thus,
$$\int_0^d  \frac{\alpha_1(x,r) + \alpha_2(x,r) -2}r\,dr \leq
\log\frac{J(x,d)}{\inf_{0<\rho\leq d} J(x,\rho)}.$$
Hence, in order to show that \rf{eqint1} holds for $x\in\partial\Omega_1\cap\partial\Omega_2$, it suffices to show that
$$\frac{J(x,d)}{\inf_{0<\rho\leq d} J(x,\rho)}<\infty.$$

Notice first that, by \rf{eqaux*},
$$J(x,d) \lesssim \avint_{B(x,2d)} |\nabla u_1|^2\,dy\; \cdot\avint_{B(x,2d)} |\nabla u_2|^2\,dy.$$
By the Caccioppoli inequality and the subharmonicity of $u_i$, for each $i$,
$$\avint_{B(x,2d)} |\nabla u_i|^2\,dy \lesssim \frac1{r^2} \avint_{B(x,3d)} |u_i|^2\,dy \lesssim \frac1{r^2} \bigg(\avint_{B(x,4d)} u_i\,dy\bigg)^2.$$
So the continuity of $u_i$ ensures that $J(x,d)<\infty$.

To estimate $J(x,\rho)$ from below,  let $\vphi_{x,\rho}$ be a $C^\infty$ bump function such that $\chi_{B(x,\rho/2)}\leq \vphi_{x,\rho}\leq \chi_{B(x,\rho)}$, with
$\|\nabla\vphi_{x,\rho}\|_\infty\lesssim \rho^{-1}$. Then, denoting by $\omega_i^{p_i}$ the harmonic measure for $\Omega_i$ with
pole at $p_i$, by the properties of the Green function, it holds
\begin{align*}
\omega_i^{p_i}(B(x,\rho/2)) & \leq \int \vphi_{x,\rho}\,d\omega^{p_i} = -\int \nabla \vphi_{x,\rho}\,\nabla u_i\,dy 
\leq \|\nabla \vphi_{x,\rho}\|_2\,\|\nabla u_i\|_{2,B(x,\rho)}\\
& \lesssim \rho^{(n-1)/2} \bigg(\int_{B(x,\rho)} \frac{\rho^{n-1}|\nabla u_i|^2}{|x-y|^{n-1}}\,dy\bigg)^{1/2} =  \rho^{n} \bigg(\frac1{\rho^2}\int_{B(x,\rho)} \frac{|\nabla u_i|^2}{|x-y|^{n-1}}\,dy\bigg)^{1/2}.
\end{align*}
Therefore,
$$J(x,\rho)^{1/2} \gtrsim \frac{\omega_1^{p_1}(B(x,\rho/2))}{\rho^n}\cdot \frac{\omega_2^{p_2}(B(x,\rho/2))}{\rho^n}.$$
Hence, to prove the proposition it suffices to show that
\begin{equation}\label{eqjrho}
\liminf_{\rho\to 0}\frac{\omega_i^{p_i}(B(x,\rho/2))}{\rho^n} >0\quad \text{ for $\HH^n$-a.e.\ tangent point $x\in\partial\Omega_1\cap\partial\Omega_2$,}
\end{equation}
for $i=1,2$. To this end, consider a subset $E$ of the tangent points for the pair $\Omega_1$, $\Omega_2$ such that $\HH^1(E)<\infty$. 
Since the set of tangent points is $n$-rectifiable, we have
\begin{equation}\label{eqmult1}
\lim_{\rho\to0} \frac{\HH^n(B(x,\rho)\cap E)}{(2\rho)^n} = 1\quad \text{ for $\HH^n$-a.e.\ $x\in E$.}
\end{equation}

By standard arguments, using that the tangent points for the pair $\Omega_1$, $\Omega_2$ are cone points for $\Omega_i$ (for $i=1,2$), it follows that $\HH^n|_E$ is absolutely continuous with respect to $\omega^{p_i}|_E$ for
$i=1,2$
 (see \cite[Theorem III]{AAM}, for example\footnote{Actually, in Theorem III from \cite{AAM} it is assumed that
$\partial\Omega$ is lower $n$-content regular in order to prove the mutual absolute continuity of $\HH^n|_E$ and $\omega^{p_i}|_E$. However, a quick inspection of the arguments shows that for the absolute continuity $\HH^n|_E\ll \omega_i^{p_i}|_E$ one only needs 
$\Omega_i$ to be Wiener regular.}). Then, by the Lebesgue-Radon-Nykodim differentiation theorem, we have
$$\lim_{\rho\to0} \frac{\HH^n(B(x,\rho)\cap E)}
{\omega_i^{p_i}(B(x,\rho)\cap E)}<\infty \quad\text{ for $\omega^{p_i}$-a.e.\ $x\in E$.}$$
Since null sets for $\omega_i^{p_i}$ are also null sets for $\HH^n|_E$, we infer that
\begin{equation}\label{eqmult2}
\lim_{\rho\to0} \frac{\omega_i^{p_i}(B(x,\rho)\cap E)}{\HH^n(B(x,\rho)\cap E)}>0 \quad\text{ for $\HH^n$-a.e.\ $x\in E$.}
\end{equation}
Multiplying \rf{eqmult1} and \rf{eqmult2}, we deduce \rf{eqjrho} and we complete the proof of the proposition.
\qed
\vv

% **************************************************************************************************************

\section{A counterexample to a question of Allen, Kriventsov and Neumayer}\label{sec-counter}

In this section we show that there are two Wiener regular domains $\Omega_1,\Omega_2\subset\R^2$ with rectifiable boundary
such that $E:=\pom_1\cap \pom_2$ satisfies:
\begin{itemize}
\item[(a)] $\HH^1(E)>0$,
\item[(b)] $\omega_1^{p_1}(E)=\omega^{p_2}(E)=0$, where $\omega_i^{p_i}$ stands for the harmonic measure
for $\Omega_i$ with pole at some point $p_i\in\Omega_i$, and
\item[(c)] the Alt-Caffarelli-Friedman functional associated with the respective Green functions $g_1(\cdot,p_1)$, $g_2(\cdot,p_2)$ of $\Omega_1$, $\Omega_2$
satisfies $J(x,0+)=0$ for $\HH^1$-a.e.\ $x\in E$.
\end{itemize}
 
The domain $\Omega_1$ is defined as follows:
$$\Omega_1= (0,1)^2\setminus \bigcup_{i=1}^\infty F_i,$$
where each $F_i$ is a union of $N_i$ equispaced closed intervals $\{I_j^i\}_{1\leq j\leq N_i}$ which are contained in the segment $L_i:=[0,1]\times \{2^{-i}\}$, so that the leftmost point of the leftmost interval $I_1^i$ is $(0,2^{-i})$ and the rightmost point of the rightmost interval 
$I_{N_i}^i$ is $(1,2^{-i})$.  The  intervals $I_j^i$ have length $\frac1{N_i}2^{-i}$, so that
$\HH^1(F_i)=2^{-i}$.
The numbers $N_i$ tend to $\infty$ as $i\to\infty$ and will be chosen below.

We let $\Omega_2$ be the domain  symmetric to $\Omega_1$ with respect to $x$ axis, so that $E=[0,1]\times\{0\}.$ 
We choose $p_1=(1/2,3/4)$ and $p_2=(1/2,-3/4)$.
See Figure 1.

%\RandomWalk{number = 100, length={4pt, 10pt}}
\usetikzlibrary{patterns}
\begin{figure}
\centering
\begin{tikzpicture}
	 \shadedraw[shading=axis,shading angle=180, opacity=0.2] (0pt,128pt) rectangle (256pt, 0pt);
	 \draw [ultra thick, dash=on 32pt off 42.6pt phase  0pt] (0pt,64pt) -- (256pt ,64pt);
	 \draw [very thick, dash=on 4pt off 12.8pt phase  0pt] (0pt,32pt) -- (256pt ,32pt);
	 \draw [very thick, dash=on 0.5pt off 3.5pt phase  0pt] (0pt,16pt) -- (256pt ,16pt);
	 \draw [ultra thick] (0pt, 0pt) -- (256pt, 0pt);
	  \draw [thick, dash=on 0.5pt off 3.5pt phase  0pt] (0pt,-16pt) -- (256pt ,-16pt);
	 \draw [very thick, dash=on 4pt off 12.8pt phase  0pt] (0pt,-32pt) -- (256pt ,-32pt);
	 \draw [ultra thick, dash=on 32pt off 42.6pt phase  0pt] (0pt,-64pt) -- (256pt ,-64pt);
	 \draw (0pt, 128pt) -- (256pt, 128pt) -- (256pt, -128pt) -- (0pt, -128pt) -- (0pt, 128pt);
	 \filldraw[black] (128pt,85.3pt) circle (2pt) node[anchor=west]{$p_1$};
	 \filldraw[black] (128pt,-85.3pt) circle (2pt) node[anchor=west]{$p_2$};
	   \node[left] at (0,0) {$E$};
	   \node[left] at (40pt, 100pt) {$\Omega_1$};
	   \node[left] at (40pt, -100pt) {$\Omega_2$};
	   \node[left] at (20pt, 75pt) {$I_1^1$};
	   %\node[left] at (20pt, -75pt) {$I_1^2$};
	   \node[left] at (259pt, 45pt) {$I_{16}^2$};
	   %\node[left] at (240pt, -45pt) {$I_{16}^2$};
	   %\draw[-] (240pt, 45pt) -- (253pt, 32pt);
	   %\draw[-] (240pt, -45pt) -- (253pt, -32pt);
	   %\draw  (128pt,85.3pt)
	   %\foreach \i in {1,...,50} {
	   %	-- ++(rnd*360:rnd)
	   %};
\end{tikzpicture}
\caption{The figure depicts $(0,1)^2 \setminus F_1 \cup F_2 \cup F_3$, and its reflection through the $x$-axis. The intuition behind the example is that, while total length of the segments $\{I_j^i\}$ goes to zero, overall $\cup_i F_i$ is sufficently scattered around so to obstruct the random walk leaving at $p_i$ from reaching $E$ (the segment $[0,1]$).}
\end{figure}
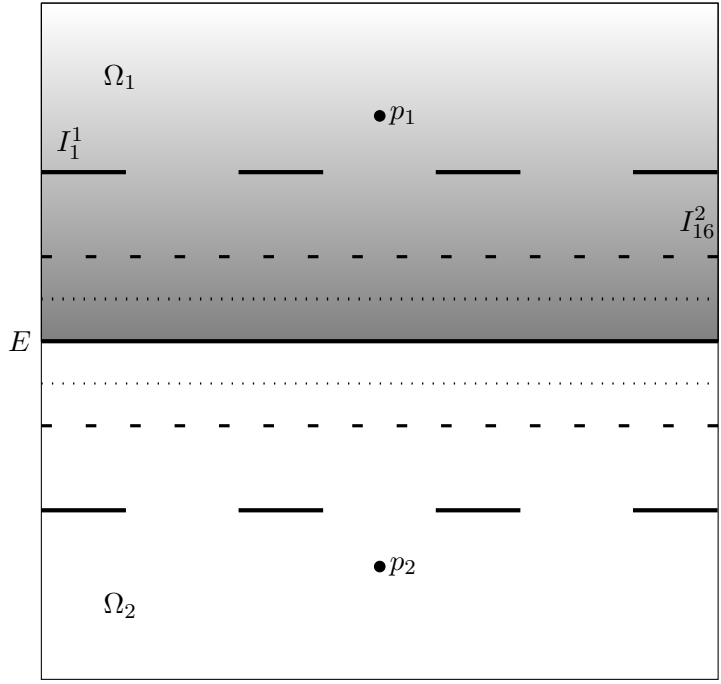

By Wiener's criterion, it is immediate to check that $\Omega_1$ and $\Omega_2$ are Wiener regular. Further, $\HH^1(\partial\Omega_1)
=\HH^1(\partial\Omega_2)=5$ and $\HH^1(E)=1$, and both $\pom_1$ and $\pom_2$ are rectifiable. 

Next we will check that $\omega^{p_1}_1(E)=0$ (by symmetry, we will have $\omega_2^{p_2}(E)=0$ too).
To this end, for each $i\geq1$ we consider the auxiliary domain $U_i = \Omega_1\cap ((0,1)\times (2^{-i},1))$. That is, $U_i$ is the region of $\Omega_1$ which is above the segment $L_i$. We denote 
%$$u(x) = \omega_1^x(\pom_1\setminus \partial U_m).$$ 
$$u(x) = \omega_1^x(E).$$
Obviously $u$ is harmonic in $\Omega_1$ and in each domain $U_i$. Notice also that it vanishes  in $\pom_1 \setminus E$
(in a continuous way in its relative interior with respect to $\pom_1$).

We will prove the following. 

\begin{claim}\label{claim33}
There exists some $\alpha\in (0,1)$ such that for all $i\geq 1$ we have
$$\sup_{x\in L_i} u(x) \leq \alpha\,\sup_{x\in L_{i+1}} u(x),$$
choosing $N_i=2^{2i}$.
\end{claim}

From this claim and the maximum principle (applied to $U_1$) we get, for any $m>1$,:
$$\omega_1^{p_1}(E)= u(p_1) \leq \sup_{x\in L_1}u(x) \leq \alpha\,\sup_{x\in L_2}u(x)\leq\cdots \leq \alpha^{m-1}\!\sup_{x\in L_{m}}u(x)\leq
\alpha^{m-1}.$$
%In particular, since $E\subset \pom_1\setminus \partial U_m$, we also have $\omega_1^{p_1}(E)\leq \alpha^{m-1}$.
So letting $m\to\infty$, we derive $\omega_1^{p_1}(E)=0$

\begin{proof}[Proof of Claim \ref{claim33}]
For any given $i\geq 1$ and all $x\in L_i$, we have
$$u(x) = \int_{\partial U_{i+1}} u(\xi)\,d\omega_{U_{i+1}}^x(\xi) = \int_{L_{i+1}} u(\xi)\,d\omega_{U_{i+1}}^x(\xi),$$
taking into account that $u$ vanishes in $\partial U_{i+1}\setminus L_{i+1}$. Thus,
\begin{equation}\label{equu889}
u(x) \leq \sup_{\xi\in L_{i+1}} \!u(\xi)\,\omega_{U_{i+1}}^x(L_{i+1})\quad \text{ for all $x\in L_i$.}
\end{equation}
We write
$$\omega_{U_{i+1}}^x(L_{i+1}) = 1- \omega_{U_{i+1}}^x(\partial U_{i+1}\setminus L_{i+1}).$$
Since $\partial U_i\cap B(x,2^{-i-1})\subset (L_{i+1})^c$ (because $x\in L_i$), by Lemmas \ref{lem2.1} and \ref{lemcap-contingut},
we have
\begin{align}\label{eqaljdy42}
\omega_{U_{i+1}}^x(\partial U_{i+1}\setminus L_{i+1}) & \geq \omega_{U_{i+1}}^x(\partial U_{i+1}\cap B(x,2^{-i-1}))\\
& \gtrsim
 \frac1{\log\dfrac{2^{-i-1}}{\capp_L(\partial U_{i+1}\cap B(x,2^{-i-3}))}} \gtrsim 
  \frac1{\log\dfrac{2^{-i-1}}{(\HH_\infty^{1/2}(\partial U_{i+1}\cap B(x,2^{-i-3})))^{2}}}. \nonumber
\end{align}
Notice that  $\HH_\infty^{1/2}(\partial U_{i+1}\cap B(x,2^{-i-3}))\geq \HH_\infty^{1/2}(F_i\cap B(x,2^{-i-3}))$.  Moreover,     
we will show below that
\begin{equation}\label{eqcontent7}
\HH_\infty^{1/2}(F_i\cap B(x,2^{-i-3}))\gtrsim 2^{-i/2}.
\end{equation} 
Assuming this  for the moment, plugging this estimate into \rf{eqaljdy42}, we deduce that
$$\omega_{U_{i+1}}^x(\partial U_{i+1}\setminus L_{i+1})\gtrsim 1,$$
and thus 
$\omega_{U_{i+1}}^x(L_{i+1})\leq \alpha,$
for some fixed $\alpha\in (0,1)$. Together with \rf{equu889}, this proves the claim.
 
Now it remains to prove \rf{eqcontent7}. 
 To simplify notation, write $r=2^{-i-3}$.
 By Frostman's lemma, it suffices
to show that there exists some measure $\mu$ supported on $F_i\cap B(x,r)$ with total mass $\|\mu\|= r^{1/2}$ such that $\mu(B(y,t))\lesssim
t^{1/2}$ for all $y\in F_i\cap B(x,r)$, $t>0$. To define $\mu$, we let $\{I_j^i\}_{j\in J_{x,r}}$ be the family of intervals $I_j^i$ which are contained in $B(x,r)$
and then we set
$$\mu= \frac{r^{1/2}}{\sum_{j\in J_{x,r}} \HH^1(I_j^i)}\sum_{j\in J_{x,r}} \HH^1|_{I_j^i}.$$
Clearly $\|\mu\|=r^{1/2}$ and so it remains to check that $\mu(B(y,t))\lesssim
t^{1/2}$ for all $y\in F_i\cap B(x,r)$, $t>0$. Since the number of segments $I_j^i$ contained in $B(x,r)$ is comparable to $rN_i$, we have
$$\sum_{j\in J_{x,r}} \HH^1(I_j^i) \approx r\,\HH^1(F_i)  = r\,2^{-i}$$
and 
recall that
$$\HH^1(I_j^i) = \frac1{N_i}\,\HH^1(F_i)=\frac1{N_i}\,2^{-i}.$$
So we have
$$\mu(I_j^i) \approx \frac{ r^{1/2}}{r\,2^{-i}}\,\frac1{N_i}\,2^{-i} = \frac1{N_i\,r^{1/2}}.$$
Observe first that, by the choice $N_i = 2^{2i}$,
$$\mu(I_j^i) \approx \frac1{N_i\,r^{1/2}}\approx \frac1{N_i\,2^{-i/2}} = \left(\frac1{N_i}\,2^{-i}\right)^{1/2} = \HH^1(I_j^i)^{1/2}.$$
From this estimate,   we deduce that, for $y\in I_j^i$ and $0<t\leq \HH^1(I_j^i)$,
$$\mu(B(y,t)) \leq \frac{\mu(I_j^i)}{\HH^1(I_j^i)}\,t\lesssim \frac{\HH^1(I_j^i)^{1/2}}{\HH^1(I_j^i)}\,t \leq t^{1/2}.$$

From the fact that $\HH^1(F_i)\leq 1/4$, it easily follows that $\dist(I_j^i,I_{j+1}^i)\geq 1/(2N_i)$ (we assume that $I_j^i$ and $I_{j+1}^i$ are consecutive intervals in $F_i$). So for $y\in I_j^i$ and $\HH^1(I_j^i)\leq t\leq 1/(2N_i)$, it follows that $B(y,t)\cap F_i = B(y,t)\cap I_j^i$ (i.e., $B(y,t)$ does not intersect any interval $I_k^i$ different from $I_j^i$). Then we have
$$\mu(B(y,t)) \leq \mu(I_j^i) \lesssim \HH^1(I_j^i)^{1/2}\leq t^{1/2}.$$
Finally, for $1/(2N_i)<t\leq 2r$, 
$$\mu(B(y,t)) \lesssim \frac tr \,\|\mu\|= \frac{t\,r^{1/2}}r \lesssim t^{1/2}.$$
So $\mu$ satisfies the desired growth condition and then \rf{eqcontent7} holds.
\end{proof}

\vv

Next we will show that $J(x,0+)=0$ for $\HH^1$-a.e.\ $x\in E$. 
Recall that, for such $x$ and for $0<r\leq 1/4$,
$J(x,r) = J_1(x,r)^{1/2}J_2(x,r)^{1/2},$
where 
$$J_i(x,r) = \frac{1}{r^{2}}\int_{B(x,r)} |\nabla g_i(y,p_i)|^{2}\,dy.$$
We assume that the Green functions $g_i(\cdot,p_i)$ are extended by $0$ away from $\Omega_i$. By Caccioppoli's inequality, since 
$g_i(\cdot,p_i)$ is subharmonic and non-negative in $B(x,8r)$,
\begin{equation}\label{eqJi**}
J_i(x,r)\approx \avint_{B(x,r)} |\nabla g_i(y,p_i)|^2\,dy \lesssim \frac1{r^2}\avint_{B(x,2r)} |g_i(y,p_i)|^2\,dy
\leq \frac1{r^2}\sup_{y\in B(x,2r)} g_i(y,p_i)^2.
\end{equation}
To bound $g_i(y,p_i)$ in terms of the harmonic measure $\omega_i^{p_i}$, we appeal to the following lemma:

\begin{lemma}\label{lemWGpla}
Let $\Omega\subset\R^{2}$ be a bounded, Wiener regular open set.
Let $\bar B$ be a closed ball centered at $\partial\Omega$. Let $\omega$ be the harmonic measure for $\Omega$ and $g$ its Green function.
Then,
\begin{equation}\label{eqgreen73}
g(z,y) \lesssim \omega^z(8\bar B)\,\bigg(\log \frac{\capp_L(\bar B)}{\capp_L(\frac14 \bar B\setminus\Omega)}\bigg)^2
\qquad\mbox{
 for all $z\in \Omega\setminus  2\bar B$ and $y\in \frac15\bar B\cap\Omega$.}
 \end{equation}
\end{lemma}

For the proof of the lemma, see \cite[Chapter 7]{Prats-Tolsa-Notes}.
In the case when $\Omega$ is a domain satisfying the CDC condition, the right hand side of \rf{eqgreen73} is comparable to $1$. In this situation, the result above appears in \cite{AH}.
However, in the generality stated above, it is difficult to find the lemma in the literature. 

The domains $\Omega_1$, $\Omega_2$ above do not satisfy the CDC condition (roughly speaking, because the holes between the segments $I_j^i$ are two big). However, any closed ball $\bar B$  with radius at most $1$ centered in $E$ satisfies
$$\capp_L(\tfrac14 \bar B\setminus\Omega_i) \geq \capp_L(\tfrac14 \bar B\cap E) \approx \rad(\bar B)\approx \capp_L(\bar B).$$
Thus, applying Lemma \ref{eqgreen73} to $\Omega_i$ and the ball $\bar B= B(x,10r)$, we get
$$g_i(p_i,y) \lesssim \omega_i^{p_i}(\bar B(x,80r))\quad \mbox{ for all $y\in B(x,2r)$,}
$$
assuming $r$ small enough so that $p_i\not\in B(x,20r)$.
Plugging this estimate into \rf{eqJi**}, we obtain
\begin{equation}\label{eqji**2}
J_i(x,r)\lesssim \left(\frac{\omega_i^{p_i}(\bar B(x,80r))}{r}\right)^2
\end{equation}
for all $x\in E$ and $r$ small enough.
Consequently, by Fatou's lemma,
\begin{align*}
\int_E J(x,0+)^{1/4}\,d\HH^1(x) & \leq \liminf_{r\to0} \int_E J(x,r)^{1/4}\,d\HH^1(x) \\
& \leq 
\liminf_{r\to0} \bigg[\left(\int_E J_1(x,r)^{1/2}\,d\HH^1(x)\right)^{1/2}
\left(\int_E J_2(x,r)^{1/2}\,d\HH^1(x)\right)^{1/2}\bigg].
\end{align*}
Observe now that, for $i=1,2$, by \rf{eqji**2} and Fubini,
\begin{align*}
\int_E J_i(x,r)^{1/2}\,d\HH^1(x) &
\lesssim \int_E \frac{\omega_i^{p_i}(\bar B(x,80r))}{r} \,d\HH^1(x)\\
& = \int_{y\in \pom:\dist(y,E)\leq 80r} \frac1r\int_{x\in E\cap \bar B(y,80r)}\,d\HH^1(x)\,d\omega_i^{p_i}(y)\\
& \lesssim  \omega_i^{p_i}(\bar U_{80r}(E)),
\end{align*}
where $\bar U_t(F)$ stands for the closed $t$-neighborhood of $F$.
So we deduce
$$\int_E J(x,0+)^{1/4}\,d\HH^1(x)\lesssim \lim_{r\to 0} \omega_1^{p_1}(\bar U_{80r}(E))\cdot \lim_{r\to 0} \omega_2^{p_2}(\bar U_{80r}(E)) = \omega_1^{p_1}(E)\,\omega_2^{p_2}(E) =0,$$
using the outer regularity of the measures $\omega_i^{p_i}$. Thus, $J(x,0+)=0$ for $\HH^1$-a.e.\ $x\in E$, as wished.

\vv

\begin{rem}
We already mentioned that the domains $\Omega_1$, $\Omega_2$ defined above do not satisfy the CDC. However, they can be modified suitably so that 
the CDC holds. For example, we could define $\wt \Omega_i=\Omega_i\setminus K$, where $K$ is the union of a countable family of homothetic copies
of the $1/3$ Cantor set located between all consecutive intervals $I_j^i$, $I_{j+1}^i$. That is, in the ``hole" in $L_i$ between the intervals $I_j^i$, $I_{j+1}^i$, we put a homothetic copy of the $1/3$ Cantor set with diameter comparable to the size of the hole.
We could also replace the $1/3$ Cantor set by
a suitable approximation by intervals, in case we wanted the boundary of $\wt\Omega_i$ to be made up of a countable family of intervals.
We leave the details for the reader.
\end{rem}

% **************************************************************************************************************

\vvv

\end{document}